\documentclass[11pt,twoside,reqno]{amsart} \usepackage{a4wide} \usepackage[T1]{fontenc} \usepackage[utf8]{inputenc} \usepackage{soul}  \usepackage{amsmath} \usepackage{amsfonts} \usepackage{amssymb} \usepackage{mathtools} \usepackage{amsthm} \usepackage{xcolor} \usepackage{graphicx} \usepackage{hyperref} \usepackage{standalone} \usepackage{esint} \usepackage{comment} \usepackage[left=2.3cm,top=3cm,right=2.3cm]{geometry} \usepackage{bm} \usepackage{svg} \usepackage{float} \numberwithin{equation}{section} \theoremstyle{plain} \newtheorem{thm}{Theorem}[section] \newtheorem{theorem}[thm]{Theorem} \newtheorem*{theorem*}{Theorem} \newtheorem{lemma}[thm]{Lemma}  \newtheorem{proposition}[thm]{Proposition}     \theoremstyle{definition}      \newtheorem*{ack}{Acknowledgements} \theoremstyle{remark} \newtheorem{remark}[thm]{Remark}  \usepackage{xcolor}   \author{Kai Hippi} \address{Aalto University, Espoo, Finland} \email{kai.hippi@aalto.fi}
\usepackage{subcaption} \usepackage{xfrac} \usepackage{cite}

\title[Quantum Mixing and Benjamini--Schramm Convergence of Hyperbolic Surfaces]{Quantum Mixing and Benjamini--Schramm Convergence of Hyperbolic Surfaces}

\date{\today}

\begin{document}

\begin{abstract}
	We study compact hyperbolic surfaces and multiplication observables, establishing a large-scale analogue of Zelditch’s quantum mixing theorem with hypotheses that hold for both arithmetic and Weil--Petersson random surfaces of large genus. This complements the large-scale quantum ergodicity theorems of Le Masson and Sahlsten, which themselves are large-scale analogues of the quantum ergodicity theorem of Shnirelman, Zelditch, and Colin de Verdi\`{e}re, thereby providing a more complete picture of the asymptotic behavior of observables in the large-scale limit. Our approach does not rely on the ball averaging operator or Nevo’s ergodic theorem. Instead, we introduce a new method based on the hyperbolic wave equation and the quantitative exponential mixing of the geodesic flow established by Ratner and Matheus.
\end{abstract}

\thanks{The research conducted here is supported by the Vilho, Yrjö and Kalle Väisälä Foundation and the Research Council of Finland project: Quantum chaos on large and many body systems, grant numbers 347365, 353738.}


\maketitle

\section{Introduction}

\subsection{Background}

\label{Ssec: Background}

In quantum systems, the state of the system can be described in terms of the eigenfunctions of its energy operator. Their evolution is governed by the Schrödinger equation, the quantum analogue of Newton’s equations of motion. Observables correspond to measurable quantities of the system. Understanding the interplay between the energy operator, its eigendata (eigenvalues and eigenfunctions), and the observables leads to a deeper understanding of the dynamics of the system. A central topic in quantum chaos is the study of localization and delocalization phenomena of the eigenfunctions of the energy operator. Here, the Hamiltonian system is assumed to be the geodesic flow $ \varphi_{t} $ of a compact connected Riemannian manifold $ ( M,g ) $ and the quantum Hamiltonian is taken to be $ \sqrt{ -\Delta } $, where $ -\Delta $ is the Laplace--Beltrami operator of the manifold. For more background and a better overview, we refer to the following surveys: \cite{Zel05, Sar11, Zel19, Ana22, Mar25}.

Consider a compact Riemannian manifold $(M,g)$ and the quantum Hamiltonian $\sqrt{-\Delta}$ with eigenvalues $ \{ \mu_{j} = \sqrt{ \lambda_{j} } \} $ and eigenfunctions $ \{ \psi_{j} \} $, where $ \{ \lambda_{j} \} $ and $ \{ \psi_{j} \} $ are the eigenvalues and eigenfunctions of the Laplace--Beltrami operator $ -\Delta $. Let $A$ be a zeroth-order pseudo-differential operator representing an observable. We are interested in the diagonal average $\langle \psi_j, A \psi_j \rangle$,
which corresponds to the expectation value of $A$ in the state $\psi_j$, and in the off-diagonal average $\langle \psi_j, A \psi_k \rangle$, which represents the transition amplitude from $\psi_j$ to $\psi_k$.

Quantum ergodicity theorem (Shnirelman \cite{Shn74}, Colin de Verdi\`{e}re \cite{CdV85}, and Zelditch \cite{Zel87}) is a cornerstone result of quantum chaos and serves as a quantum analogue of the classical ergodic theorem. For an ergodic compact Riemannian manifold, it states that almost all subsequences of eigenfunctions corresponding to increasing eigenvalues satisfy that
\[
	\langle \psi_{j_l}, A \psi_{j_l} \rangle \xrightarrow{l \to \infty} \omega(A),
\]
where $\omega(A)$ denotes the average of the principal symbol $\sigma_A$ of $A$ (the classical counterpart of observable $A$) over the unit co-tangent bundle of $M$. Roughly speaking, the theorem asserts that the probability distributions induced by the eigenfunctions converge to the uniform distribution in the weak sense of testing the distributions against observables.

The quantum ergodicity theorem focuses on the diagonal averages of the observables. Zelditch \cite{Zel90, Zel96a} extended it to the off-diagonal averages, given that the geodesic flow on the compact Riemannian surface is now weak mixing, proving that for any $\epsilon > 0$, there exists $\delta(\epsilon) > 0$ such that for all $\tau \in \mathbb{R}$
$$
\limsup_{\mu \to \infty} \
\frac{1}{\# \{ j  :  \mu_{j} < \mu \} } \sum_{j: \mu_{j} < \mu }^{} \sum_{\substack{ k: \mu_k < \mu , \\ k \neq j, \\   | \mu_{j} - \mu_k - \tau| < \delta(\epsilon)}}
\left| \left\langle \psi_j, A \psi_k \right\rangle \right|^2 < \epsilon.
$$
This means that almost all subsequences of eigenfunctions with increasing eigenvalues, $ ( \psi_{j_{l}} ) $, satisfy that the transition amplitudes into a small energy window of separation $ \tau $ from the eigenvalue of $ \psi_{j_{l}} $ can be made arbitrarily small as $ l \to \infty $ by choosing a small enough energy window.

A unified formulation encompassing both of these results can be given following the works of multiple authors: \cite{Shn74, CdV85, Zel87, Zel90, GL93, LS93, Zel96b, Zel96a, ZZ96, Sun97}. The version below adapts Theorems 4.1 and 4.4 from \cite{Zel05}.

\begin{theorem}[Semi-classical quantum mixing, {\cite{Sun97, Zel05}}]
\label{Thm: Quantum ergodicity and quantum weak mixing}
Let $(M,g)$ be a compact Riemannian manifold (possibly with a boundary) with quantum Hamiltonian $ \sqrt{-\Delta} $ together with eigenvalues $ \{ \mu_{j} =~\sqrt{\lambda_{j}} \} $ and eigenfunctions $ \{ \psi_{j} \} $, where $ \{ \lambda_{j} \} $ and $ \{ \psi_{j} \} $ are the eigenvalues and eigenfunctions of the Laplace--Beltrami operator $ -\Delta $. Let $A$ be a zeroth-order pseudo-differential operator. Consider the following properties:
\begin{itemize}
    \item[(I)] 
    $$\displaystyle \lim_{ \mu \to \infty} \ \frac{1}{ \# \{ j :  \mu_{j} < \mu \} } \sum_{j: \mu_j < \mu } 
    \big| \langle \psi_j, A \psi_j \rangle - \omega(A) \big|^2  = 0,$$
    \item[(II)] 
    For every $\epsilon > 0$, there exists $\delta(\epsilon) > 0$ such that
    \[
    \limsup_{  \mu \to \infty} \ \frac{1}{ \# \{ j  :  \mu_{j} < \mu \} } 
    \sum_{j: \mu_j < \mu} 
    \sum_{\substack{k: \mu_k < \mu, \\ k \neq j , \\ |\mu_j - \mu_k| < \delta(\epsilon)}} 
    \left|\langle \psi_j, A \psi_k \rangle \right|^2  < \epsilon,
    \]
    \item[(III)] 
    For every $\epsilon > 0$, there exists $\delta(\epsilon) > 0$ such that for all $\tau \in \mathbb{R}$,
    \[
    \limsup_{\mu \to \infty} \ \frac{1}{ \# \{ j  :  \mu_{j} < \mu \} } 
    \sum_{j: \mu_{j} < \mu} 
    \sum_{\substack{ k: \mu_{k} < \mu , \\ k \neq j ,\\ | \mu_j - \mu_k - \tau| < \delta(\epsilon)}} 
    |\langle \psi_j, A \psi_k \rangle|^2  < \epsilon.
    \]
\end{itemize}
Then the geodesic flow $\varphi_t$ is ergodic on $(S^*M, d\mu)$ if and only if Properties {(I)} and {(II)} hold. It is weakly mixing on $(S^*M, d\mu)$ if and only if Properties {(I)}--{(III)} hold. Here $S^*M$ denotes the unit co-tangent bundle of $M$ and $d\mu$ is the Liouville measure on $S^*M$.
\end{theorem}

This theorem exhibits a delicate connection between the classical and quantum settings. If Properties \emph{(I)} and \emph{(II)} hold in the quantum setting, the corresponding classical geodesic flow is ergodic, and conversely. Likewise, if Properties \emph{(I)}--\emph{(III)} hold, the classical flow is weak mixing, and vice versa. We therefore say that a quantum system satisfying Properties \emph{(I)} and \emph{(II)} is quantum ergodic, while one satisfying Properties \emph{(I)}--\emph{(III)} is quantum mixing.

Our focus will be on compact connected hyperbolic surfaces, which are inherently chaotic due to their (constant) negative curvature. In this context, the following theorem of Ratner \cite{Rat87}, with quantitative bounds due to Matheus \cite{Mat13}, will be central. The theorem is stated here in a slightly modified form using only the physical space observables.

\begin{theorem}[Position space exponential mixing]
\label{Thm: Exponential mixing}
Let $X$ be a compact connected hyperbolic surface with geodesic flow $ \varphi_{t} $ and with $\lambda_1( X )$ being the second smallest eigenvalue of the Laplace--Beltrami operator $ -\Delta $. Let $f,g \in L^2(X)$. Then for all $t \ge 0$
\[
\bigg| \big\langle f \circ \varphi_t, g \big\rangle_{L^2(T^1X)} 
- \int_{T^1X} f\,dx \int_{T^1X} g\,dx \bigg|
\leq 11 \, e^{ \beta( \lambda_{1}( X ) ) } \, ( 1 + t ) \, e^{- \beta(\lambda_1( X ) ) t} \, 
\|f\|_2 \, \|g\|_2,
\]
where
\begin{align*}
	&
	\beta( x ) = 
	\begin{cases}	
		1 - \sqrt{1 - 4x}, \qquad & \text{ if } x \leq \frac{1}{4},
		\\
		1, \qquad & \text{ if } x > \frac{1}{4},
	\end{cases}
\end{align*}
and where $f(x) = f(x,\theta)$ and $g(x) = g(x,\theta)$ are viewed as functions on $T^1X$, the unit tangent bundle.

\end{theorem}

The more general phase space version (see \cite{Rat87, Mat13}) implies that the geodesic flow on a compact connected hyperbolic surface is weak mixing, which further implies ergodicity. By Theorem \ref{Thm: Quantum ergodicity and quantum weak mixing}, a compact connected hyperbolic surface satisfies Properties \emph{(I)}--\emph{(III)}, and hence is both quantum ergodic and quantum mixing.

From Theorem~\ref{Thm: Quantum ergodicity and quantum weak mixing}, it follows that neither the $2$-torus nor the $2$-sphere is quantum ergodic or quantum mixing, since their geodesic flows are neither ergodic nor weakly mixing. However, if one restricts the class of observables to physical-space observables, that is, multiplication operators corresponding to functions, then all flat $2$-tori satisfy Property~\emph{(I)}; see \cite{MR12} for details. This illustrates that the choice of observables plays a crucial role. On the other hand, it is an interesting question whether Properties \emph{(II)} and \emph{(III)} hold for flat $2$-tori. For further discussion, see Appendix \ref{App: Appendix}.

These phenomena on Riemannian surfaces naturally motivate the search for analogous results in other settings, for example, the discrete setting. Anantharaman and Le Masson \cite{AL15} established that, for certain finite graphs, the analogue of Property \emph{(I)} holds. Their result, however, is not a strict analogue, since finite graphs have only finitely many eigenvalues. Instead, they considered a sequence $(G_n)$ of graphs with increasing size $n$ and a fixed energy window $I$; as the graph grows, the number of eigenvalues in the energy window increases. This differs from the quantum ergodicity theorem, in which one considers a fixed structure at increasing energy levels, so that all eigenvalues below a given energy and the corresponding eigenfunctions are considered. 

A special case of the theorem of Anantharaman and Le Masson \cite{AL15} states that if $(a_n)$ is a uniformly bounded sequence of diagonal matrices acting on functions on the vertices of the graph, and if $(\Delta_n)$ is a sequence of graph Laplacians, each with eigenvalues $\{ \lambda_j^{(n)} \}$ and eigenfunctions $\{ \psi_j^{(n)} \}$, then under certain conditions (analogous to Properties \textbf{(BSC)} and \textbf{(EXP)} introduced in Subsection \ref{Ssec: Main theorems}) on the structure of the graphs, we have
\[
\lim_{n \to \infty} \frac{1}{\# \{ j : \lambda_j^{(n)} \in I \} } 
\sum_{j:\lambda_j^{(n)} \in I} 
\left| \left\langle \psi_j^{(n)}, a_n \psi_j^{(n)} \right\rangle - \frac{\operatorname{Tr}(a_n)}{n} \right|^2 = 0.
\]
As in the quantum ergodicity theorem, the averages of the expectation values converge to the average of the observable with respect to the uniform distribution.

We refer to such results in which the structure grows while the energy remains bounded as large-scale results, in contrast to the original results in which the energy grows. We refer to the latter results as large-energy results. The study of the large-scale quantum ergodicity of graphs and related problems is an active field with a vigorous interplay with the continuous setting. For more related results in the discrete graph setting, see \cite{DP12, BL13, EKYY13, TVW13, AL15, Gei15, BLL16, Ana17, AS17, BHKY17, BKY17, AS19, BS19, BHY19, McK22, MS23}. An impressive result established in \cite{BHKY17, BKY17, BHY19} shows an average-free large-scale analogue of the large-scale quantum ergodicity theorem of Anantharaman and Le Masson \cite{AL15} for random regular graphs. This resembles the quantum unique ergodicity conjecture of Rudnick and Sarnak \cite{RS94} in the random setting. However, the observable in \cite{BHKY17, BKY17, BHY19} must be chosen independently of the graph. Nevertheless, the randomization effectively suppresses pathological cases, helping to establish stronger results when compared to the deterministic setting.

Since the large-scale quantum ergodicity theorem for graphs by Anantharaman and Le Masson \cite{AL15} differs from the large-energy quantum ergodicity theorem of Shnirelman \cite{Shn74}, Colin de Verdi\`{e}re \cite{CdV85}, and Zelditch \cite{Zel87}, it is natural to seek a corresponding surface result for the large-scale graph theorem that would highlight distinctions when compared to the original large-energy quantum ergodicity theorem. Such an analogue for hyperbolic surfaces was established by Le Masson and Sahlsten \cite{LS17} (see Subsection \ref{Ssec: Main theorems} for more details). Their result closely resembles that of Anantharaman and Le Masson in that they prove a similar analogue of Property \emph{(I)} of the original quantum ergodicity theorem; minor differences reflect the transition from discrete graphs to continuous surfaces.

It is tempting to try to equate the large-energy setting with the large-scale setting. Heuristically, one might say that in the large-energy regime we use higher and higher energies to resolve finer details, whereas in the large-scale setting we zoom into the structure to capture finer features. However, the large-scale setting appears to be much more sensitive to the underlying structure of the manifold or graph. The large-energy quantum ergodicity theorem says that all Riemannian manifolds for which the geodesic flow is ergodic are quantum ergodic and that those for which the geodesic flow is weak mixing are quantum mixing; this is open for the variable-curvature large-scale setting. Moreover, in the case of graphs, the semiclassical measure is not the uniform measure, so it is not expected to be uniform either on manifolds of variable negative curvature, in contrast to the large-energy regime, which does not see this subtlety. 

To our knowledge, the only previously existing large-scale off-diagonal results establishing analogues of Properties \emph{(I)}--\emph{(III)} of Theorem~\ref{Thm: Quantum ergodicity and quantum weak mixing} (see Subsection~\ref{Ssec: Main theorems} for the properties) are those of Cipolloni, Erdős, and Schröder \cite{CES21} for Wigner matrices. Their setting resembles that of Anantharaman and Le Masson \cite{AL15}, where graphs are replaced by random matrices. Their result is remarkably strong, as they require no averaging: they prove that
$$
\lim_{n \to \infty} \left| \left\langle \psi_j^{(n)}, A_n \psi_k^{(n)} \right\rangle  - \delta_{jk}\tfrac{ \operatorname{Tr}(A_{n}) }{n} \right| = 0;
$$
the decay to $ 0 $ has the optimal and uniform rate of $n^{-1/2}$ for all $j,k$, where $ n \times n $ matrices are considered with $ n $ growing. Thereby Cipolloni, Erdős, and Schröder \cite{CES21} establish a result that resembles the quantum unique ergodicity conjecture of Rudnick and Sarnak \cite{RS94} in the large-scale setting; moreover, they extend it to the off-diagonal case as well. We remark that, as in \cite{BHKY17, BKY17, BHY19}, the observable is chosen independently of the matrix. This result is relevant, as random matrix theory is closely related to the study of quantum ergodicity and related problems on graphs and surfaces, especially if Berry's conjecture holds (see \cite{Ber77}). For more related results for random matrices, see \cite{ESY09a, ESY09b, BY17}. 

Studying the off-diagonal setting is timely: a new result by Bordenave, Letrouit, and Sabri \cite{BLS26} establishes quantum ergodicity and quantum mixing results on growing Schreier graphs, where even quantum ergodicity was previously unknown. To our knowledge, this is the first mixing result in the graph setting. This is especially interesting, as the graphs are intuitively closer to the original setting on manifolds when compared to random matrices. We also remark that the techniques presented in this article are successfully applied in the work of Hippi, Lequen, Mikkelsen, Sahlsten, and Ueberschär \cite{HLMSU26}, where large-scale quantum ergodicity and quantum mixing are established for more general Schrödinger operators.


\subsection{Main theorems}
\label{Ssec: Main theorems}

We prove a large-scale analogue of the corollary of Theorem \ref{Thm: Quantum ergodicity and quantum weak mixing} for physical-space observables, that is, for functions viewed as multiplication operators, thereby complementing the result of Le Masson and Sahlsten \cite{LS17}. We also prove a probabilistic analogue of this using the Weil--Petersson random model, complementing the probabilistic quantum ergodicity result of Le Masson and Sahlsten \cite{LS24}. A small difference compared to the large-energy setting is that we use the quantum Hamiltonian $ \sqrt{-\Delta - \tfrac14} $ instead of $ \sqrt{-\Delta} $; however, this causes only a spectral shift. We proceed to introduce the main theorems now.

\begin{theorem}[Large-scale quantum ergodicity and quantum mixing]
\label{Thm: Large-scale analogue}
Let $ (X_{n}) $ be a sequence of compact connected hyperbolic surfaces with corresponding Laplace--Beltrami operators \( (-\Delta_{X_{n}}) \). The Laplacian $ -\Delta_{X_{n}} $ has eigenvalues
	\begin{align*}
		0 = \lambda_{0}^{( n )} < \lambda_{1}^{( n )} \leq \dots \leq \lambda_{j}^{( n )} \leq \dots \xrightarrow{j \to \infty} \infty
	\end{align*}
together with a corresponding orthonormal eigenbasis $ \{ \psi_{j}^{( n )} \} $ over $ L^{2}( X_n ) $. Define associated values of the eigenvalues as follows:
\begin{align*}
	\rho_{j}^{( n )} = \sqrt{ \lambda_{j}^{( n )} - \tfrac{1}{4} }.
\end{align*}
Assume that the sequence has the following properties:
\begin{itemize}
    \item[\textbf{(BSC)}] \textbf{Benjamini--Schramm convergence:}
    for every \( R > 0 \),
    $$
        \lim_{n \to \infty} 
        \frac{
            \big| \{ x \in X_{n} : \operatorname{InjRad}_{X_{n}}(x) < R \} \big|
        }{
            |X_{n}|
        } = 0,
    $$
    \item[\textbf{(EXP)}] \textbf{Expander property:}
    the sequence admits a uniform lower bound on the spectral gap,
    \item[\textbf{(UND)}] \textbf{Uniform discreteness:}
    the sequence admits a uniform lower bound on the injectivity radius.
\end{itemize}
Let $ I \subset ( 0,\infty ) $ be a compact interval and let $ ( a_{n} ) $ be a sequence of uniformly bounded functions on $ ( X_{n} ) $, each acting as a multiplication operator. Then the following properties hold:
\begin{itemize}
    \item[(i)] 
    $$
        \lim_{n \to \infty} \
        \frac{1}{ \# \{ j  :  \rho_{j}^{( n )} \in I \} }
        \sum_{j : \rho_{j}^{(n)} \in I}
        \left|
            \left\langle \psi_{j}^{(n)}, a_{n} \psi_{j}^{(n)} \right\rangle
            - \frac{1}{| X_{n} |} \int_{X_{n}} a_n \, dx
        \right|^{2} = 0,
    $$
    \item[(ii)] 
    For every \( \epsilon > 0 \) there exists \( \delta(\epsilon) > 0 \) such that
    \[
        \limsup_{n \to \infty} \
        \frac{1}{ \# \{ j  :  \rho_{j}^{( n )} \in I \} }
        \sum_{j : \rho_{j}^{(n)} \in I}
        \sum_{\substack{k: \rho_{k}^{( n )} \in I, \\ k \neq j, \\ | \rho_{j}^{(n)} - \rho_{k}^{(n)} | < \delta( \epsilon )}}
        \left| \left\langle \psi_{j}^{(n)}, a_{n} \psi_{k}^{(n)} \right\rangle \right|^{2}
        < \epsilon,
    \]
    \item[(iii)]
    For every \( \epsilon > 0 \) there exists \( \delta(\epsilon) > 0 \) such that for all \( \tau \in \mathbb{R} \),
    \[
        \limsup_{n \to \infty} \
        \frac{1}{ \# \{  j : \rho_{j}^{( n )} \in I \} }
        \sum_{j: \rho_{j}^{(n)} \in I}
        \sum_{\substack{k : \rho_{k}^{( n )} \in I, \\ k \neq j , \\ | \rho_{j}^{(n)} - \rho_{k}^{(n)} - \tau | < \delta( \epsilon )}}
        \left| \left\langle \psi_{j}^{(n)}, a_{n} \psi_{k}^{(n)} \right\rangle \right|^{2}
        < \epsilon.
    \]
\end{itemize}
\end{theorem}

Properties \emph{(i)}--\emph{(iii)} are the large-scale analogues of Properties \emph{(I)}--\emph{(III)} in Theorem~\ref{Thm: Quantum ergodicity and quantum weak mixing}. Accordingly, we refer to systems satisfying Properties \emph{(i)} and \emph{(ii)} as large-scale quantum ergodic, and those satisfying Properties \emph{(i)}--\emph{(iii)} as large-scale quantum mixing. We remark that Property \emph{(i)} was proven by Le Masson and Sahlsten in \cite{LS17}. Property \emph{(i)} is also the surface analogue of the large-scale quantum ergodicity theorem of Anantharaman and Le Masson \cite{AL15} discussed in the previous subsection.  

In Appendix \ref{App: Appendix}, we construct an example of a sequence of growing flat $2$-tori for which Properties \emph{(ii)} and \emph{(iii)} fail. This demonstrates that large-scale quantum mixing is not a generic property for sequences of Riemannian surfaces. It remains an interesting question whether Property \emph{(i)} holds for sequences of growing tori; establishing this would complement the large-energy result, analogous to Property \emph{(I)} of Theorem \ref{Thm: Quantum ergodicity and quantum weak mixing}, proved by Marklof and Rudnick \cite{MR12}. For the standard eigenfunctions
\begin{align*}
	\psi_{(m,n)}(x,y) = e^{-2\pi i \left( \frac{mx}{L_x} + \frac{ny}{L_y} \right)}, \qquad (m,n) \in \mathbb{Z}^2,
\end{align*}
the claim is straightforward. However, establishing the result for all eigenbases seems to be subtle.

An example of a sequence of surfaces that is both quantum ergodic and quantum mixing in the above sense is a sequence of compact arithmetic surfaces. This follows from the result of Katz, Schnaps, and Vishne \cite{KSV07}, who showed that the injectivity radius grows (ensuring Properties \textbf{(BSC)} and \textbf{(UND)}), and from the result of Selberg \cite{Sel65}, who established a uniform lower bound of $ \tfrac{3}{16} $ on the spectral gap (ensuring Property \textbf{(EXP)}). We remark that the result of Kim, Sarnak, and Ramakrishnan \cite{KRS03} is the best deterministic bound currently available on the spectral gap. 

Moreover, the class of admissible sequences is quite broad, which is seen when using the probabilistic approach via the Weil--Petersson random model; for background on the Weil--Petersson random model, see \cite{Mir13, Wri20}. We state the probabilistic analogue of Theorem \ref{Thm: Large-scale analogue}, thus showing that a sequence of growing Weil--Petersson random surfaces is quantum ergodic and quantum mixing with high probability.

\begin{theorem}[Probabilistic version]
	\label{Thm: Probabilistic main theorem}

	Let $ ( X_{g} ) $ be a sequence of $ \mathbb{P}_{g} $-random surfaces, where $ \mathbb{P}_{g} $ is a Weil--Petersson probability distribution on the moduli space $ \mathcal{M}_{g} $, and where $ g $ is the genus. Let $-\Delta_{g}$ be the Laplace--Beltrami operator of $ X_{g} $ with eigenvalues 
	\begin{align*}
		0 = \lambda_{0}^{( g )} < \lambda_{1}^{( g )} \leq \dots \leq \lambda_{j}^{( g )} \leq \dots \xrightarrow{j \to \infty} \infty
	\end{align*}
	together with a corresponding orthonormal eigenbasis $ \{ \psi_{j}^{( g )} \} $ over $ L^{2}( X_{g} ) $. Define associated values of the eigenvalues as follows:
	\begin{align*}
		&
		\rho_{j}^{( g )} = \sqrt{ \lambda_{j}^{( g )} - \tfrac{1}{4} }.
	\end{align*}
	Let $ I \subset ( 0,\infty ) $ be a compact interval, and let $ ( a_{g} ) $ be a sequence of uniformly bounded functions on $ X_{g} $, each acting as a multiplication operator. Then Properties \emph{(i)}--\emph{(iii)} of Theorem \ref{Thm: Large-scale analogue} hold with high probability as the genus $ g $ grows to infinity.
	
\end{theorem}

We remark that the use of randomization is highly active in the field of quantum chaos. Randomization effectively suppresses pathological cases, helping to establish strong results; the analogous deterministic settings are far from acquiring the corresponding results. We especially want to highlight the following important results: the Benjamini--Schramm convergence and a lower bound on the injectivity radius in the Weil--Petersson random model were established in \cite{Mir13, Mon21}. For recent results on the spectral gap, see \cite{AM24, LW24, Mag24, HMT25, HMN25, WX25}. Another random model, the random covering model, has been used with great success to obtain results on the spectral gap; see \cite{MN20, MNP22, HM23}.

As mentioned, Le Masson and Sahlsten have already proven that Property \emph{(i)} holds in the setting of Theorem \ref{Thm: Large-scale analogue} (see \cite{LS17}). They also proved that Property \emph{(i)} holds in the setting of Theorem \ref{Thm: Probabilistic main theorem} (see \cite{LS24}). Hence, to establish Theorems \ref{Thm: Large-scale analogue} and \ref{Thm: Probabilistic main theorem}, it remains to prove Properties \emph{(ii)} and \emph{(iii)} in both settings. 

The approaches used to prove Theorems \ref{Thm: Large-scale analogue} and \ref{Thm: Probabilistic main theorem} are very similar: they rely on the quantitative version of the theorems, namely Theorem \ref{Thm: Quantitative main theorem}. Using the quantitative theorem together with Properties \textbf{(BSC)}, \textbf{(EXP)}, \textbf{(UND)}, and the large-scale Weyl law (see Theorem \ref{Thm: Weyl}) yields a relatively straightforward proof of Theorem \ref{Thm: Large-scale analogue}. 

When proving the probabilistic theorem, Theorem \ref{Thm: Probabilistic main theorem}, we do not have Properties \textbf{(BSC)}, \textbf{(EXP)}, and \textbf{(UND)} available, nor do we have the large-scale Weyl law. These tools can be replaced with those recorded in Theorem \ref{Thm: Random surface theory}. We remark that Theorem \ref{Thm: Random surface theory} is compiled from \cite[Thm. 1.7, 1.9]{LS24}; the references given for this theorem differ from those of Le Masson and Sahlsten \cite{LS24}, reflecting that the understanding of the spectral gap is evolving rapidly.

Next, we state the quantitative main theorem and comment on the approach taken in its proof.

\begin{theorem}[Quantitative version]
\label{Thm: Quantitative main theorem}

Let $ X $ be a compact connected hyperbolic surface with Laplace--Beltrami operator $ -\Delta_{X} $ that has eigenvalues
\begin{align*}
	0 = \lambda_{0} < \lambda_{1} \leq \dots \leq \lambda_{j} \leq \dots \xrightarrow{j \to \infty} \infty
\end{align*}
together with a corresponding orthonormal eigenbasis $ \{ \psi_{j} \} $ over $ L^{2}( X ) $. Define associated values of the eigenvalues as follows: 
\begin{align*}
	\rho_{j} = \sqrt{ \lambda_{j} - \tfrac{1}{4} }.
\end{align*}
Let $I \subset (0,\infty)$ be a compact interval and let $a$ be a mean $0$ function on $ X $ acting as a multiplication operator. Let 
\[
0 < \delta < \tfrac{2}{9} \min(I),
\qquad
T\delta = \tfrac{\pi}{2},
\]
and fix $\tau \in \mathbb{R}$. Then
\begin{align*}
\sum_{j:\rho_j \in I}
\ \sum_{\substack{ k: \rho_{k} \in I, \\ k \neq j, \\ |\rho_j - \rho_k - \tau| < \delta}}
\left| \langle \psi_j, a \, \psi_k \rangle \right|^2
\leq C_{8} \, \max( I )^{4} \, \frac{1}{T \, \beta^{3}} \, \left( \| a \|_{2}^{2} \, + \,  \| a \|_{\infty}^{2} \, | X \setminus X( 4T ) | \frac{e^{6 T}}{ \operatorname{InjRad}_{X} } \right),
\end{align*}
where $ | X \setminus X( 4T ) | $ is the volume of the points of $ X $ with injectivity radius smaller than $ 4T $, and where $\beta = \beta( \lambda_{1} )$, defined by
\begin{align*}
	\beta( x ) = 
	\begin{cases}	
		1 - \sqrt{1 - 4x}, \qquad &\text{ if } x \leq \frac{1}{4},
		\\
		1, \qquad &\text{ if } x > \frac{1}{4}.
	\end{cases}
\end{align*}
The coefficient appearing is $C_{8} = 5120000 \pi^{2}$.

\end{theorem}

\begin{remark}
\label{Rmk: Quantitative main theorem}

We observe from the statement of Theorem \ref{Thm: Quantitative main theorem} that we could relax many of the conditions appearing in the statement of Theorem \ref{Thm: Large-scale analogue}. For example, we could let the spectral gap and the injectivity radius slowly shrink as we proceed along the sequence. Thereby, Theorem \ref{Thm: Large-scale analogue} could be extended to encompass a larger family of sequences of surfaces.

\end{remark}

The proof of the quantitative theorem also follows the idea of Le Masson and Sahlsten \cite{LS17} in that we divide the argument into a spectral and a geometric component. However, we employ a different propagator, namely the hyperbolic wave propagator
$$
    P_{t} = h_{t}\!\left( \sqrt{ -\Delta_{X} - \tfrac{1}{4} } \right),
    \qquad h_{t}(x) = x^{-1} \sin(tx).
$$
The intuition behind $ P_{t} $ is that $u_{tt} = -\Delta \, u$, the wave equation in the hyperbolic plane, has solution $u = P_{t} \, f$
with initial data
\begin{align*}
	u( z,0 ) = 0, \qquad u_{t}( z,0 ) =  f,
\end{align*}
for some function $ f $. For details, see \cite{LP82}. A convenient form of the integral kernel of $ P_{t} $ is established in Section \ref{Sec: Propagators}.

Unlike the ball-averaging operator used by Le Masson and Sahlsten \cite{LS17}, which does not have as immediate a connection to the underlying classical dynamics, the hyperbolic wave propagator is obtained from the quantization of the geodesic flow, thus linking the quantum and classical settings in a natural manner. This choice of propagator renders the spectral side particularly clean, though it makes the geometric side more involved compared to the approach of Le Masson and Sahlsten \cite{LS17}.

To access the off-diagonal averages \( \langle \psi_{j}, a \psi_{k} \rangle \), where the energy levels of \( \psi_{j} \) and \( \psi_{k} \) may differ substantially, we introduce a modified propagator
\[
P_{t,\tau} = \cos(t \tau)\, P_{t}.
\]  
Using \( P_{t} \) and \( P_{t,\tau} \) simultaneously allows us to capture the behavior of these off-diagonal averages. However, this complicates both the spectral and geometric analyses compared to the diagonal case studied by Le Masson and Sahlsten \cite{LS17}.

In particular, the Heisenberg evolution of the observable \( a \), \( P_{t}^{*} a P_{t} \), which appears in the proof of quantum ergodicity by Le Masson and Sahlsten, is replaced here by \( P_{t,\tau}^{*} a P_{t} \). The physical interpretation of this operator is not immediately transparent, but it can be formally verified that
\[
P_{t,\tau}^{*} a P_{t}
= \cos(t \tau)\, P_{t}^{*} a P_{t} 
= 2 \cos\!\left(t \frac{\tau}{2}\right) P_{t}^{*} a \cos\!\left(t \frac{\tau}{2}\right) P_{t} - P_{t}^{*} a P_{t} 
= 2 P_{t,\frac{\tau}{2}}^{*} a P_{t,\frac{\tau}{2}} - P_{t}^{*} a P_{t},
\]
where we used the identity \(\cos(2x) = 2 \cos^2(x) - 1\). This shows that \( P_{t,\tau}^{*} a P_{t} \) can be written as the subtraction of two Heisenberg evolutions of \( a \), one with \( P_{t,\frac{\tau}{2}} \) and one with \( P_t \).

Finally, we replace the rather abstract Nevo’s theorem with the exponential mixing theorem (see Theorem~\ref{Thm: Exponential mixing}), which elucidates the connection between classical and quantum dynamics. To use the exponential mixing theorem, we open the squares of the Hilbert--Schmidt norm appearing in the proof, as opposed to Le Masson and Sahlsten \cite{LS17}, who opt to use Minkowski's integral inequality. In this way, we lose less decay and see an interesting exponential decorrelation of Heisenberg evolutions $ P_{t}^{*} a P_{t} $ and $ P_{t'}^{*} a P_{t'} $ when $ | t - t' | $ increases. The use of exponential mixing is both simple and transparent: the exponential decay of correlations of \( a \) and \( a \circ \varphi_{t} \), even for small \( t \), is essential in obtaining our result. Small propagation times suffice because we work with observables in a macroscopic sense, unlike the large-energy quantum ergodicity and quantum mixing setting, which employs microlocal analysis and requires long propagation times for decorrelation to occur at the microlocal scale. We already refer to Section \ref{Sec: Strategy} for a more technical overview of the proofs.


\subsection{Further discussions}

Like Le Masson and Sahlsten \cite{LS17}, we restrict the energy window to a compact interval in \( \left( 1/4, \infty \right) \). This assumption plays a crucial role in our proofs. However, it is plausible that with a different treatment of the low-lying modes, namely, those eigenvalues and corresponding eigenfunctions with eigenvalues in \( [0, 1/4] \), one could extend the results to an energy window covering the full range \( [0, \infty) \).

Another avenue of extension concerns the choice of observables. In Theorem \ref{Thm: Quantum ergodicity and quantum weak mixing}, the observable class was quite broad, zeroth-order pseudo-differential operators, so it seems that with some modifications, a larger class of observables could be accommodated in the large-scale setting as well. In the large-scale quantum ergodicity setting for graphs \cite{AL15} and surfaces \cite{ABL22}, a broader class of observables is already considered when treating the diagonal averages (Property \emph{(i)}): namely, observables whose integral kernels are concentrated near the diagonal, rather than being strictly supported on it as in the present work.

In contrast to Theorem \ref{Thm: Quantum ergodicity and quantum weak mixing}, Theorem \ref{Thm: Large-scale analogue} is not a characterization-type result linking the classical and quantum dynamics. A natural question arises: can Theorem \ref{Thm: Large-scale analogue} be complemented to yield such a characterization? As noted in Remark \ref{Rmk: Quantitative main theorem}, the sequences of surfaces may in fact possess slowly decaying injectivity radii and spectral gaps. Given the connection between weak mixing and the spectral gap, it is plausible that a sequence of hyperbolic manifolds converging to a hyperbolic space in the Benjamini--Schramm sense is large-scale quantum mixing if and only if the spectral gap tends to zero sufficiently slowly. Establishing such characterizations for both large-scale quantum ergodicity and large-scale quantum mixing would be remarkable for further illustrating the connections between the classical and quantum settings.

While full characterization results may still be out of reach, several directions for possible extensions present themselves, albeit with varying levels of difficulty. For instance, can analogous results be established in other structures such as graphs? We conjecture that the graph result of Brooks, Lindenstrauss, and Le Masson \cite{BLL16} could be extended to include off-diagonal averages in the setting of the large-scale quantum ergodicity theorem for graphs by Anantharaman and Le Masson \cite{AL15} by using the ideas presented here. Performing quantum mixing on graphs seems highly timely, as Bordenave, Letrouit, and Sabri have a recent result on quantum mixing on growing Schreier graphs \cite{BLS26}.

Yet another natural question is whether the results can be adapted to the case of variable curvature or to accommodate more general energy operators than the Laplacian. For graphs, analogous results are known: Anantharaman and Sabri \cite{AS17, AS19} established large-scale quantum ergodicity for non-regular graphs with more general Schrödinger operators than the Laplacian, under suitable restrictions. Non-regular graphs can be viewed as discrete analogues of variable-curvature surfaces. It would also be of great interest to construct counterexamples that would highlight the necessity of the conditions imposed on the sequences, namely Properties \textbf{(BSC)}, \textbf{(EXP)}, and \textbf{(UND)}. In the graph setting, the results of Anantharaman and Le Masson were complemented in this manner by McKenzie \cite{McK22}; however, constructing such examples in the surface setting seems substantially more difficult.

We pose perhaps an obvious question: is it possible to remove the averaging present in our theorems? In the large-energy regime, this corresponds to the quantum unique ergodicity conjecture of Rudnick and Sarnak, see \cite{RS94}. The conjecture has been resolved in certain special cases (notably by Lindenstrauss for arithmetic surfaces and Hecke eigenfunctions \cite{Lin06}), but remains open in general. We expect that the conclusions of Theorem \ref{Thm: Large-scale analogue} and Theorem \ref{Thm: Probabilistic main theorem} remain valid even without averaging, possibly requiring only minor adjustments to the statements. Especially, the averaging-free version of the probabilistic Theorem \ref{Thm: Probabilistic main theorem} would parallel the averaging-free results for random graphs \cite{BHKY17, BKY17, BHY19} and random matrices \cite{CES21}. However, the averaging-free version of the deterministic Theorem \ref{Thm: Large-scale analogue} would more closely parallel the quantum unique ergodicity conjecture; it is likely harder to establish than the probabilistic version.

For arithmetic surfaces, averaging-free off-diagonal bounds are of significant interest: strong estimates on 
\( |\langle \psi_{j}, \psi_{k} \psi_{l} \rangle| \) could contribute to a proof of the Lindelöf hypothesis, to which quantum chaos is elegantly connected; see Watson's work on \( L \)-functions \cite{Wat02}.

In any case, a deeper understanding of averaging-free results may provide a pathway toward results in many-body quantum chaos, such as the Eigenstate Thermalization Hypothesis proposed by Deutsch in the physics literature \cite{Deu18}. Deutsch predicts that $ |\langle \psi_{j}, a \psi_{k} \rangle| $ goes to zero as the distance between the eigenvalues of $ \psi_{j} $ and $ \psi_{k} $ increases. Looking at the large-scale results discussed here, Theorems \ref{Thm: Large-scale analogue} and \ref{Thm: Probabilistic main theorem}, we see that the parameter $ \tau $ controls the distance of the eigenvalues. However, the eigenvalues are taken from fixed energy windows, which means that for large enough $ \tau $, the averages of the off-diagonal terms will be zero, as the large distance is not possible in the energy window. If we let the energy window grow, then $ \tau $ can also grow without giving trivial conclusions. It is natural to assume that in this setting of growing energy window and growing $ \tau $, the average will be close to zero, but verifying this requires careful study of the sequence of growing surfaces, the parameter $ \tau $, and the energy window.


\section{Preliminaries}

\noindent
In this section, we recall some basic concepts and results that will be used throughout this article without proofs. For background on hyperbolic geometry concerning the hyperbolic plane and hyperbolic surfaces, we refer to Katok \cite{Kat92}, and for details on the Laplace--Beltrami operator, to Iwaniec \cite{Iwa02}.

\subsection{Hyperbolic plane}

The hyperbolic plane \( \mathbb{H} \) is identified with the upper half-plane
\[
    \{ x + i y \in \mathbb{C} : y > 0 \},
\]
equipped with the Riemannian metric
\[
    ds^{2} = y^{-2}(dx^{2} + dy^{2}).
\]
For points \( z, z' \in \mathbb{H} \), we denote by \( d(z, z') \) the hyperbolic distance between them. The associated volume measure is
\[
    d\mu(x + i y) = y^{-2} \, dx \, dy,
\]
and we denote by \( |A| \) the hyperbolic volume of a set \( A \subset \mathbb{H} \).

The group \( PSL_{2}(\mathbb{R}) \), consisting of \( 2 \times 2 \) real matrices with determinant one, modulo \( \pm \mathrm{id} \), acts on \( \mathbb{H} \) by Möbius transformations:
\[
    \begin{pmatrix} a & b \\ c & d \end{pmatrix} \cdot z = \frac{a z + b}{c z + d}.
\]
This action identifies \( PSL_{2}(\mathbb{R}) \) with the group of orientation-preserving isometries of \( \mathbb{H} \).

The unit tangent bundle of \( \mathbb{H} \) is
\[
    T^{1}\mathbb{H} = \{ (z, v) \in \mathbb{H} \times T_{z}\mathbb{H} : \|v\|_{z} = 1 \}.
\]
The group \( PSL_{2}(\mathbb{R}) \) acts on \( T^{1}\mathbb{H} \) via the action
\[
    \begin{pmatrix} a & b \\ c & d \end{pmatrix} \cdot (z, v)
    = \left( \frac{a z + b}{c z + d}, \frac{v}{(c z + d)^{2}} \right).
\]
Evaluating this action at the point \( (i, i) \) shows that \( PSL_{2}(\mathbb{R}) \) is homeomorphic to \( T^{1}\mathbb{H} \). Thus these two spaces can be identified.

The geodesic flow is the one-parameter family of maps
\[
    \varphi_{t} : T^{1}\mathbb{H} \to T^{1}\mathbb{H}, \qquad t \in \mathbb{R},
\]
such that for any \( (z, v) \in T^{1}\mathbb{H} \), the curve \( (\eta(t), \eta'(t)) = \varphi_{t}(z, v) \) is the unique geodesic parametrized by arc length satisfying \( \eta(0) = z \) and \( \eta'(0) = v \). The geodesic
\[
    (i, i) \mapsto \varphi_{t}(i, i)
    = \begin{pmatrix} e^{t/2} & 0 \\ 0 & e^{-t/2} \end{pmatrix} \cdot (i, i)
\]
corresponds to the imaginary axis. For any \( (z, v) \in T^{1}\mathbb{H} \), there exists \( g \in PSL_{2}(\mathbb{R}) \) such that $ g \cdot (i, i) = (z, v) $, and thus
\[
    \varphi_{t}(z, v) = g
    \begin{pmatrix} e^{t/2} & 0 \\ 0 & e^{-t/2} \end{pmatrix}
    \cdot (i, i).
\]
Hence, identifying \( T^{1}\mathbb{H} \) with \( PSL_{2}(\mathbb{R}) \), the geodesic flow corresponds to the right action of the one-parameter subgroup
\[
    A = \left\{ a_{t} =
    \begin{pmatrix} e^{t/2} & 0 \\ 0 & e^{-t/2} \end{pmatrix}
    : t \in \mathbb{R} \right\}.
\]

For any distinct points \( z, z' \in \mathbb{H} \), there exists a unique geodesic of length \( r \) from \( z \) to \( z' \). Using the geodesic flow, there exist unique \( \theta \in \mathbb{S}^{1} \) and \( r \in (0, \infty) \) such that \( z' \) is the projection of \( \varphi_{r}(z, \theta) \) onto the first coordinate. The change of variables \( z' \mapsto (r, \theta) \) defines the polar coordinates around \( z \), in which the metric and volume element are
\[
    ds^{2} = dr^{2} + \sinh^{2}(r)\, d\theta^{2}, \qquad
    d\mu(r, \theta) = \sinh(r)\, dr\, d\theta.
\]
Point $ x \in \mathbb{H} $ can be written as
\begin{align*}
	&
	x = \pi_{1}( \phi_{r}( z, \theta ) ),
\end{align*}
where the polar coordinates are around $ z $; $ r $ and $ \theta $ depend on $ x $, and $ \pi_{1} $ is the projection onto the first element of the pair in $ T^{1} \mathbb{H} $.

We denote by \( [z, z'] \) the geodesic segment going from $ z \in \mathbb{H} $ to another point $ z' \in \mathbb{H} $. By \( [z, z'](w) \) we denote the direction of the tangent vector of the geodesic at point \( w \in [z, z'] \), and by \( |[z, z']| \) we denote the length of the geodesic segment.

The Laplace--Beltrami operator on \( \mathbb{H} \) is the differential operator
\[
    -\Delta = -y^{2}\left( \frac{\partial^{2}}{\partial x^{2}} + \frac{\partial^{2}}{\partial y^{2}} \right),
\]
which commutes with all isometries of \( \mathbb{H} \). We refer to it as the Laplacian operator.

\subsection{Hyperbolic surfaces}

A compact connected hyperbolic surface can be realized as a quotient
\[
    X = \Gamma \setminus \mathbb{H},
\]
where \( \Gamma \subset PSL_{2}(\mathbb{R}) \) is a discrete subgroup consisting solely of hyperbolic elements, that is, matrices whose trace satisfies \( |\mathrm{tr}(\gamma)| > 2 \); $ \Gamma $ is a Fuchsian group. In this way, \( X \) can be represented by a compact subset of $ \mathbb{H} $ called a fundamental domain.

A function on \( X \) can equivalently be regarded as a \( \Gamma \)-periodic function on \( \mathbb{H} \), and vice versa. Moreover, the integral of a function \( f \) over \( X \) equals the integral of the corresponding \( \Gamma \)-invariant function over a fundamental domain of $ X $.

For a hyperbolic surface \( X = \Gamma \setminus \mathbb{H} \), the unit tangent bundle can be identified with \( \Gamma \setminus PSL_{2}(\mathbb{R}) \), and under this identification the geodesic flow takes the form
\[
    \varphi_{t}(\Gamma g) = \Gamma g
    \begin{pmatrix} e^{t/2} & 0 \\ 0 & e^{-t/2} \end{pmatrix}.
\]

The injectivity radius of \( X \) at a point \( z \in X \) is defined by
\[
    \operatorname{InjRad}_{X}(z)
    := \tfrac{1}{2} \min_{\gamma \in \Gamma \setminus \{\mathrm{id}\}} d(z, \gamma z),
\]
and the (global) injectivity radius of \( X \) is
\[
    \operatorname{InjRad}_{X}
    := \inf_{z \in X} \operatorname{InjRad}_{X}(z).
\]

Since the Laplacian on \( \mathbb{H} \) commutes with isometries, it descends naturally on the quotient \( X = \Gamma \setminus \mathbb{H} \). To distinguish the two, we denote by \( -\Delta_{\mathbb{H}} \) the Laplacian acting on \( \mathbb{H} \), and by \( -\Delta_{X} \) the Laplacian acting on \( \Gamma \)-invariant functions. On a compact connected hyperbolic surface, the Laplacian \( -\Delta_{X} \) has a discrete spectrum
\[
    0 = \lambda_{0} < \lambda_{1} \le \lambda_{2} \le \dotsb \xrightarrow{j \to \infty} \infty.
\]
The spectral gap of $ -\Delta_{X} $ is the difference between its two smallest eigenvalues.


\section{Proof ideas}

\label{Sec: Strategy}

\noindent
Three main theorems, Theorems \ref{Thm: Large-scale analogue}, \ref{Thm: Probabilistic main theorem}, and \ref{Thm: Quantitative main theorem}, are proven in Section \ref{Sec: Main theorems} using the auxiliary results established in Sections \ref{Sec: Spectral data}, \ref{Sec: Propagators}, \ref{Sec: Geometric data}, and \ref{Sec: Weight function}. In this section we give an overarching idea of all the proofs to provide a coherent narrative; we refer to the latter sections for the exact definitions of the appearing quantities. We remark that in the actual proof, the coefficients are written carefully to make all dependencies explicit; here, we adopt a less explicit presentation.

The goal is to prove Properties \emph{(ii)} and \emph{(iii)} of Theorems \ref{Thm: Large-scale analogue} and \ref{Thm: Probabilistic main theorem}. We prove the properties simultaneously: We want to show that for any $ \epsilon > 0 $ there exists $ \delta( \epsilon ) > 0 $ such that for any $ \tau \in \mathbb{R} $ we have
\begin{align}
	&
	\label{Expr: Limit and double sum}
	\limsup_{n \to \infty}  \frac{1}{N( X_{n}, I )} \, \sum_{j: \rho_{j}^{( n )} \in I}^{}  \sum_{ \substack{k: \rho_{k}^{( n )} \in I, \\ k \neq j, \\  | \rho_{j} - \rho_{k} - \tau | < \delta( \epsilon )} }^{} | \langle \psi_{j}^{( n )}, a_{n} \psi_{k}^{( n )} \rangle |^{2} < \epsilon.
\end{align}
The setting is deterministic in Theorem \ref{Thm: Large-scale analogue} and probabilistic in Theorem \ref{Thm: Probabilistic main theorem}. To achieve this goal, we aim to bound the double sum for a fixed surface $ X $, which is the content of the quantitative main theorem, Theorem \ref{Thm: Quantitative main theorem}.

We proceed with the idea of the proof of Theorem \ref{Thm: Quantitative main theorem}. We set out to bound the following expression:
\begin{align}
	&
	\label{Expr: Double sum}
	\sum_{j: \rho_{j} \in I}^{} \, \sum_{ \substack{ k: \rho_{k} \in I, \\ k \neq j, \\ | \rho_{j} - \rho_{k} - \tau | < \delta } }^{} \left| \langle \psi_{j}, a \psi_{k} \rangle \right|^{2}.
\end{align}
We can rewrite the inner product as follows:
\begin{align}
	&
	\label{Expr: Geometric-spectral split}
	| \langle \psi_{j}, a \psi_{k} \rangle | = \frac{ \left| \frac{1}{T} \int_{0}^{T} h_{t,\tau}( \rho_{j}  ) h_{t}( \rho_{k} ) dt \right| }{ \left| \frac{1}{T} \int_{0}^{T} h_{t,\tau}( \rho_{j}  ) h_{t}( \rho_{k} ) dt \right| } \, | \langle \psi_{j}, a \psi_{k} \rangle | = \frac{ \left| \left \langle \psi_{j} , \frac{1}{T} \int_{0}^{T} P_{t,\tau}^{*} a P_{t} \, \psi_{k} \, dt \right \rangle \right| }{ \left| \frac{1}{T} \int_{0}^{T} h_{t, \tau}( \rho_{j} ) \, h_{t}( \rho_{k} ) \, d t \right| },
\end{align}
where 
\begin{align*}
	P_{t} = h_{t}\left( \sqrt{ -\Delta_{X} - \frac{1}{4} } \right), \qquad h_{t}( x ) = \frac{\sin( tx )}{x},
\end{align*}
and
\begin{align*}
	P_{t,\tau} = h_{t,\tau}\left( \sqrt{ -\Delta_{X} - \frac{1}{4} } \right), \qquad h_{t,\tau} = \cos( t \tau ) \, h_{t}.
\end{align*}
$ T $ depends on $ \delta $ via $T \delta = \tfrac{\pi}{2}$.

We refer to this as the geometric-spectral split, since the denominator contains the spectral data of the problem and the numerator contains the geometric data. Two different operators, $ P_{t} $ and $ P_{t,\tau} $, are used because we are dealing with off-diagonal averages where the separation of the eigenvalues is significant. 

The denominator of Expression \eqref{Expr: Geometric-spectral split} contains the spectral data. We acquire the following bounds for the denominator in Section \ref{Sec: Spectral data}:
\begin{align*}
	&
	\left| \frac{1}{T} \int_{0}^{T} h_{t, \tau}( \rho_{j} ) \, h_{t}( \rho_{k} ) \, d t \right|^{-1} \leq 8 \pi \rho_{j} \rho_{k} \leq 8 \pi \max( I )^{2}.
\end{align*}
The bound is obtained by carefully integrating and using the interrelations of $ \delta, T, \rho_{j}, $ and $ \rho_{k} $. Since we are working in a compact energy window, $ \rho_{j}, \rho_{k} \leq \max( I ) $. It is especially important that $ \delta \in (0 , \tfrac{2}{9} \min I) $. Being able to choose such a positive $ \delta $ is crucial; this is why we have excluded the low-lying notes, i.e., those eigenfunctions of the Laplacian $ -\Delta_{X} $ with eigenvalues in the window $ [0, \tfrac{1}{4}] $.

Using the geometric-spectral split and the bounds for the spectral data gives an upper bound for Expression \eqref{Expr: Double sum}:
\begin{align*}
	&
	\sum_{j: \rho_{j} \in I}^{} \, \sum_{ \substack{ k: \rho_{k} \in I, \\ k \neq j, \\ | \rho_{j} - \rho_{k} - \tau | < \delta } }^{} \left| \langle \psi_{j}, a \psi_{k} \rangle \right|^{2}
\leq
	64 \pi^{2} \max( I )^{4}  \,  \sum_{j: \rho_{j} \in I}^{} \, \sum_{ \substack{ k: \rho_{k} \in I, \\ k \neq j, \\ | \rho_{j} - \rho_{k} - \tau | < \delta } }^{} \left| \left \langle \psi_{j}, \frac{1}{T} \int_{0}^{T} P_{t,\tau}^{*} a P_{t} dt \, \psi_{k} \right \rangle \right|^{2} ,
\end{align*}
where this has the following Hilbert-Schmidt norm upper bound by enlarging the double sum to include all basis elements:
\begin{align*}
	&
	64 \pi^{2} \max( I )^{4} \, \left \| \frac{1}{T} \int_{0}^{T} P_{t,\tau}^{*} a P_{t} dt \right \|_{HS}^{2}.
\end{align*}
The Hilbert-Schmidt norm contains the geometric data of the problem; it contains no eigenvalues or eigenfunctions and is now purely geometric.

In Section \ref{Sec: Propagators} we 
establish convenient forms for the integral kernels of $ P_{t} $ and $ P_{t,\tau} $ which are used to bound the Hilbert-Schmidt norm in Section \ref{Sec: Geometric data}. $ P_{t} $ has integral kernel $ K_{t}^{\Gamma} $:
\begin{align*}
	&
	K_{t}^{\Gamma}( x,y ) = \frac{1}{2 \sqrt{2} \pi} \sum_{\gamma \in \Gamma}^{} K_{t}( x, \gamma y ), \qquad K_{t}( x,y ) = \frac{\boldsymbol 1_{t > d( x,y )}}{\sqrt{ \cosh( t ) - \cosh( d( x,y ) ) }}.
\end{align*}
$ P_{t,\tau} $ has integral kernel $ K_{t,\tau}^{\Gamma} = \cos( t \tau ) K_{t}^{\Gamma} $. We notice that $ P_{t} $ and $ P_{t,\tau} $ are self-adjoint from these integral kernel formulas. Informally, the previous formulas for the integral kernels are obtained by using the inverse Selberg transform $ \mathcal{S}^{-1} $ (e.g., see \cite{LS17}): 
$$ K_{t}^{\Gamma}( x,y ) = \sum_{\gamma \in \Gamma}^{} \mathcal{S}^{-1}( h_{t} )( x, \gamma y ).$$
However, $ h_t $ does not behave well enough for this to work immediately, and hence a more careful approach is needed; we approximate $ P_{t} $ and $ P_{t,\tau} $ to use analogous techniques. 

By using the acquired formulas for the kernels, we bound the Hilbert-Schmidt norm in Section \ref{Sec: Geometric data}. Opening the Hilbert-Schmidt norm using the obtained kernel formulas and carefully rewriting it gives the following equivalent expression:
\begin{align*}
	&
	\frac{1}{8 \pi^{2} T^{2}} \int_{0}^{T} \int_{0}^{T} \cos( t \tau ) \cos( t' \tau ) \int_{D} a(z) \int_{ B( z,2T ) } a( z' )^{*} F_{t,t', d( z,z' )} \sum_{\gamma \in \Gamma}^{} F_{t,t',d( \gamma z,z' )} dz' dz dt' dt,	
\end{align*}
where $ D $ is a fundamental domain of $ X $. The previous expression has the following upper bound obtained by taking absolute values, moving them inside, and bounding $ | \cos( \cdot ) | $ by $ 1 $:
\begin{align}
	&
	\label{Expr: HS tidying}
	\frac{1}{8 \pi^{2} T^{2}} \int_{0}^{T} \int_{0}^{T} \left| \int_{D} a(z) \int_{ B( z,2T ) } a( z' )^{*} F_{t,t', d( z,z' )} \sum_{\gamma \in \Gamma}^{} F_{t,t',d( \gamma z,z' )} dz' dz \right| dt' dt.
\end{align}
Eliminating $\cos$ factors, which appear due to the simultaneous presence of $P_{t}$ and $P_{t,\tau}$, might lead to some loss of decay. A more careful approach could yield better estimates. $ F_{t,t',\rho} $ contains the core of the geometric data and is defined as follows:
\begin{align*}
	F_{t,t',\rho} =  \int_{ B( z,t ) \cap B( z',t' ) } \frac{dx}{ \sqrt{ \cosh( t ) - \cosh( d( x,z ) ) } \, \sqrt{ \cosh( t' ) - \cosh( d( x,z' ) ) }  }.
\end{align*}
Here $ z,z' \in \mathbb{H} $ are any points separated by $ \rho $. $ F_{t,t',\rho} $ is studied in Section \ref{Sec: Weight function}. An especially important fact is that $ F_{t,t', \rho} = 0 $ when $ \rho > t + t' $.

In Expression \eqref{Expr: HS tidying}, we split the integration over the fundamental domain $ D $ into points with injectivity radius larger than $ 4T $, $ D( 4T ) $, and into points with injectivity radius smaller than $ 4T $, $ D \setminus D( 4T ) $. This approach leads to the following upper bound:
\begin{align}
	&
	\label{Expr: IR split}
	\left| \int_{D} a( z ) \int_{ B( z, 2T ) } a( z' )^{*}   F_{t,t', d( z,z' )}^{2}   dz' dz \right| 
	&
	+ \ 2 \, \|a \|_{\infty}^{2} \int_{D \setminus D( 4T )} \int_{ B( z, 2T ) }  F_{t,t', d( z,z' )}  \sum_{\gamma \in \Gamma}^{} F_{t,t', d( \gamma z, z' )} dz' dz.
\end{align}
For points with large injectivity radius, all group elements except $ id \in \Gamma $ contribute zero, thus eliminating the sum. The latter term deals with points with small injectivity radius. The terms are treated separately.

The first term of Expression \eqref{Expr: IR split} can be written using polar coordinates around $ z $ and has the following upper bound:
\begin{align*}
	&
	\int_{0}^{2T} \sinh( \rho ) F_{t,t',\rho}^{2} \left| \langle a, a \circ \phi_{\rho} \rangle \right|  d\rho \leq 30 \| a \|_{2}^{2} \int_{0}^{2T} \sinh( \rho ) ( 1 + \rho ) e^{ -\beta \rho } F_{t,t',\rho}^{2},
\end{align*}
where the latter inequality is obtained using the exponential mixing theorem, Theorem \ref{Thm: Exponential mixing}. The exponential decay obtained via the exponential mixing theorem is perhaps the most crucial step in the entire proof. $ \beta $ is a constant depending on the spectral gap.

The second term of Expression \eqref{Expr: IR split} is bounded more roughly using Hölder's inequality. The interplay of $ F_{t,t', d( z,z' )} $ and $ F_{t,t',d( \gamma z, z' )} $ ensures that only finitely many group elements contribute. The upper bound can be written as:
\begin{align*}
	&
	8 \pi^{2} \| a \|_{\infty}^{2} |D \setminus D( 4T )| \frac{e^{4T}}{ \operatorname{InjRad}_{X} } \int_{0}^{2T} \sinh( \rho ) ( 1 + \rho ) e^{ -\beta \rho } F_{t,t',\rho}^{2}.
\end{align*}

Using the above bounds, we obtain an upper bound for Expression \eqref{Expr: HS tidying}:
\begin{align*}
	&
	\frac{1}{T^{2}} \left( \| a \|_{2}^{2} + \| a\|_{\infty}^{2} |D \setminus D( 4T )| \frac{e^{6T}}{ \operatorname{InjRad}_{X} } \right) \int_{0}^{T} \int_{0}^{T} \int_{0}^{2T} \sinh( \rho ) e^{ - \beta \rho } ( 1+ \rho ) F_{t,t',\rho}^{2} d \rho dt' dt. 
\end{align*}
This is straightforwardly bounded when we have the following bound for $ F_{t,t',\rho} $:
\begin{align}
	\label{Expr: Final bound}
	F_{t,t',\rho} \lesssim \sinh( \max( \rho, | t - t' | ) )^{-\tfrac{1}{2}}.
\end{align}
Obtaining this bound for \( F_{t,t',\rho} \) is the most technical part of the paper and is carried out in Section \ref{Sec: Weight function}. The idea is to compute a weighted volume of the intersection of two hyperbolic balls using hyperbolic geometry and trigonometry. In practice, the integration is handled in multiple separate cases.

The final upper bound is a constant times
\begin{align*}
	&
	\frac{1}{T \beta^{3}} \left( \| a \|_{2}^{2} + \| a \|_{\infty}^{2} | X \setminus X( 4T ) | \frac{ e^{6T} }{ \operatorname{InjRad}_{X} }  \right),	
\end{align*}
where $ | X \setminus X( 4T ) | $ denotes the set of points with small injectivity radius. This shows that the double sum Expression \eqref{Expr: Double sum} is bounded using the Hilbert-Schmidt norm bound:
\begin{align*}
	&
	\max( I )^{4} \, \frac{1}{T \beta^{3}} \left( \| a \|_{2}^{2} + \| a \|_{\infty}^{2} | X \setminus X( 4T ) | \frac{ e^{6T} }{ \operatorname{InjRad}_{X} }  \right).
\end{align*}
A constant factor is omitted. This establishes the proof idea of Theorem \ref{Thm: Quantitative main theorem}.

We proceed to Theorems \ref{Thm: Large-scale analogue} and \ref{Thm: Probabilistic main theorem}. We want to find an upper bound $ \epsilon $ for a suitable $ \delta $ for the left-hand side of Expression \eqref{Expr: Limit and double sum}. Using the bounds from Theorem \ref{Thm: Quantitative main theorem}, we get the following upper bound (omitting a constant factor):
\begin{align}
	&
	\label{Expr: Works both deterministic and probabilistic}
	\max( I )^{4} \frac{1}{T} \, \limsup_{n \to \infty} \frac{1}{N( X_{n}, I )} \frac{1}{ \beta( \lambda_{1}( X_{n} ) )^{3} } \left( \| a_{n} \|_{2}^{2} + \| a_{n} \|_{\infty}^{2} | X \setminus X( 4T ) | \frac{ e^{6T} }{ \operatorname{InjRad}_{X} }  \right).
\end{align}
In the deterministic setting of Theorem \ref{Thm: Large-scale analogue}, we use that $ ( a_{n} ) $ is uniformly bounded by $ a_{max} $, $ I $ is compact, \textbf{(EXP)} gives a universal lower bound for $ ( \beta(\lambda_{1}( X_{n} ) ) ) $ denoted by $ \beta_{min} $, and \textbf{(UND)} provides a uniform lower bound for $ ( \operatorname{InjRad}_{X_{n}} ) $ denoted by $ \operatorname{InjRad}_{min} $. The upper bound is
\begin{align*}
	&
	\frac{ \max( I )^{4}  a_{max}^{2} }{ \beta_{min}^{3} } \frac{1}{T} \limsup_{n \to \infty} \left( \frac{ |X_{n}| }{ N( X_{n}, I ) } + \frac{ |X_{n} \setminus X_{n}( 4T )| \, e^{6T} }{ N( X_{n}, I  ) \, \operatorname{InjRad}_{min}  } \right).
\end{align*}
The large-scale Weyl law gives $ N( X_{n}, I ) \sim | X_{n} | $. Using Property \textbf{(BSC)}, the volume ratio of points with small injectivity radius shrinks to zero. Hence, the upper bound becomes
\begin{align*}
	 \frac{ \max( I )^{4} a_{max}^{2}  }{ \beta_{min}^{3} } \frac{1}{T}.
\end{align*}
Since \(T\delta=\tfrac{\pi}{2}\), we may choose \(\delta\) sufficiently small so that the bound is less than \(\epsilon\). This completes the proof idea for Theorem \ref{Thm: Large-scale analogue}.

In the probabilistic setting of Theorem \ref{Thm: Probabilistic main theorem}, a similar approach is taken; here $ n = g $ is the genus of the surface. We do not have Properties \textbf{(BSC)}, \textbf{(EXP)}, \textbf{(UND)}, or the large-scale Weyl law a priori. However, Theorem \ref{Thm: Random surface theory} provides analogous tools. For a fixed surface $X_{g}$ of genus $ g $,
\begin{align*}
	&
		\frac{1}{ N( X_{g}, I ) \, \beta( \lambda_{1}^{( g )} )^{3} } \left( | X_{g} | + | X_{g} \setminus X_{g}( 4T ) | \, \frac{ e^{6T} }{ \operatorname{InjRad}_{X_{g}} } \right)
		\\
		& \leq \
		\frac{| X_{g} |}{ | X_{g} | \, \mathcal{O}\left( \max( J ) - \min( J ) + \sqrt{ \frac{ \max( J  ) + 1 }{ \log( g ) } } \right) \, \left( 1 - \sqrt{1 - 4 ( \tfrac{1}{4} - \eta ) } \right)^{3}} \, \left( 1 + \mathcal{O}( g^{ - \frac{1}{4} } ) \, \frac{e^{6T}}{ g^{ - \frac{1}{24}  } \, \log( g )^{ \frac{9}{16} } } \right),
\end{align*}
with probability $ 1 - o( 1 ) $, where $ \eta \in ( 0, \tfrac{1}{4} ) $ and
\begin{align*}
	J = \left \{ x^{2} + \tfrac{1}{4} \ : \ x \in I \right \}.
\end{align*}
Hence Expression \eqref{Expr: Works both deterministic and probabilistic} equals, with high probability in the high-genus limit (up to a constant factor),
\begin{align*}
	&
	a_{max}^{2} \max( I )^{4} \frac{1}{T}.
\end{align*}
Since \(T\delta=\tfrac{\pi}{2}\), we may choose \(\delta\) sufficiently small so that the bound is less than \(\epsilon\). This completes the proof idea for Theorem \ref{Thm: Probabilistic main theorem}.


\section{Spectral data}
\label{Sec: Spectral data}

\noindent
In this section, we prove a bound for the spectral quantity
\begin{align*}
	&
	\left| \frac{1}{T} \int_{0}^{T} h_{t,\tau}(a)\, h_{t}(b)\, dt \right|^{-1}
\end{align*}
under certain assumptions on $a,b,t,\tau,$ and $T$, given that
\begin{align*}
	&
	h_{t}( x ) = x^{-1} \, \sin( tx ), \qquad h_{t,\tau}( x ) = \cos( t \tau) \, h_{t}( x ).
\end{align*}
This spectral quantity appears in the proof of Theorem \ref{Thm: Quantitative main theorem}, where it encodes the spectral data of the main problem. We state the proposition.

\begin{proposition}[Spectral data]
\label{Prp: Spectral data}
Let $\tau \in \mathbb{R}$. Assume that $0 < m < a,b < \infty$ and that 
\[
a - b - \tau \in (-\delta,\delta),
\]
where $ \delta \in ( 0, \tfrac{2}{9} m)$. Then for $ T \delta = \tfrac{\pi}{2} $ we have that
$$
\left| \frac{1}{T} \int_{0}^{T} h_{t,\tau}(a)\, h_{t}(b)\, dt \right|^{-1}  <  8 \pi a b .
$$
\end{proposition}

The proof consists of straightforward integration and careful application of the relationships among the variables. The simplicity follows from the convenient choices of $h_{t}$ and $h_{t,\tau}$.

\begin{proof}[Proof of Proposition \ref{Prp: Spectral data}]
We set out to bound 
\[
G_{\tau, T} = \left| \frac{1}{T} \int_{0}^{T} h_{t,\tau}(a)\, h_{t}(b)\, dt \right|
\]
from below. $G_{\tau,T}$ is illustrated in Figure \ref{Fig: Spectral}.
\begin{figure}[H]
    \centering
    \includegraphics[width=250px]{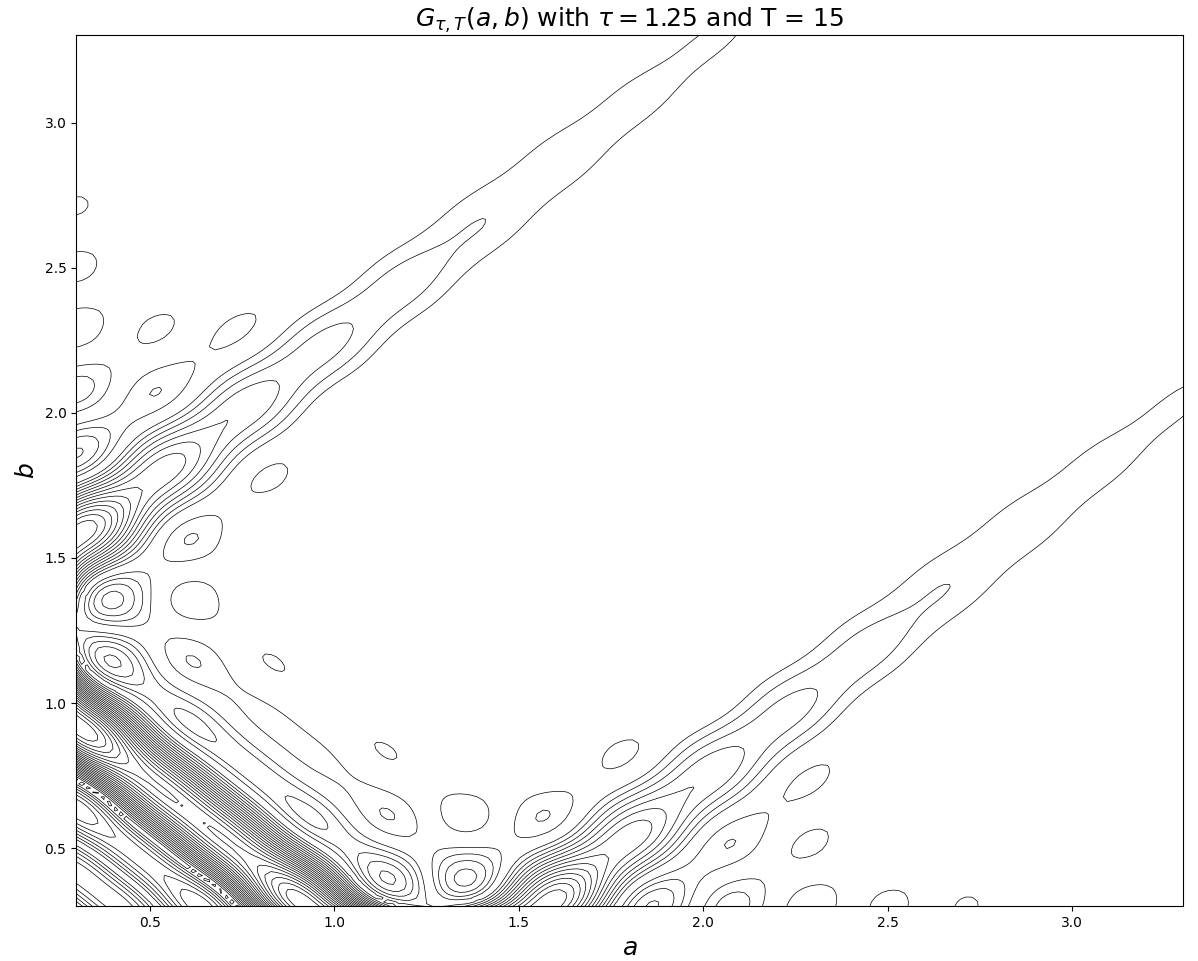}
    \caption{Contour plot of $G_{\tau,T}(a,b)$. Two pronounced ``ridges'' are visible along the lines $b = a + \tau$ and $b = a - \tau$. The ridges become more pronounced as $ T $ increases and decay as $ a $ and $ b $ increase.}
    \label{Fig: Spectral}
\end{figure}

Using the definitions of $h_{t}$ and $h_{t,\tau}$, we expand $G_{\tau, T}(a,b)$ to
\begin{align}
	&
	\label{Expr: Spectral integral}
	 \frac{1}{T a b} \left| \int_{0}^{T} \cos( t \tau ) \sin( t a ) \sin( t b ) dt \right|.
\end{align}
Applying trigonometric sum-to-product identities, we have
\begin{align*}
	&
	\cos( t \tau ) \sin( t a ) = \frac{ \sin\left( t( \tau + a ) \right) - \sin \left( t (  \tau - a ) \right) }{2},
\end{align*}
and hence
\begin{align*}
	&
	\cos( t \tau ) \sin( t a ) \sin( t b ) =
	\frac{  \cos(t ( a - b - \tau ) ) + \cos( t( a - b + \tau ) ) - \cos( t( a + b - \tau ) ) - \cos( t( a + b + \tau ) ) }{4}.
\end{align*}
Therefore, Expression \eqref{Expr: Spectral integral} equals
\begin{align*}
	&
	\frac{1}{4 T a b} \left| \int_{0}^{T} \cos(t ( a - b - \tau ) ) + \cos( t( a - b + \tau ) ) - \cos( t( a + b - \tau ) ) - \cos( t( a + b + \tau ) ) \, dt \right|.
\end{align*}
Integrating each cosine term separately gives
\begin{align*}
	&
	\frac{1}{4 T a b} \left| \frac{ \sin( T( a - b - \tau ) ) }{ a - b - \tau } + \frac{ \sin( T( a - b + \tau ) ) }{a - b + \tau} - \frac{ \sin( T( a + b - \tau ) ) }{a + b - \tau} - \frac{ \sin( T( a + b + \tau ) ) }{a + b + \tau} \right|.
\end{align*}
Applying the reverse triangle inequality yields a lower bound:
\begin{align}
	\label{Expr: Reverse triangle}
\frac{1}{4Tab} \bigg(
\left| \frac{\sin(T(a - b - \tau))}{a - b - \tau}
     + \frac{\sin(T(a - b + \tau))}{a - b + \tau} \right|
-
\left| \frac{\sin(T(a + b - \tau))}{a + b - \tau} \right|
-
\left| \frac{\sin(T(a + b + \tau))}{a + b + \tau} \right|
\bigg).
\end{align}

We now bound the absolute value terms separately. Consider the first absolute value. The function
$$x \mapsto \frac{ \sin( T \, x ) }{x}$$ 
is positive when $ x \in  \left(- \frac{\pi}{T} , \frac{\pi}{T} \right)$, and decreases as $|x|$ grows in this interval. Since $T\delta = \frac{\pi}{2}$ and $a - b - \tau \in (-\delta,\delta)$, we have
$$
\frac{\sin\!\big( T(a - b - \tau) \big)}{a - b - \tau}
\geq 
\frac{\sin(T\delta)}{\delta}
= \frac{1}{\delta}.
$$
Moreover, for all $ x \in \mathbb{R} $:
$$
\frac{\sin(Tx)}{x} \ge - \frac{1}{ \pi/T } = -\frac{1}{2\delta}.
$$
Thus,
\begin{align*}
	&
	\frac{\sin(T(a - b - \tau))}{a - b - \tau}
     + \frac{\sin(T(a - b + \tau))}{a - b + \tau} > \frac{1}{2 \delta}.
\end{align*}

Next, consider the second and third absolute value terms in Expression \eqref{Expr: Reverse triangle}. Since
\begin{align*}
	&
	\left| \frac{ \sin( Tx ) }{x} \right| \leq \frac{1}{| x |},
\end{align*}
we have
\begin{align*}
	&
	\left|\frac{ \sin( T( a + b - \tau ) ) }{ a + b - \tau } \right| \leq \frac{1}{ | a + b - \tau | }.
\end{align*}
Moreover,
\begin{align*}
	&
	| a + b - \tau | \geq 2b - | a - b - \tau | > 2m - \delta > 8 \delta,
	\\
	&
	| a + b + \tau | \geq 2a - | a - b - \tau | > 2m - \delta > 8 \delta.
\end{align*}
Hence,
\begin{align*}
	&
	\left|\frac{ \sin( T( a + b - \tau ) ) }{ a + b - \tau } \right| < \frac{1}{8 \delta}, \qquad
	\left|\frac{ \sin( T( a + b + \tau ) ) }{ a + b + \tau } \right| < \frac{1}{8 \delta},
\end{align*}
so that
\begin{align*}
	&
	-
\left| \frac{\sin(T(a + b - \tau))}{a + b - \tau} \right|
-
\left| \frac{\sin(T(a + b + \tau))}{a + b + \tau} \right| > - \frac{1}{4 \delta}.
\end{align*}

Combining these bounds, Expression \eqref{Expr: Reverse triangle} has the lower bound
\begin{align*}
\frac{1}{4Tab}\,\frac{1}{4\delta}
= \frac{1}{16\frac{\pi}{2}ab}
= \frac{1}{8\pi ab},
\end{align*}
where we used \(T\delta=\frac{\pi}{2}\). Hence,
$$
\left| \frac{1}{T}\int_0^T h_{t,\tau}(a)\, h_t(b)\, dt \right|
 > 
\frac{1}{8\pi ab},
$$
and therefore
$$
\left| \frac{1}{T}\int_0^T h_{t,\tau}(a)\, h_t(b)\, dt \right|^{-1} = G_{\tau,T}^{-1} < 8 \pi a b.
$$
\end{proof}


\section{Propagators}
\label{Sec: Propagators}

\noindent
In order to estimate the Hilbert–Schmidt norm
$$
\left\| \frac{1}{T} \int_0^T P_{t,\tau}^{*} a P_t \, dt \right\|_{HS}^2,
$$
the geometric data of the problem that appears in the proof of Theorem \ref{Thm: Quantitative main theorem}, it is necessary to analyze the kernels of the hyperbolic wave propagator $P_t$ and a slightly modified version of it, $P_{t,\tau}$, since these kernels are then used to bound the Hilbert–Schmidt norms. Although the hyperbolic wave propagator is a natural choice, it exhibits greater technical complexity than, for instance, the ball averaging operator (cf. \cite{LS17}); the choice is deliberate as it makes the spectral side easier. We state a proposition giving a convenient form for the integral kernels of $P_t$ and $P_{t,\tau}$.

\begin{proposition}[Propagators]
\label{Prp: Propagators}
Let
\[
P_t = h_t\Big( \sqrt{-\Delta_X - \tfrac14} \Big), \quad
P_{t,\tau} = h_{t,\tau}\Big( \sqrt{-\Delta_X - \tfrac14} \Big),
\]
be defined on a compact connected hyperbolic surface $X = \Gamma \setminus \mathbb{H}$, with $-\Delta_{X}$ being the Laplace–Beltrami operator, and where
\[
h_t(x) = \frac{\sin(tx)}{x}, \quad h_{t,\tau}(x) = \cos(t\tau) \, h_t(x).
\]
Then $P_t$ admits the integral kernel $K_{t}^{\Gamma}$ defined as
\[
K_t^\Gamma(x,y) = \frac{1}{2\sqrt{2} \pi} \sum_{\gamma \in \Gamma} K_t(x, \gamma y), \quad
K_t(x,y) = \frac{\mathbf{1}_{t > d(x,y)}}{\sqrt{\cosh(t) - \cosh(d(x,y))}},
\]
and $P_{t,\tau}$ admits the integral kernel $\cos(t\tau) \, K_t^\Gamma$.
\end{proposition}

Given that $\tilde{P}_{t}$ is the integral operator with integral kernel $K_{t}^{\Gamma}$, the proposition can be proved by showing that $\langle P_{t} f, g \rangle = \langle \tilde{P}_{t} f, g \rangle$ for $f ,g \in C^{\infty}( X )$, since $P_{t}$ and $\tilde{P}_{t}$ are bounded operators on $L^{2}( X )$. The core idea of showing the equality of the inner products is representing $P_{t}$ as a limit of propagators $P_{t}^{(\varepsilon)}$ that have convenient explicit formulae for their integral kernels. The proof relies on multiple auxiliary results presented after the proof of Proposition \ref{Prp: Propagators}.

\begin{proof}[Proof of Proposition \ref{Prp: Propagators}]

The assertion for $P_{t, \tau}$ follows directly from the assertion for $P_t$ since $P_{t,\tau} = \cos(t \tau) P_{t}$. Hence, it suffices to consider $P_t$ and $K_{t}^{\Gamma}$.

Due to the way $P_{t}$ is defined, it is bounded and has the property
\begin{align*}
	P_{t} \psi_{j} = h_{t}(\rho_{j}) \psi_{j},
\end{align*}
where $\psi_{j}$ is an eigenfunction of $-\Delta_{X}$ with eigenvalue $\lambda_{j}$ such that $\rho_{j} = \sqrt{\lambda_{j} - \frac{1}{4}}$. Define regularized operators for $\varepsilon > 0$:
\[
P_t^{(\varepsilon)} = h_t^{(\varepsilon)}\Big(\sqrt{-\Delta_{X} - \tfrac14}\Big), \quad
h_t^{(\varepsilon)}(x) = \frac{\sin(tx)}{x} \, e^{- \frac{\varepsilon^2 x^2}{2}}.
\]
Due to the way $P_{t}^{(\varepsilon)}$ is defined, it is bounded and has the following property:
$$P_t^{(\varepsilon)} \psi_j = h_t^{(\varepsilon)}(\rho_j) \, \psi_j.$$ 
Since $h_{t}^{(\varepsilon)}$ converges to $h_{t}$ pointwise as $\varepsilon \to 0$, we have that $P_{t}^{(\varepsilon)} f$ converges to $P_{t} f$ pointwise for every $f \in C^{\infty}( X )$ due to spectral theory.

Proposition~\ref{Prp: Marklof} ensures the existence of an integral operator $Q_t^{(\varepsilon)}$ on $X$ with integral kernel $K_{t}^{(\varepsilon)}$ defined as
\begin{align*}
	&
	K_t^{(\varepsilon)}(x,y) = \sum_{\gamma \in \Gamma} k_t^{(\varepsilon)}(x, \gamma y), 
\\
	&
k_t^{(\varepsilon)}(x,y) = - \frac{1}{\pi i} \frac{1}{2\sqrt{2}\pi} \int_{-\infty}^{\infty} \Bigg( \int_{d(x,y)}^\infty \frac{e^{-i\rho v}}{\sqrt{\cosh(v) - \cosh(d(x,y))}} \, dv \Bigg) \rho h^{(\varepsilon)}(\rho) \, d\rho,
\end{align*}
with the property: 
$$Q_t^{(\varepsilon)} \psi_j = h_t^{(\varepsilon)}(\rho_j) \psi_j.$$  

Lemma~\ref{Lmm: Understanding Green's} allows us to rewrite $k_t^{(\varepsilon)}$ in a more tractable form:
\[
k_t^{(\varepsilon)}(x,y) = \frac{1}{2\sqrt{2} \pi} \int_{d(x,y)}^\infty \frac{\phi_\varepsilon(t-v) - \phi_\varepsilon(-t-v)}{\sqrt{\cosh(v) - \cosh(d(x,y))}} \, dv.
\]
Here 
$$\phi_\varepsilon(x) = \frac{1}{\sqrt{2\pi \varepsilon^{2}}} e^{- \frac{x^2}{2\varepsilon^2}}$$ 
is the standard Gaussian mollifier. Since a bounded operator is uniquely determined by its action on an orthonormal basis, it follows that $Q_t^{(\varepsilon)} = P_t^{(\varepsilon)}$.

Let $\tilde{P}_t$ denote the integral operator on $X$ with kernel $K_t^\Gamma$. To show that $\tilde{P}_{t}$ is well-defined and bounded, we use the Schur test (see \cite{HS12}). It suffices to verify that for all $x \in X$,
\begin{align*}
	\int_{X} K_{t}^{\Gamma}(x,y) \, dy < \infty,
\end{align*}
since $K_{t}^{\Gamma}(x,y)$ is symmetric in $x$ and $y$. We open the integral using the definition of $K_{t}^{\Gamma}$:
\begin{align*}
	\frac{1}{2 \sqrt{2} \pi } \, \int_{X} \sum_{\gamma \in \Gamma} \frac{ \mathbf{1}_{t > d(x, \gamma y)} }{ \sqrt{ \cosh(t) - \cosh(d(x, \gamma y)) } } dy.
\end{align*}
Uniting the integral and the sum into an integral over the universal cover gives:
\begin{align*}
	\frac{1}{2 \sqrt{2} \pi} \int_{\mathbb{H}} \frac{ \mathbf{1}_{t > d(x, y)} }{ \sqrt{ \cosh(t) - \cosh(d(x,y)) } } dy.
\end{align*}
Writing this using polar coordinates around $x$ yields:
\begin{align*}
	\frac{2\pi}{2 \sqrt{2} \pi} \int_{0}^{t} \frac{\sinh(\rho)}{\sqrt{\cosh(t) - \cosh(\rho)}} d\rho = \frac{4\pi}{2 \sqrt{2} \pi} \sqrt{\cosh(t) - 1},
\end{align*}
where the coefficient $2\pi$ arises from the polar coordinates, and the equality follows from computing the integral. The result is finite, and $\tilde{P}_{t}$ is well-defined and bounded.

To show that $P_{t}$ has integral kernel $K_{t}^{\Gamma}$, as claimed in the proposition, it suffices to verify that for every $f, g \in C^{\infty}(X)$,
\begin{align*}
	\left| \langle P_{t} f, g \rangle - \langle \tilde{P}_{t} f, g \rangle \right| = 0,
\end{align*}
since $P_{t}$ and $\tilde{P}_{t}$ are bounded operators on $L^{2}(X)$ where $X$ is compact. To show this, consider the absolute value:
\[
\left| \int_X g(x)^* \big( P_t f(x) - \tilde{P}_t f(x) \big) dx \right|.
\]
By moving the absolute value inside the integral and bounding \( |g| \) by its supremum norm, we obtain the upper bound:
\begin{align*}
	\| g \|_{\infty} \int_{X} \left| P_{t} f(x) - \tilde{P}_{t} f(x) \right| dx.
\end{align*}

We focus on the absolute value inside the previous integral. Writing \(P_t f\) using the limit and expressing the operator action via kernels, we obtain
\begin{align*}
	\left| \lim_{\varepsilon \to 0} \int_{X} K_{t}^{(\varepsilon)}(x,y) f(y) - K_{t}^{\Gamma}(x,y) f(y) \, dy \right|.
\end{align*}
Expressing the kernels using the sum over \( \Gamma \) and combining the sum with the integral over \( X \) into an integral over the hyperbolic plane yields the following equivalent expression:
\begin{align*}
	\left| \lim_{\varepsilon \to 0} \int_{\mathbb{H}} k_{t}^{(\varepsilon)}(x,y) f(y) - K_{t}(x,y) f(y) \, dy \right|.
\end{align*}
Here, $f$ is viewed as a $\Gamma$-periodic function on $\mathbb{H}$. By moving the absolute value inside and bounding \( |f| \) by its supremum norm, we obtain the upper bound:
\begin{align*}
	\| f \|_{\infty} \lim_{\varepsilon \to 0} \int_{\mathbb{H}} \left| k_{t}^{(\varepsilon)}(x,y) - K_{t}(x,y) \right| dy.
\end{align*}

Focus on the limit of the integral. By opening $k_t$ and $K_t$, the limit can be written as
\begin{align*}
	\lim_{\varepsilon \to 0} \frac{1}{2 \sqrt{2} \pi} \int_{\mathbb{H}} \left| \left( \int_{\mathbb{R}} \frac{ (\phi_\varepsilon(t-v) - \phi_\varepsilon(-t-v)) \mathbf{1}_{v > d(x,y)}}{\sqrt{\cosh(v) - \cosh(d(x,y))}} dv \right) - \frac{\mathbf{1}_{t > d(x,y)}}{\sqrt{\cosh(t) - \cosh(d(x,y))}} \right| dy.
\end{align*}
Define
$$
\chi_{d(x,y)}(t) = \frac{\mathbf{1}_{t > d(x,y)}}{\sqrt{\cosh(t) - \cosh(d(x,y))}}.
$$
Using this, the previous integral expression becomes
$$
\frac{1}{2 \sqrt{2} \pi} \lim_{\varepsilon \to 0} \int_{\mathbb{H}} \left| \phi_\varepsilon \ast_\mathbb{R} \chi_{d(x,y)}(t) - \phi_\varepsilon \ast_\mathbb{R} \chi_{d(x,y)}(-t) - \chi_{d(x,y)}(t) + \chi_{d(x,y)}(-t) \right| dy,
$$
where $\ast_{\mathbb{R}}$ denotes convolution on $\mathbb{R}$. The term $\chi_{d(x,y)}(-t)$ equals $0$, so its addition has no effect. Using the triangle inequality, the expression is bounded by
\begin{align*}
	\frac{1}{2 \sqrt{2} \pi} \lim_{\varepsilon \to 0} \int_{\mathbb{H}} \left| \phi_\varepsilon \ast_\mathbb{R} \chi_{d(x,y)}(t) - \chi_{d(x,y)}(t) \right| + \left| \phi_\varepsilon \ast_\mathbb{R} \chi_{d(x,y)}(-t) - \chi_{d(x,y)}(-t) \right| dy.
\end{align*}
Finally, Lemma~\ref{Lmm: Hard convolution} guarantees that $y \mapsto \phi_\varepsilon \ast_{\mathbb{R}} \chi_{d(x,y)}(t)$ converges to $y \mapsto \chi_{d(x,y)}(t)$ in $L^1(\mathbb{H})$. Hence, the limit equals $0$, and
\begin{align*}
	\left| \langle P_{t} f, g \rangle - \langle \tilde{P}_{t} f, g \rangle \right| = 0,
\end{align*}
thus concluding the proof.
\end{proof}

The remaining part of this section contains lemmas and propositions used in the proof of Proposition \ref{Prp: Propagators}. First, we record a slightly modified formulation of \cite[Prop. 5]{Mar12}. It gives a family of operators with explicit formulae for their integral kernels on a compact connected hyperbolic surface. These operators share the eigenfunctions with the Laplace–Beltrami operator of the surface, and their eigenvalues have a representation in terms of the Laplace–Beltrami eigenvalues.

\begin{proposition}
\label{Prp: Marklof}
Let $h: \{ x+iy \in \mathbb{C} \ : \ |y| < C \} \to \mathbb{C}$ satisfy the following properties, where $C$ is a positive constant:
\begin{itemize}
	\item $h$ is analytic,
	\item $h$ is even,
	\item $h(x+iy)$ decays superpolynomially as $|x| \to \infty$.
\end{itemize}
Then the integral operator $P$ on a compact connected hyperbolic surface $X = \Gamma \backslash \mathbb{H}$ with integral kernel
$$
K^\Gamma(x,y) = \sum_{\gamma \in \Gamma} K(x, \gamma y),
$$
where
$$
K(x,y) = - \frac{1}{\pi i} \frac{1}{2 \sqrt{2} \pi} \int_{-\infty}^\infty \left( \int_{d(x,y)}^\infty \frac{e^{-i \rho v}}{\sqrt{\cosh(v) - \cosh(d(x,y))}} \, dv \right) \rho h(\rho) \, d\rho,
$$
satisfies
\[
P \psi_j = h(\rho_j) \psi_j,
\]
where $\psi_j$ is a Laplacian eigenfunction on $X$ with eigenvalue $\lambda_j$ and $\rho_j = \sqrt{\lambda_j - \tfrac{1}{4}}$.
\end{proposition}

In the next lemma, we derive a convenient expression for $k_t^{(\varepsilon)}$, which is an integral kernel crucial in establishing a tractable formula for the kernel of the integral operator $P_{t}$. More specifically, $k_{t}^{(\varepsilon)}$ is used to approximate the integral kernel of $P_{t}$. We state the lemma.

\begin{lemma}
\label{Lmm: Understanding Green's}
Let
\[
k_t^{(\varepsilon)}(x,y) = \frac{1}{\pi i} \int_{-\infty}^{\infty} \left( -\frac{1}{2\pi \sqrt{2}} \int_{d(x,y)}^\infty \frac{e^{-i \rho v}}{\sqrt{\cosh(v) - \cosh(d(x,y))}} \, dv \right) \rho h_t^{(\varepsilon)}(\rho) \, d\rho,
\]
where $x,y \in \mathbb{H}$ and
$$h_t^{(\varepsilon)}(\rho) = \frac{\sin(t\rho)}{\rho} \, e^{-\frac{\varepsilon^2 \rho^2}{2}}.$$  
Then
\[
k_t^{(\varepsilon)}(x,y) = \frac{1}{2\sqrt{2} \pi} \int_{d(x,y)}^\infty \frac{\phi_\varepsilon(t-v) - \phi_\varepsilon(-t-v)}{\sqrt{\cosh(v) - \cosh(d(x,y))}} \, dv,
\]
where 
$$\phi_\varepsilon(x) = \frac{1}{\sqrt{2\pi \varepsilon^{2}}} e^{-\frac{x^2}{2\varepsilon^2}}$$
is the standard Gaussian mollifier.
\end{lemma}

The proof is based on noting that there is a Fourier transform inside the expression used to define $k_{t}^{(\varepsilon)}$. By using the inverse Fourier transform of $\phi_{\varepsilon}$ and the Fourier transform of $s_{t}(x) = x^{-1} \, \sin(t x)$ (in the distributional sense) together with the Fourier convolution identity, we obtain the desired conclusion.

\begin{proof}[Proof of Lemma \ref{Lmm: Understanding Green's}]
By applying Fubini's theorem to interchange the order of integration, we obtain that $k_{t}^{(\varepsilon)}$ equals
\begin{equation}
\label{Expr: Understanding Green's}
- \frac{1}{\pi i} \frac{1}{2\sqrt{2} \pi} \int_{d(x,y)}^\infty \frac{1}{\sqrt{\cosh(v) - \cosh(d(x,y))}} \left( \int_{-\infty}^{\infty} e^{-i \rho v} \rho h_t^{(\varepsilon)}(\rho) \, d\rho \right) dv.
\end{equation}
The inner integral is the Fourier transform of $\rho \mapsto \rho h_t^{(\varepsilon)}(\rho)$, where the Fourier transform is
$$
\mathcal{F}(f)(x) = \int_{\mathbb{R}} f(y) e^{-i x y} \, dy
$$
and the inverse Fourier transform is
\begin{align*}
	\mathcal{F}^{-1}f(x) = \frac{1}{2 \pi} \int_{\mathbb{R}} f(y) \, e^{ i x y } dy.
\end{align*}
We observe that
$$
\rho \, h_t^{(\varepsilon)}(\rho) = \sin(t \rho) \, e^{-\frac{\varepsilon^2 \rho^2}{2}} = 2\pi \, \sin(t \rho) \, \mathcal{F}^{-1}(\phi_\varepsilon)(\rho),
$$
since
\begin{align*}
	\mathcal{F}^{-1}( \phi_{\varepsilon} )(x) = \frac{1}{2\pi} \, e^{ -\frac{\varepsilon^2 x^2}{2} }.
\end{align*}

The inner integral in Expression \eqref{Expr: Understanding Green's} equals
$$
\mathcal{F}\left( \rho \mapsto 2 \pi \, \sin( t \rho ) \, \mathcal{F}^{-1}(\phi_{\varepsilon}) \right)(v) = \frac{2\pi}{2\pi} \left( \mathcal{F}(s_{t}) \ast \phi_{\varepsilon} \right)(v) = \pi i \left( \phi_{\varepsilon}(v + t) - \phi_{\varepsilon}(v - t) \right),
$$
where the first equality follows from the convolution identity, and the second equality uses
\begin{align*}
\mathcal{F}(s_t)(x) = \pi i \bigl(\delta(x + t) - \delta(x - t)\bigr),
\end{align*}
with \(s_t(x) = \sin(tx)\) and \(\delta\) denoting the Dirac distribution.

Substituting this into Expression \eqref{Expr: Understanding Green's} gives:
\begin{align*}
	k_t^{(\varepsilon)}(x,y) =
	\frac{1}{2 \sqrt{2} \pi} \int_{d(x,y)}^{\infty} \frac{ \phi_{\varepsilon}(v - t) - \phi_{\varepsilon}(v + t) }{ \sqrt{ \cosh(v) - \cosh(d(x,y)) } } \, dv
	=
	\frac{1}{2 \sqrt{2} \pi} \int_{d(x,y)}^{\infty} \frac{ \phi_{\varepsilon}(t - v) - \phi_{\varepsilon}(t+v) }{ \sqrt{ \cosh(v) - \cosh(d(x,y)) } } \, dv,
\end{align*}
where the equality follows from the symmetry of the Gaussian mollifier. 
\end{proof}

Finally, the following lemma establishes the convergence of $y \mapsto \phi_{\varepsilon} \ast_{\mathbb{R}} \chi_{d(x,y)}(t)$ to $y \mapsto \chi_{d(x,y)}(t)$ in $L^1(\mathbb{H})$. Interestingly, the convolution is over the real line, but the convergence is studied in the hyperbolic plane.

\begin{lemma}
\label{Lmm: Hard convolution}
Let
$$
\chi_{d(x,y)}(t) = \frac{\mathbf{1}_{t > d(x,y)}}{\sqrt{\cosh(t) - \cosh(d(x,y))}}
$$
and
\begin{align*}
	\phi_\varepsilon(x) = \frac{1}{\sqrt{2\pi \varepsilon^{2}}} e^{-\frac{x^2}{2\varepsilon^2}}.
\end{align*}
Then the map $y \mapsto \phi_\varepsilon \ast_\mathbb{R} \chi_{d(x,y)}(t)$ converges to $y \mapsto \chi_{d(x,y)}(t)$ in $L^1(\mathbb{H})$ for every $t \in \mathbb{R}$ as $\varepsilon \to 0$, where $\ast_\mathbb{R}$ denotes convolution on $\mathbb{R}$.
\end{lemma}

In the proof, we show that
\begin{align*}
	\int_\mathbb{H} \left| \phi_\varepsilon \ast_\mathbb{R} \chi_{d(x,y)}(t) - \chi_{d(x,y)}(t) \right| \, dy \xrightarrow{\varepsilon \to 0} 0.
\end{align*}
The idea is to present this integral as a convolution $\phi_{\varepsilon} \ast_{\mathbb{R}} A_{t}$, where $A_{t}$ is a continuous function with $A_{t}(t) = 0$. Then the conclusion follows from the properties of convolution.

\begin{proof}[Proof of Lemma \ref{Lmm: Hard convolution}]
We want to show that the following expression goes to $0$ as $\varepsilon \to 0$:
\begin{align}
	\label{Expr: Convolution}
	\int_\mathbb{H} \left| \phi_\varepsilon \ast_\mathbb{R} \chi_{d(x,y)}(t) - \chi_{d(x,y)}(t) \right| \, dy
	= \int_\mathbb{H} \left| \int_\mathbb{R} \phi_\varepsilon(t-v) \big( \chi_{d(x,y)}(v) - \chi_{d(x,y)}(t) \big) \, dv \right| dy.
\end{align}
Here, we used that the integral over the Gaussian mollifier equals $1$:
\begin{align*}
	\chi_{d(x,y)}(t) = \int_\mathbb{R} \phi_{\varepsilon}(t - v) \, \chi_{d(x,y)}(t) \, dv.
\end{align*}
By the triangle inequality and Fubini's theorem, Expression \eqref{Expr: Convolution} is bounded by
\[
\int_\mathbb{R} \phi_\varepsilon(t-v) \int_\mathbb{H} \left| \chi_{d(x,y)}(v) - \chi_{d(x,y)}(t) \right| \, dy \, dv = \phi_\varepsilon \ast_\mathbb{R} A_{t}(t),
\]
where 
$$A_{t}(v) = \int_\mathbb{H} |\chi_{d(x,y)}(v) - \chi_{d(x,y)}(t)| \, dy.$$
We see that $A_{t}(t) = 0$, and we want to show that $\lim_{v \to t} A_{t}(v) = 0$. By the triangle inequality:
\[
A_t(v) \le \int_\mathbb{H} \frac{\mathbf{1}_{t > d(x,y)}}{\sqrt{\cosh(t) - \cosh(d(x,y))}} \, dy + \int_\mathbb{H} \frac{\mathbf{1}_{v > d(x,y)}}{\sqrt{\cosh(v) - \cosh(d(x,y))}} \, dy < \infty.
\]
The finiteness can be verified using polar coordinates around $x$. Hence, $y \mapsto |\chi_{d(x,y)}(v) - \chi_{d(x,y)}(t)|$ has a dominating function when $v$ is in a compact interval around $t$. By the dominated convergence theorem, $A_{t}(v) \to 0$ as $v \to t$. Since $A_{t}(t) = 0$ and $A_{t}$ is continuous at $t$, the properties of convolution yield
$$
\lim_{\varepsilon \to 0} \phi_\varepsilon \ast_\mathbb{R} A_{t}(t) = 0,
$$
establishing the claim.
\end{proof}


\section{Geometric data}
\label{Sec: Geometric data}

\noindent
In this section, we bound the Hilbert-Schmidt norm 
\[
\left\| \frac{1}{T} \int_0^T P_{t,\tau}^{*} a P_t \, dt \right\|_{HS}^{2} ,
\]
which is the geometric data appearing in the proof of Theorem \ref{Thm: Quantitative main theorem} in Section \ref{Sec: Main theorems}. $ P_{t} $ and $ P_{t,\tau} $ are the integral operators we extensively studied in the previous section, Section \ref{Sec: Propagators}. We now state the proposition.

\begin{proposition}[Geometric data]
\label{Prp: Geometric data}

Let $P_t$ and $P_{t,\tau}$ be integral operators on a compact connected hyperbolic surface $X = \Gamma \setminus \mathbb{H}$ with Laplace--Beltrami operator $ -\Delta_{X} $ such that $ P_{t} $ has integral kernel $ K_{t}^{\Gamma} $, defined as
\[
K_t^\Gamma(x,y) = \frac{1}{2 \sqrt{2} \pi} \sum_{\gamma \in \Gamma} K_t(x, \gamma y), \quad
K_t(x,y) = \frac{\mathbf{1}_{t > d(x,y)}}{\sqrt{\cosh(t) - \cosh(d(x,y))}},
\]
and $ P_{t,\tau} $ has integral kernel $K_{t,\tau}^{\Gamma} = \cos(t \tau) \, K_t^\Gamma$. Let $ a $ be a function on $ X $ with mean zero. Then
\[
	\left\| \frac{1}{T} \int_0^T P_{t,\tau}^{*} a P_t \, dt \right\|_{HS}^{2} \leq \frac{C_{7}}{T \, \beta^{3}} \, \left( \| a \|_{2}^{2} \, + \,  \| a \|_{\infty}^{2} \, | X \setminus X( 4T ) | \frac{e^{6 T}}{ \operatorname{InjRad}_{X} } \right)
\]
where $ \beta = \beta( \lambda_{1}( X ) ) $,  $|X \setminus X( 4T )|$  is the volume of the points of $ X $ having injectivity radius less than $ 4T $, $ \lambda_{1}( X ) $ is the spectral gap of $ -\Delta_{X} $, and where
\begin{align*}
	\beta( x ) = 
	\begin{cases}	
		1 - \sqrt{ 1 - 4x }, \qquad & \text{ if } x \leq \frac{1}{4},
		\\
		1, \qquad & \text{ if } x > \frac{1}{4}.
	\end{cases}
\end{align*}
The constant appearing is $C_{7} = 80000$.

\end{proposition}

By careful bounding and rewriting the Hilbert-Schmidt norm expression, we get a tidy expression in which the observable $ a $ is separated from the geometric data arising from the choice of the propagator; this geometric data is contained in a certain weight function $ F_{t,t',\rho} $. This weight is studied in Section \ref{Sec: Weight function}. The bound is acquired by separating an integration over the fundamental domain into a part dealing with the points with large injectivity radius and a part dealing with the points with small injectivity radius. In the large injectivity radius case, the exponential mixing theorem, Theorem \ref{Thm: Exponential mixing}, gives a crucial decay needed to acquire the bounds. The final bounds are obtained by using bounds for $ F_{t,t',\rho} $. 

\begin{proof}[Proof of Proposition \ref{Prp: Geometric data}]
	
Writing the Hilbert-Schmidt norm as a double integral over the fundamental domain $ D $ of $ X $ yields the following expression:
\begin{align*}
	&
	\int_{D} \int_{D} \left|  \frac{1}{T} \int_{0}^{T} \int_{D} K_{t,\tau}^{\Gamma} a( z ) K_{t}^{\Gamma}( z,y ) dt \right|^{2} dy dx.
\end{align*}
We also used that $ P_{t,\tau} $ is self-adjoint. Using that $ K^{\Gamma}_{t,\tau} = \cos(t \tau) K_t^\Gamma $ and by expanding the squares, we can write the previous expression as follows:
\begin{align*}
	&
	\int_{D \times D} \frac{1}{T^{2}} \int_{0}^{T} \int_{0}^{T} \cos( t \tau ) \cos( t' \tau ) \int_{D \times D} K_{t}^{\Gamma}( x,z ) a( z ) K_{t}^{\Gamma}( z,y ) \, K_{t'}^{\Gamma}( x,z' ) a( z' )^{*} K_{t'}^{\Gamma}( z', y )  dz' dz dt' dt dy dx.
\end{align*}
Using Fubini's theorem to reorder the integrals yields the following equal expression:
\begin{align*}
	&
	\frac{1}{T^{2}} \int_{0}^{T} \int_{0}^{T} \cos( t \tau ) \cos( t' \tau ) \int_{D \times D \times D \times D } K_{t}^{\Gamma}( x,z ) a( z ) K_{t}^{\Gamma}( z,y ) K_{t'}^{\Gamma}( x,z' ) a( z' )^{*} K_{t'}^{\Gamma}( z', y ) dx dy dz' dz dt' dt. 	
\end{align*}
This expression is bounded by adding the absolute values around it. Then we pull the absolute values inside to get the following upper bound:
\begin{align*}
	&
	\frac{1}{T^{2}} \int_{0}^{T} \int_{0}^{T} \left| \int_{D} \int_{D} \int_{D}^{} \int_{D} K_{t}^{\Gamma}( x,z ) a( z ) K^{\Gamma}_{t}( z,y ) \, K_{t'}^{\Gamma}( x,z' ) a( z' )^{*} K_{t'}^{\Gamma}( z', y )  dy dx dz' dz \right| dt' dt.
\end{align*}
The cosines were bounded from above by $ 1 $ in the previous step. The previous expression can be written in a more compact form when gathering the factors containing $ x $ into one integral and the factors containing $ y $ into another integral and noticing that these two integrals are the same, thus resulting in an integral squared:
\begin{align*}
	&
	\frac{1}{T^{2}} \int_{0}^{T} \int_{0}^{T} \left| \int_{D} \int_{D} a( z ) a( z' )^{*} \left( \int_{D} K_{t}^{\Gamma}( x,z ) K_{t'}^{\Gamma}( x,z' ) dx \right)^{2} dz' dz \right| dt' dt.
\end{align*}

Lemma \ref{Lmm: Eliminate sum 1} allows us to write the previous expression using the hyperbolic plane kernels $ K_{t} $ and $ K_{t'} $ instead of the surface kernels $ K_{t}^{\Gamma} $ and $ K_{t'}^{\Gamma} $: 
\begin{align*}
	&
	\frac{1}{8 \pi^{2} T^{2}} \int_{0}^{T} \int_{0}^{T} \left| \int_{D} \int_{D} a( z ) a( z' )^{*} \left( \int_{\mathbb{H}} \sum_{\gamma \in \Gamma}^{} K_{t}( x,z ) K_{t'}( x, \gamma z' ) dx \right)^{2} dz' dz \right| dt' dt,
\end{align*}
where the coefficient $ ( 8 \, \pi^{2} )^{-1} $ comes from the definition of $ K_{t}^{\Gamma} $:
\begin{align*}
	&
	K_{t}^{\Gamma}( x,y ) = \frac{1}{2 \sqrt{2 }\pi} \sum_{\gamma \in \Gamma}^{} K_{t}( x, \gamma y ).
\end{align*}

Lemma \ref{Lmm: Eliminate sum 2} allows us to rewrite the integral inside the absolute value in a convenient form, eliminating all but one sum. Thus the previous expression equals:
\begin{align}
	&
	\label{Expr: Introducing F}
	\frac{1}{ 8 \pi^{2} T^{2}} \int_{0}^{T} \int_{0}^{T} \left| \int_{D} a( z ) \int_{\mathbb{H}} a( z' )^{*}   F_{t,t', d( z,z' )}  \sum_{\gamma \in \Gamma}^{} F_{t,t', d( \gamma z, z' )} dz' dz \right| dt' dt,
\end{align}
where $ F_{t,t', \rho} $ is defined as
\begin{align*}
\int_{\mathbb{H}} K_{t}( x,w ) K_{t'}( x,w' ) dx,
\end{align*}
with $ w, w' \in \mathbb{H} $ separated by distance $ \rho $; Lemma \ref{Lmm: Points do not matter} ensures that $ F_{t,t',\rho} $ is well-defined. 

The definition of $ F_{t,t', d( z,z' )} $ shows that $ F_{t,t',d( z,z' )} = 0 $ if $ d( z,z' ) > 2T $, which follows from the definitions of $ K_{t} $ and $ K_{t'} $. Using this, Expression \eqref{Expr: Introducing F} can be rewritten by restricting the integration domain of \(z'\) as follows:
\begin{align}
	&
	\label{Expr: Before large small InjRad split}
	\frac{1}{8 \pi^{2} T^{2}} \int_{0}^{T} \int_{0}^{T} \left| \int_{D} a( z ) \int_{ B( z, 2T ) } a( z' )^{*}   F_{t,t', d( z,z' )}  \sum_{\gamma \in \Gamma}^{} F_{t,t', d( \gamma z, z' )} dz' dz \right| dt' dt.
\end{align}
We split the integral over $ D $ into points with injectivity radius greater than $ 4T $, $ D( 4T ) $, and points with injectivity radius less than $ 4T $, $ D \setminus D( 4T ) $:
\begin{align*}
	&
	\left| \int_{D} \cdots \right| \leq
	\left| \int_{D( 4T )} \cdots \right| 
	+  \left| \int_{D \setminus D( 4T )} \cdots \right|.
\end{align*}
By moving the absolute value inside the latter term and bounding \( |a| \) by its supremum norm, the previous expression yields the following upper bound:
\begin{align}
	&
	\nonumber
	\left| \int_{D( 4T )} a( z ) \int_{ B( z, 2T ) } a( z' )^{*}   F_{t,t', d( z,z' )}  \sum_{\gamma \in \Gamma}^{} F_{t,t', d( \gamma z, z' )} dz' dz \right| 
	\\
	+ \ & 
	\| a \|_{\infty}^{2} \int_{D \setminus D( 4T )} \int_{ B( z, 2T ) }  F_{t,t', d( z,z' )}  \sum_{\gamma \in \Gamma}^{} F_{t,t', d( \gamma z, z' )} dz' dz,
\end{align}
where the large and small injectivity radius terms are bounded separately. 

The small injectivity radius term is bounded in Lemma \ref{Lmm: Large injectivity radius}. Using this estimate, we obtain the following upper bound:
\begin{align*}
	&
	30 \| a \|_{2}^{2} \, \int_{0}^{2T}  \sinh( \rho ) \, F_{t,t',\rho}^{2} \, ( 1 + \rho ) \, e^{ -\beta \rho } d \rho 
	+ 2 \, \| a \|_{\infty}^{2} \int_{D \setminus D( 4T )} \int_{ B( z, 2T ) }   F_{t,t', d( z,z' )}  \sum_{\gamma \in \Gamma}^{} F_{t,t', d( \gamma z, z' )} dz' dz.
\end{align*}
The second term of the previous expression is bounded in Lemma \ref{Lmm: Small injectivity radius}, yielding the following upper bound:
\begin{align*}
	&
	30 \| a \|_{2}^{2} \, \int_{0}^{2T}  \sinh( \rho ) \, F_{t,t',\rho}^{2} \, ( 1 + \rho ) \, e^{ -\beta  \rho } d \rho 
	\ + \ 8\pi^{2}  \, \| a \|_{\infty}^{2} | D \setminus D( 4T ) | \frac{e^{4T}}{\operatorname{InjRad}_{X}} \int_{0}^{2T} \sinh( \rho ) F_{t,t',\rho}^{2} d\rho.
\end{align*}
This is further bounded by
\begin{align*}
	&
	8 \pi^{2}\left( \| a \|_{2}^{2} \, + \, \| a \|_{\infty}^{2} \, | D \setminus D( 4T ) | \, \frac{ e^{ 6T }}{ \operatorname{InjRad}_{X} } \right) \int_{0}^{2T} \sinh( \rho ) e^{ -\beta \rho} ( 1 + \rho ) F_{t,t',\rho}^{2} d \rho,
\end{align*}
using that $ 8 \pi^{2} > 30 $ and $ e^{ 2T } \, e^{ -\beta \rho } \, ( 1 + \rho ) \geq 1 $.

Hence, Expression \eqref{Expr: Introducing F} is bounded by
\begin{align*}
	&
	\frac{1}{ T^{2}}  \,  \left( \| a \|_{2}^{2} \, + \,  \| a \|_{\infty}^{2} \, | D \setminus D( 4T ) | \frac{e^{6T}}{ \operatorname{InjRad}_{X} } \right)\int_{0}^{T} \int_{0}^{T} \int_{0}^{2T} \sinh( \rho ) e^{-\beta \rho} ( 1 + \rho ) F_{t,t', \rho}^{2} d\rho dt' dt.
\end{align*} 
By symmetry of $ F_{t,t',\rho} $ in $ t $ and $ t' $, we can restrict the inner integral to $ 0 \le t' \le t $ and multiply by $ 2 $:
\begin{align*}
	&
	\frac{2}{ T^{2}} \, \left( \| a \|_{2}^{2} \, + \, \| a \|_{\infty}^{2} \, | D \setminus D( 4T ) | \frac{e^{6T }}{ \operatorname{InjRad}_{X} } \right)\int_{0}^{T} \int_{0}^{t} \int_{0}^{2T} \sinh( \rho ) e^{ -\beta \rho } ( 1 + \rho ) F_{t,t', \rho}^{2} d\rho dt' dt.
\end{align*}
Applying Proposition \ref{Lmm: Weight integral} to the innermost integral, we obtain the upper bound
\begin{align*}
\frac{2}{T^{2}} \left( \| a \|_{2}^{2} + \| a \|_{\infty}^{2} \, | D \setminus D( 4T ) | \frac{e^{6T}}{\operatorname{InjRad}_{X}} \right)
\int_{0}^{T} \int_{0}^{t} \frac{C_{6}}{\beta^{2}} \, e^{- \frac{\beta}{4}(t-t')} \, dt' \, dt,
\end{align*}
where \(C_{6} = 10000\).
Integrating in $ t' $ and $ t $ gives an upper bound of $ 4T C_{6} / \beta^{3} $, so the entire expression is bounded by
\begin{align*}
	&
	\frac{C_{7}}{T \, \beta^{3}} \, \left( \| a \|_{2}^{2} \, + \, \| a \|_{\infty}^{2} \, | D \setminus D( 4T ) | \frac{e^{6 T}}{ \operatorname{InjRad}_{X} } \right),
\end{align*}
with $2 \cdot 4 \cdot C_{6} = 80000 = C_{7}$.

\end{proof}

The remaining part of this section contains the lemmas needed to prove Proposition \ref{Prp: Geometric data}, together with the contents of Section \ref{Sec: Weight function}. The following lemma allows us to write the integration containing the kernels living on the hyperbolic surface using the kernels living on the hyperbolic plane.

\begin{lemma}[]
	\label{Lmm: Eliminate sum 1}
	Let 
	$$ K^{\Gamma}( x,y ) = \sum_{\gamma \in \Gamma}^{} K( x, \gamma y ), \qquad L^{\Gamma}( x,y ) = \sum_{\gamma \in \Gamma}^{} L( x, \gamma y ) $$ 
	be integral kernels on a compact connected hyperbolic surface $ X= \Gamma \setminus \mathbb{H} $, where $ K $ and $ L $ are integral kernels on $ \mathbb{H} $ that depend only on distance $ d( x,y ) $. Then
	\begin{align*}
		&
		\int_{D} K^{\Gamma}( x,z ) L^{\Gamma}( x,z' ) dx = \int_{\mathbb{H}}  \sum_{\gamma \in \Gamma}^{} K(  x, z ) L( x, \gamma z' ) dx,
	\end{align*}
	where $ D $ is a fundamental domain of $ X $.
\end{lemma}

The proof is based on uniting the sum over the group $ \Gamma $ and the integration over the fundamental domain $ D $ into an integration over the hyperbolic plane.

\begin{proof}[Proof of Lemma \ref{Lmm: Eliminate sum 1}]
	We have that
	\begin{align*}
		&
		\int_{D} K^{\Gamma}( x,z ) L^{\Gamma}( x,z' ) dx = \int_{D} \sum_{\gamma_{1} \in \Gamma}^{} \sum_{\gamma_{2} \in \Gamma}^{} K( x, \gamma_{1} z ) L( x, \gamma_{2} z' ) \,dx
	\end{align*}
	by opening the kernels using their sum forms. Doing a change of variables $ \gamma_{2} \mapsto \gamma_{1} \gamma_{2} $ for the inner sum allows us to write the previous expression as
	\begin{align*}
		&
		\int_{D} \sum_{\gamma_{1} \in \Gamma}^{} \sum_{\gamma_{2} \in \Gamma}^{} K( x, \gamma_{1} z ) L( x, \gamma_{1} \gamma_{2} z' ) \, dx.
	\end{align*}
	Using that $ K( x, \gamma_{1} z ) = K( \gamma_{1}^{-1} x, z ) $ and that $ L( x, \gamma_{1} \gamma_{2} z' ) = L( \gamma_{1}^{-1} x, \gamma_{2} z' ) $ allows us then to write the previous expression as:
	\begin{align*}
		&
		\int_{D} \sum_{\gamma_{1} \in \Gamma}^{} \sum_{\gamma_{2} \in \Gamma}^{} K( \gamma_{1}^{-1} x, z ) L( \gamma_{1}^{-1} x, \gamma_{2} z' ) dx.
	\end{align*}
	Uniting the integral and the sum over $ \gamma_{1} $ into the integral over the universal cover gives us the following equal expression:
	\begin{align*}
		&
		\int_{\mathbb{H}}  \sum_{\gamma \in \Gamma}^{} K(  x, z ) L( x, \gamma z' ) \, dx,
	\end{align*}
	where we opted to use $ \gamma $ since there is no need to distinguish the sums anymore. This concludes the proof.
\end{proof}

The following lemma allows us to simplify an integral expression in which sums over the group $ \Gamma $ appear.

\begin{lemma}[]
	\label{Lmm: Eliminate sum 2}
	Let 
	$$ K^{\Gamma}( x,y ) = \sum_{\gamma \in \Gamma}^{} K( x, \gamma y ), \qquad L^{\Gamma}( x,y ) = \sum_{\gamma \in \Gamma}^{} L( x, \gamma y ) $$ 
	be integral kernels on a compact connected hyperbolic surface $ X= \Gamma \setminus \mathbb{H} $, where $ K $ and $ L $ are integral kernels on $ \mathbb{H} $ that depend only on distance $ d( x,y ) $. Let $ a $ be a function on $ X $. Then
	\begin{align*}
		&
		\int_{D}  a( z' )^{*} \left( \int_{\mathbb{H}} \sum_{\gamma \in \Gamma}^{}K( x,z ) L( x, \gamma z' ) dx \right)^{2}  dz'
		\\
		& = \
		\int_{\mathbb{H}} a( z' )^{*}  \left( \int_{\mathbb{H}} K( x,z ) L( x, z' ) dx \right) \sum_{\gamma \in \Gamma}^{}  \left( \int_{\mathbb{H}} K( y,\gamma z ) L( y, z' ) dy \right) dz',
	\end{align*}
	where $ D $ is a fundamental domain of $ X $.
\end{lemma}

The proof is based on the same principle as the proof of Lemma \ref{Lmm: Eliminate sum 1}: we unite the sum over the group and the integral over the fundamental domain into an integration over the hyperbolic plane.

\begin{proof}[Proof of Lemma \ref{Lmm: Eliminate sum 2}]
	Focus on the following expression:
	\begin{align*}
		&
		\int_{D} a( z' )^{*} \left( \int_{\mathbb{H}} \sum_{\gamma \in \Gamma}^{} K( x,z ) L( x, \gamma z' ) dx \right)^{2}  dz'.
	\end{align*}
	We expand the squares and reorder the integrals and the sums using Fubini's theorem to get the following expression:
	\begin{align*}
		&
		\int_{D} \sum_{\gamma_{1} \in \Gamma}^{} \sum_{\gamma_{2} \in \Gamma}^{} a( z' )^{*} \left( \int_{\mathbb{H}} K( x, z ) L( x, \gamma_{1} z' ) dx \right) \left( \int_{\mathbb{H}} K( y, z ) L( y, \gamma_{2} z' ) dy \right) dz'.
	\end{align*}
	Doing a change of variables $ \gamma_{2} \mapsto \gamma_{2}^{-1} \gamma_{1} $ in the inner sum allows us to write the previous expression as
	\begin{align*}
		&
		\int_{D} \sum_{\gamma_{1} \in \Gamma}^{} \sum_{\gamma_{2} \in \Gamma}^{} a( z' )^{*} \left( \int_{\mathbb{H}} K( x, z ) L( x, \gamma_{1} z' ) dx \right) \left( \int_{\mathbb{H}} K( y, z ) L( y, \gamma_{2}^{-1} \gamma_{1} z' ) dy \right) dz',
	\end{align*}
	which further equals
	\begin{align*}
		&
		\int_{D} \sum_{\gamma_{1} \in \Gamma}^{} \sum_{\gamma_{2} \in \Gamma}^{} a( z' )^{*} \left( \int_{\mathbb{H}} K( x, z ) L( x, \gamma_{1} z' ) dx \right) \left( \int_{\mathbb{H}} K( y, z ) L( \gamma_{2} y, \gamma_{1} z' ) dy \right) dz',
	\end{align*}
	when using that $ L( x, \gamma y ) = L( \gamma^{-1} x, y  ) $. Uniting the integral and the sum over $ \gamma_{1} $ into the integral over the universal cover is possible since $ a^{*} $ can be seen as a $ \Gamma $-periodic function on $ \mathbb{H} $. This gives the following equal expression:
	\begin{align}
		&
		\label{Expr: One sum left}
		\int_{\mathbb{H}} \sum_{\gamma \in \Gamma}^{} a( z' )^{*} \left( \int_{\mathbb{H}} K( x,z ) L( x, z' ) dx \right) \left( \int_{\mathbb{H}} K( y,z ) L( \gamma y, z' ) dy \right) dz',	
	\end{align}
	where we opted to use $ \gamma $ since there is no need to distinguish the sums anymore. For the latter integral we have that
	\begin{align*}
		&
		\int_{\mathbb{H}} K( y,z ) L( \gamma y, z' ) dy = \int_{\mathbb{H}} K( \gamma^{-1} y, z ) L( y, z' ) dy = \int_{\mathbb{H}} K( y, \gamma z ) L( y, z' ) dy,
	\end{align*}
	where the first equality follows from the change of variables $ y \mapsto \gamma^{-1} y $, and the second equality follows from the fact that $ L( \gamma^{-1}x,y ) = L( x, \gamma y ) $. With this observation, we can rewrite Expression \eqref{Expr: One sum left} as
	\begin{align*}
		&
		\int_{\mathbb{H}} a( z' )^{*}  \left( \int_{\mathbb{H}} K( x,z ) L( x, z' ) dx \right) \sum_{\gamma \in \Gamma}^{}  \left( \int_{\mathbb{H}} K( y,\gamma z ) L( y, z' ) dy \right) dz'.
	\end{align*}
\end{proof}

The following lemma addresses integration over the part of the fundamental domain with large injectivity radius, yielding an upper bound consisting of two terms. 
The first term is an integral whose integrand exhibits exponential decay, while the second term is an integral over points with small injectivity radius. 
The latter term is bounded in Lemma \ref{Lmm: Small injectivity radius}.

\begin{lemma}[]
	\label{Lmm: Large injectivity radius}
	Let $ X = \Gamma \setminus \mathbb{H}$ be a compact connected hyperbolic surface with fundamental domain $ D $. Let $ D( 4T ) $ denote the points with injectivity radius larger than $ 4T $. Let $ a $ be a mean zero function on $ X $. Then
	\begin{align}
		&
		\nonumber
		\left| \int_{D( 4T )} a( z ) \int_{ B( z, 2T ) } a( z' )^{*}   F_{t,t', d( z,z' )}  \sum_{\gamma \in \Gamma}^{} F_{t,t', d( \gamma z, z' )} dz' dz \right|
		\\
		& = \
		\label{Expr: Large injectivity radius}
		30 \| a \|_{2}^{2} \, \int_{0}^{2T}  \sinh( \rho ) \, F_{t,t',\rho}^{2}  \, ( 1 + \rho ) \, e^{ - \beta \rho } \, d \rho 
		 + 
		\| a \|_{\infty}^{2} \int_{ D \setminus D( 4T )} \int_{ B( z, 2T ) } F_{t,t', d( z,z' )} \sum_{\gamma \in \Gamma}^{} F_{t,t', d( \gamma z, z' )} dz' dz,
	\end{align}
	where $ \beta = \beta( \lambda_{1}( X ) ) $, $ \lambda_{1}( X ) $ is the spectral gap of the Laplace--Beltrami operator $ -\Delta_{X} $, and
\begin{align*}
	\beta( x ) = 
	\begin{cases}	
		1 - \sqrt{ 1 - 4x }, & \text{ if } x \leq \frac{1}{4},\\
		1, & \text{ if } x > \frac{1}{4}.
	\end{cases}
\end{align*}
\end{lemma}

The proof is based on noticing that only the identity element can have a non-zero contribution. After this, we separate the integration over the whole fundamental domain and the integration over points with small injectivity radius using the triangle inequality. The crucial part is bounding the first term using the exponential mixing theorem, Theorem \ref{Thm: Exponential mixing}.

\begin{proof}[Proof of Lemma \ref{Lmm: Large injectivity radius}]
	As $ z \in D( 4T ) $, we have $ d( z, \gamma z ) > 4T $ for any non-identity element $ \gamma \in \Gamma $. Also, $ z' \in B( z, 2T ) $, so $ d( \gamma z, z' ) > 2T $ for any non-identity element. Since $ F_{t,t',d( z,z' )} = 0 $ if $ d( z,z' ) > 2T $, only the identity element contributes. Hence the left-hand side of Expression \eqref{Expr: Large injectivity radius} equals
	\begin{align}
		&
		\label{Expr: Sum eliminated}
		\left| \int_{D( 4T )} a( z ) \int_{ B( z, 2T ) } a( z' )^{*}   F_{t,t', d( z, z' )}^{2} dz' dz\right|.
	\end{align}
	Adding and subtracting
	\begin{align*}
		&
		\int_{D \setminus D( 4T )} a( z ) \int_{ B( z, 2T ) } a( z' )^{*}   F_{t,t', d( z, z' )}^{2} dz' dz
	\end{align*}
	and applying the triangle inequality gives the following upper bound for Expression \eqref{Expr: Sum eliminated}:
	\begin{align}
		&
		\left| \int_{D} a( z ) \int_{ B( z, 2T ) } a( z' )^{*}    F_{t,t', d( z, z' )}^{2} dz' dz\right| 
		\ + \
		\label{Expr: New large small InjRad split}
		\left| \int_{ D \setminus D( 4T )} a( z ) \int_{ B( z, 2T ) } a( z' )^{*}   F_{t,t', d( z,z' )}^{2} dz' dz\right|.
	\end{align}

	We focus on the second term of the previous expression. 
By moving the absolute value inside, bounding \( |a| \) by its supremum norm, and introducing a sum, we obtain an upper bound corresponding to the second term on the right-hand side of Expression \eqref{Expr: Large injectivity radius}. 

Focus on the first term of Expression \eqref{Expr: New large small InjRad split}. Write $ z' $ in polar coordinates around $ z $, so $ z' = \pi_{1}(\varphi_{\rho}( z,\theta )) $, giving
	\begin{align*}
		&
		\left| \int_{D} a( z ) \int_{\mathbb{S}^{1}} \int_{0}^{2T} \sinh( \rho ) \,  a( \varphi_{\rho}( z,\theta ) )^{*}  \,  F_{t,t', \rho }^{2}  \, d\rho \, d\theta \, dz\right|.
	\end{align*}
	Reordering the integrals using Fubini's theorem gives
	\begin{align*}
		&
		\left| \int_{0}^{2T} \sinh( \rho ) \, F_{t,t',\rho}^{2}  \int_{D} \int_{\mathbb{S}^{1}}  a( z ) \, a( \varphi_{\rho}( z, \theta ) )^{*} \, d\theta \, dz \, d \rho \right|.
	\end{align*}
	Pulling the absolute value inside the first integral, we have
	\begin{align*}
		&
		\int_{0}^{2T} \sinh( \rho )  \, F_{t,t',\rho}^{2} \, \left| \left \langle a, a \circ \varphi_{\rho} \right \rangle_{ L^{2}( D \times \mathbb{S}^{1} ) }  \right|  d \rho.
	\end{align*}
	Applying the Exponential Mixing Theorem, Theorem \ref{Thm: Exponential mixing}, to the inner product yields the upper bound
	\begin{align*}
		&
		\|a \|_{2}^{2} \,\int_{0}^{2T} 11 \, \sinh( \rho ) \, F_{t,t',\rho}^{2} \, e^{ \beta } ( 1 + \rho ) e^{ - \beta \rho } \, d \rho \leq 30 \| a \|_{2}^{2} \int_{0}^{2T} F_{t,t',\rho}^{2} \, ( 1 + \rho ) \, e^{ -\beta \rho } \, d\rho,
	\end{align*}
	where the latter inequality followed by using that $ \beta \leq 1 $ and $11 e \leq 30$. This concludes the proof.

\end{proof}

The following lemma bounds the integration over the part of the fundamental domain with small injectivity radius.

\begin{lemma}[]
	\label{Lmm: Small injectivity radius}
	Let $ X = \Gamma \setminus \mathbb{H} $ be a compact connected hyperbolic surface with fundamental domain $ D $, and let $ D( 4T ) $ denote the points with injectivity radius larger than $ 4T $. Then
	\begin{align*}
		&
		\int_{D \setminus D( 4T )} \int_{ B( z, 2T ) }   F_{t,t', d( z,z' )}  \sum_{\gamma \in \Gamma}^{} F_{t,t', d( \gamma z, z' )} dz' dz 
		\leq 4 \pi^{2} \, | D \setminus D( 4T ) | \, \frac{e^{4T}}{\operatorname{InjRad}_{X}} \int_{0}^{2T} \sinh( \rho ) F_{t,t',\rho}^{2} d\rho.
	\end{align*}
\end{lemma}

The proof is based on the careful use of Hölder's inequality. Of particular importance is understanding the number of group elements with non-zero contribution.

\begin{proof}[Proof of Lemma \ref{Lmm: Small injectivity radius}]
	Consider 
	\begin{align}
		&
		\label{Expr: Small injectivity radius}
		\int_{D \setminus D( 4T )} \int_{ B( z, 2T ) }   F_{t,t', d( z,z' )}  \sum_{\gamma \in \Gamma}^{} F_{t,t', d( \gamma z, z' )} dz' dz.
	\end{align}
	Only the terms $ \gamma \in \Gamma $ with $ d( \gamma z, z' ) < 2T $ for some $ z' \in B( z,2T ) $ contribute. Denote this collection by $ \Gamma( z ) $. Then
	\begin{align*}
		\# \Gamma( z ) \leq \frac{ | B( z,4T ) | }{ \operatorname{InjRad}_{X} } \leq 2 \pi \, \frac{e^{4T}}{\operatorname{InjRad}_{X}}.
	\end{align*}
	This is because \( \gamma z \) must lie in \( B(z, 4T) \) for it to be possible that \( d(\gamma z, z') < 2T \), as \( z' \in B(z, 2T) \). Restricting the sum over \( \Gamma(z) \) in Expression \eqref{Expr: Small injectivity radius}, and then applying Hölder's inequality twice, first to the integral and then to the sum, results in the following upper bound:
	\begin{align}
		&
		\label{Expr: New small InjRad}
		 \left( \int_{D \setminus D( 4T )} \int_{B( z,2T )} F_{t,t', d( z,z' ) }^{2} dz' dz \right)^{\frac{1}{2}} \left( \int_{D \setminus D( 4T )} \int_{B( z,2T )}  \#\Gamma( z ) \sum_{\gamma \in \Gamma( z )}^{} F_{t,t', d( \gamma z,z' ) }^{2}  dz' dz \right)^{\frac{1}{2}}.
	\end{align}
	Focus on the right factor. It admits the following upper bound after using Fubini's theorem to reorder the sums and integrals, and then applying the bound on \( \# \Gamma(z) \):
	\begin{align*}
		&
		\left( \frac{2\pi \, e^{4T}}{ \operatorname{InjRad}_{X} } \int_{D \setminus D( 4T )} \sum_{\gamma \in \Gamma( z )}^{} \int_{B( z,2T )}   F_{t,t', d( \gamma z,z' ) }^{2}  dz' dz \right)^{\frac{1}{2}}.
	\end{align*}
	Changing variables via \( z' \mapsto \gamma z' \) gives the following equal expression:
	\begin{align*}
		&
		\left( \frac{2\pi \, e^{4T}}{ \operatorname{InjRad}_{X} } \int_{D \setminus D( 4T )} \sum_{\gamma \in \Gamma( z )}^{} \int_{B( \gamma^{-1} z,2T )}   F_{t,t', d( \gamma z, \gamma z' ) }^{2}  dz' dz \right)^{\frac{1}{2}}.
	\end{align*}
	Since \( d(\gamma z, \gamma z') = d(z, z') \) and \( F_{t,t',\rho} = 0 \) for \( \rho > 2T \), the domain of integration in the inner integral can be restricted to \( B(\gamma^{-1}z, 2T) \cap B(z, 2T) \subset B(z, 2T) \). Hence, we obtain the following upper bound for the previous integral expression:
	\begin{align*}
		&
\left( \frac{2\pi \, e^{4T}}{ \operatorname{InjRad}_{X} } \int_{D \setminus D(4T)} \sum_{\gamma \in \Gamma(z)} \int_{B(z,2T)} F_{t,t', d(z, \gamma z')}^{2} \, dz' \, dz \right)^{\frac{1}{2}}
\\
& \leq \ 
\left( \frac{4\pi^{2} \, e^{8T}}{ \operatorname{InjRad}_{X} } \int_{D \setminus D(4T)} \int_{B(z,2T)} F_{t,t', d(z, \gamma z')}^{2} \, dz' \, dz \right)^{\frac{1}{2}}.
\end{align*}
The inequality follows because the contents of the sum no longer depends on \(\gamma\), so we can use the bound on \(\#\Gamma(z)\).

Hence, we obtain the following upper bound for Expression \eqref{Expr: New small InjRad}:
\begin{align*}
2 \pi \, \frac{e^{4T}}{\operatorname{InjRad}_{X}} \int_{D \setminus D(4T)} \int_{B(z,2T)} F_{t,t',d(z,z')}^{2} \, dz' \, dz 
= 4 \pi^{2} \, |D \setminus D(4T)| \, \frac{e^{4T}}{\operatorname{InjRad}_{X}} \int_{0}^{2T} \sinh(\rho) \, F_{t,t',\rho}^{2} \, d\rho,
\end{align*}
where the equality follows from writing \( z' \) in polar coordinates around \( z \) and then integrating over \( z \). This concludes the proof.

\end{proof}

The following lemma bounds an integral in which $ F_{t,t',\rho} $ appears.

\begin{lemma}[]
	\label{Lmm: Weight integral}
	Let $ \beta $ be a positive constant smaller than $ 1 $ and let $ 0 < t' < t $. Then
	\begin{align*}
		&
		\int_{0}^{2T} \sinh( \rho ) \, ( 1 + \rho ) \, e^{ - \beta  \rho } F_{t,t',\rho}^{2} d \rho \leq \frac{C_{6}}{ \beta^{2} }  \, e^{ - \frac{\beta}{4}( t - t' ) },
	\end{align*}
	when
	\begin{align*}
		F_{t,t',\rho} \leq 50 \frac{1}{\sqrt{ \sinh( \max( t-t', \rho ) ) }} \boldsymbol 1_{ \rho \in ( 0, t + t' ) }, 
	\end{align*}
	where $C_{6} = 10000$.
\end{lemma}

The proof is based on splitting the integral into two terms according to the upper bound of \( F_{t,t',\rho} \).

\begin{proof}[Proof]

	We seek to bound
\begin{align*}
	&
\int_{0}^{2T} \sinh(\rho) \, (1 + \rho) \, e^{-\beta \rho} \, F_{t,t',\rho}^{2} \, d\rho 
\\
&\leq 50^{2} \left( \frac{1}{\sinh(t - t')} \int_{0}^{t - t'} \sinh(\rho) \, (1 + \rho) \, e^{-\beta \rho} \, d\rho 
+ \int_{t - t'}^{t + t'} (1 + \rho) \, e^{-\beta \rho} \, d\rho \right),
\end{align*}
where the inequality follows from the bound on \( F_{t,t',\rho} \) and splitting the integration domain. Using
	\begin{align*}
		&
		( 1 + \rho ) \leq \frac{2}{\beta} \, e^{ \frac{\beta}{2} - 1 } \, e^{ \frac{\beta}{2} \rho } \leq \frac{2}{\beta} e^{ \tfrac{ \beta }{2} \rho },
	\end{align*}
	which can be verified by analyzing derivatives, we obtain the following upper bound for the previous integral expression:
	\begin{align}
		&
		\label{Expr: Three domains}
		\frac{2 \cdot 50^{2} }{ \beta}\left( \frac{ 1 }{ \sinh( t - t' )  } \int_{0}^{t-t'}   \sinh( \rho )  \, e^{ - \frac{\beta}{2} \rho }  d \rho + \int_{t-t'}^{t + t'} e^{ - \frac{\beta}{2} \rho } d \rho \right).
	\end{align}
	We have
\[
\sinh(\rho)\, e^{-\tfrac{\beta}{2}\rho} 
\leq 
\sinh(t - t')\, e^{-\tfrac{\beta}{2}(t - t')}
\]
in the first term of the previous integral expression. Applying this inequality to the first term of Expression \eqref{Expr: Three domains} and then integrating both integrals yields the following upper bound for Expression \eqref{Expr: Three domains}:
\begin{align}
	\label{Expr: Many betas}
\frac{2 \cdot 50^{2}}{\beta}\left(
(t - t')\, e^{-\tfrac{\beta}{2}(t - t')}
+
\frac{2}{\beta}
\left(
e^{-\tfrac{\beta}{2}(t - t')}
-
e^{-\tfrac{\beta}{2}(t + t')}
\right)
\right).
\end{align}
Using that
\[
x e^{-b x} \leq \frac{2}{e}\frac{1}{b} e^{-b x/2}
\]
for all \( x > 0, b \in ( 0,1 ) \), Expression \eqref{Expr: Many betas} is bounded by
\begin{align*}
\frac{10000}{\beta^{2}} \, e^{-\frac{\beta}{4}(t - t')}.
\end{align*}

\end{proof}


\section{Weight function}

\label{Sec: Weight function}

\noindent
Consider the following integral:
\begin{align*}
	&
	F_{t,t',\rho} = \int_{ B( z,t ) \cap B( z', t' ) } \frac{dx}{ \sqrt{ \cosh( t ) - \cosh( d( x,z ) )  } \, \sqrt{ \cosh( t' ) - \cosh( d( x,z' ) ) } },
\end{align*}
where $ z,z' \in \mathbb{H} $ are any two points separated by distance $ \rho $. The following lemma shows that the chosen points do not matter as long as the distance between the points is the same.

\begin{lemma}[]
	\label{Lmm: Points do not matter}
	Let $ z,z', w, w' \in \mathbb{H} $ such that $ d( z,z' ) $ equals $ d( w,w' ) $. Then
	\begin{align}
		&
	     \nonumber
		\int_{ B( z,t ) \cap B( z', t' ) } \frac{dx}{ \sqrt{ \cosh( t ) - \cosh( d( x,z ) )  } \, \sqrt{ \cosh( t' ) - \cosh( d( x,z' ) ) } }	
	     \\ & = \ 
		\label{Expr: z z' integral}
		\int_{ B( w,t ) \cap B( w', t' ) } \frac{dx}{ \sqrt{ \cosh( t ) - \cosh( d( x,w ) )  } \, \sqrt{ \cosh( t' ) - \cosh( d( x,w' ) ) } }.
	\end{align}
		
\end{lemma}

The proof is based on the properties of isometries in the hyperbolic plane.

\begin{proof}[Proof of Lemma \ref{Lmm: Points do not matter}]

	There exists an isometry $ f $ such that $ f( z) = w $ and $ f( z' ) = w' $. Using a change of variables $ x \mapsto f^{-1}( x ) $, we get the following equal expression for the left-hand side of Expression \eqref{Expr: z z' integral}:
	\begin{align*}
		&
		\int_{ f\left( B( z,t ) \cap B( z',t' ) \right) } \frac{dx}{ \sqrt{ \cosh( t ) - \cosh( d( f^{-1}( x ), z ) ) } \, \sqrt{ \cosh( t' ) - \cosh( d( f^{-1}( x ), z' ) ) } }
		\\
		& = \
		\int_{ f\left( B( z,t ) \cap B( z',t' ) \right) } \frac{dx}{ \sqrt{ \cosh( t ) - \cosh( d( x, w ) ) } \, \sqrt{ \cosh( t' ) - \cosh( d( x, w' ) ) } },
	\end{align*}
	where the equality follows from the fact that $ d( f^{-1}( x ),z ) =  d( x, f( z ) ) $, which then equals $ d( x, w ) $, and similarly $ d( f^{-1}( x ),z' ) = d( x, w' ) $. As $ f $ maps $ z $ to $ w $ and $ z' $ to $ w' $, we have that
	\begin{align*}
		&
		f\left( B( z,t ) \cap B( z', t' ) \right) = B( w,t ) \cap B( w',t' ).
	\end{align*}
	Using the previous observation gives us that the previous integral expression equals
	\begin{align*}
		&
		\int_{ B( w,t ) \cap B( w', t' ) } \frac{dx}{ \sqrt{ \cosh( t ) - \cosh( d( x,w ) )  } \, \sqrt{ \cosh( t' ) - \cosh( d( x,w' ) ) } }.
	\end{align*}

\end{proof}

The main goal of this section is to give an upper bound for $ F_{t,t',\rho} $. $ F_{t,t',\rho} $ is a crucial weight appearing in Section \ref{Sec: Geometric data} when proving bounds for the geometric data of the main problem. The upper bound is given in the following proposition.

\begin{proposition}[]

	Let
	
	\begin{align*}
		&
		F_{t,t',\rho} = \int_{ B( z,t ) \cap B( z', t' ) } \frac{dx}{ \sqrt{ \cosh( t ) - \cosh( d( x,z ) )  } \, \sqrt{ \cosh( t' ) - \cosh( d( x,z' ) ) } },
\end{align*}
where $ z,z' \in \mathbb{H} $ are any two points separated by distance $ \rho $. Assume that $ t > t' > 0 $.
	
	\label{Prp: Weight function}
\begin{align*}
	&
	F_{t,t',\rho} \leq C_{1}  \, \frac{1}{\sqrt{ \sinh( \max( t-t', \rho ) ) }} \boldsymbol 1_{ \rho \in ( 0, t + t' ) },
\end{align*}
where $C_{1} = 50$.

\end{proposition}
\begin{remark}[]
	
	The case of $ t > t' > 0 $ is symmetric to the case of $ t' > t > 0 $. 

\end{remark}

Despite the simplicity of the upper bound, obtaining it is tedious. The proof is based on carefully using hyperbolic geometry. In case $ \rho \in (-\infty, 0) \cup ( t + t', \infty ) $, $ F_{t,t', \rho} = 0 $ since the integration domain is an empty set. The remaining data for the bound is obtained in three separate lemmas, Lemmas \ref{Lmm: Case 1}, \ref{Lmm: Case 2}, and \ref{Lmm: Case 3}, as the settings differ significantly from case to case. The remainder of the section is dedicated to stating and proving these three lemmas. 

We state the first lemma concerning the case of $\rho \in ( 0 , t - t' )$.

\begin{lemma}[]
    \label{Lmm: Case 1}

    Let
	
	\begin{align*}
		&
		F_{t,t',\rho} = \int_{ B( z,t ) \cap B( z', t' ) } \frac{dx}{ \sqrt{ \cosh( t ) - \cosh( d( x,z ) )  } \, \sqrt{ \cosh( t' ) - \cosh( d( x,z' ) ) } },
	\end{align*}
where $ z,z' \in \mathbb{H} $ are any two points separated by distance $ \rho $. Assume that $ t > t' > 0 $ and that $ \rho \in ( 0, t - t' ) $. Then
	\begin{align*}
		&
		F_{t,t',\rho} \leq \frac{C_{2}}{ \sqrt{ \sinh( t - t' ) } } ,
	\end{align*}
where $C_{2} = 40$.

\end{lemma}

The following figure, Figure \ref{Fig: Case 1}, illustrates the integration domain.
\begin{figure}[H]
    \centering
    \includegraphics[width=230px]{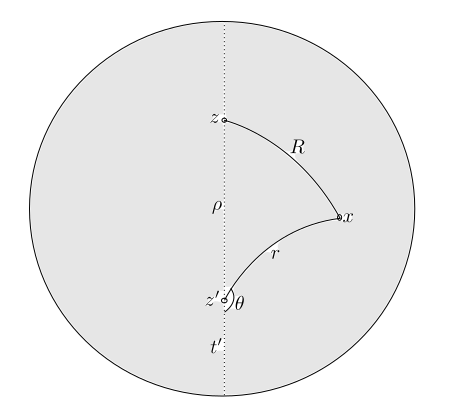}
    \caption{The case of $ \rho \in ( 0, t - t' ) $. The integration domain $ B( z,t ) \cap B( z',t' ) $ is just the smaller ball $ B( z', t' ) $ since it is contained within the larger ball $ B( z,t ) $. }
    \label{Fig: Case 1}
\end{figure}

In the proof, polar coordinates around $ z' $ are used to present the integration domain $ B( z,t ) \cap B( z', t' ) = B( z', t' ) $. The main idea is to carefully split the integration into an integration over small angles and an integration over large angles; the small angles are especially delicate, as here $\partial B( z,t )$ and $\partial B( z', t' )$ can be problematically close.

\begin{proof}[Proof of Lemma \ref{Lmm: Case 1}]

	We have that $ B( z,t ) \cap B( z', t' ) = B( z', t' ) $. Hence we write $ x $ using polar coordinates around $ z' $, taking $ \theta = 0 $ in the direction $ [z,z']( z' ) $. Hence $ x = \pi_{1}( \varphi_{\rho}( z', \theta ) ) $. $ R( r,\theta ) $ is $ d( x,z ) $ in the polar coordinates. Thus:
	\begin{align}
		&
		\label{Expr: Angles and distances}
		F_{t,t',\rho} = 2 \, \int_{0}^{t'} \frac{ \sinh( r ) }{ \sqrt{ \cosh( t' ) - \cosh( r ) } }  \left( \int_{0}^{\pi}  \frac{d \theta}{ \sqrt{ \cosh( t ) - \cosh( R( r,\theta ) ) } } \right) dr,
	\end{align}
	where only half of the angles are considered due to symmetry, thus giving the coefficient $ 2 $ in front. We notice that by the hyperbolic law of cosines
	\begin{align}
		&
		\label{Expr: Hyperbolic law of cosines}
		\cosh( R( r,\theta ) ) = \cosh( r ) \cosh( \rho ) + \sinh( r ) \sinh( \rho ) \cos( \theta ).
	\end{align}
	We also remark that 
	\begin{align}
		&
		\label{Expr: Cosh split case 1}
		\cosh( t ) = \cosh( t' ) \cosh( t - t' ) + \sinh( t' ) \sinh( t - t' ).
	\end{align}

	Focus on the inner integral of Expression \eqref{Expr: Angles and distances}. First, write it using Expressions \eqref{Expr: Hyperbolic law of cosines} and \eqref{Expr: Cosh split case 1}:
	\begin{align*}
		&
		\int_{0}^{ \pi }  \frac{d \theta}{ \sqrt{  \cosh( t-t' ) \cosh( t' ) + \sinh( t - t' ) \sinh( t' ) - \cosh( r ) \cosh( \rho ) - \sinh( r ) \sinh( \rho ) \cos( \theta )  } }.
	\end{align*} 
	We split it into three terms by splitting the integration domain:
	\begin{align}
		&
		\nonumber
		\int_{0}^{ \theta_{0}}  \frac{d \theta}{ \sqrt{  \cosh( t-t' ) \cosh( t' ) + \sinh( t - t' ) \sinh( t' ) - \cosh( r ) \cosh( \rho ) - \sinh( r ) \sinh( \rho ) \cos( \theta )  } }
		\\
		& + \
		\nonumber
		\int_{\theta_0}^{ \frac{\pi}{2} }  \frac{d \theta}{ \sqrt{  \cosh( t-t' ) \cosh( t' ) + \sinh( t - t' ) \sinh( t' ) - \cosh( r ) \cosh( \rho ) - \sinh( r ) \sinh( \rho ) \cos( \theta )  } }
		\\
		& + \
		\label{Expr: Angles split}
		\int_{\frac{\pi}{2}}^{ \pi }  \frac{d \theta}{ \sqrt{  \cosh( t-t' ) \cosh( t' ) + \sinh( t - t' ) \sinh( t' ) - \cosh( r ) \cosh( \rho ) - \sinh( r ) \sinh( \rho ) \cos( \theta )  } }.
	\end{align}
	$ \theta_{0} \in ( 0, \frac{\pi}{2} ) $ is specified later; in particular, $ \theta_{0} $ may depend on $ r $.

	We bound the first term of Expression \eqref{Expr: Angles split}. Since $ t - t' > \rho $ and $ t' > r $, we have the following upper bound for it:
	\begin{align*}
		&
		\int_{0}^{\theta_{0}}  \frac{d \theta}{ \sqrt{  \cosh( t-t' ) \cosh( t' ) - \cosh( r ) \cosh( \rho ) } } = \frac{ \theta_{0} }{ \sqrt{  \cosh( t-t' ) \cosh( t' ) - \cosh( r ) \cosh( \rho ) } }
		\\
		& \leq  \ 
		\frac{ \theta_{0} }{ \sqrt{ \cosh( t - t' ) } \sqrt{ \cosh( t' ) - \cosh( r ) } },
	\end{align*}
	where the inequality follows from $ t - t' > \rho $. 

	Let us bound the second term of Expression \eqref{Expr: Angles split}. Since $ t - t' >\rho $ and $ t' > r $, we have the following upper bound for it:
	\begin{align*}
		&
		\int_{\theta_{0}}^{ \frac{\pi}{2} } \frac{ d\theta }{ \sqrt{ \sinh( t-t' ) \sinh( t' ) - \sinh( r ) \sinh( \rho ) \cos( \theta ) } } \leq  \frac{1}{\sqrt{ \sinh( t - t' ) \, \sinh( t' ) }} \int_{\theta_{0}}^{ \frac{\pi}{2} } \frac{ d\theta }{ \sqrt{ 1 -   \cos( \theta ) } },
	\end{align*}
	where the inequality follows from $ t-t' > \rho $ and $ t' > r $. We have that
	\begin{align*}
		&
		1 - \cos( x ) \geq \frac{2}{\pi^{2}} x^{2}, \qquad x \in ( 0, \pi ),
	\end{align*}
	which means that the previous integral expression is bounded from above by
	\begin{align*}
		&
		\frac{\pi}{\sqrt{2}} \frac{1}{\sqrt{ \sinh( t-t' ) \, \sinh( t' ) }} \, \int_{\theta_{0}}^{ \frac{\pi}{2} } \frac{d\theta}{ \theta} \leq \frac{\pi^{2}}{2 \sqrt{2}} \, \frac{1}{\sqrt{ \sinh( t-t' ) \, \sinh( t' ) }} \, \frac{1}{\theta_{0}}.
	\end{align*}

	Let us bound the third term of Expression \eqref{Expr: Angles split}. Since $ t-t' > \rho $ and $ t' > r $, we have the following upper bound for it:
	\begin{align*}
		&
		\int_{ \frac{\pi}{2} }^{\pi} \frac{d \theta}{\sqrt{ \sinh( t - t' ) \, \sinh( t' ) }} = \frac{\pi}{ 2 } \, \frac{1}{\sqrt{ \sinh( t-t' ) \, \sinh( t' ) }}
	\end{align*}

	We now have an upper bound for Expression \eqref{Expr: Angles split} and thus an upper bound for Expression \eqref{Expr: Angles and distances}, the latter being:
	\begin{align*}
		&
		2 \int_{0}^{t'} \frac{\sinh( r ) \, dr}{ \cosh( t' ) - \cosh( r ) } \frac{\theta_{0}}{ \sqrt{ \cosh( t-t' ) }  } + \frac{\pi^{2}}{\sqrt{2}} \int_{0}^{t'} \frac{ \sinh( r ) \, dr }{ \sqrt{ \cosh( t' ) - \cosh( r ) } } \, \frac{1}{ \sqrt{ \sinh( t - t' ) \, \sinh( t' ) } } \, \frac{1}{\theta_{0}} 
		\\
		& + \
		\pi \int_{0}^{t'} \frac{ \sinh( r ) \, dr }{\sqrt{ \cosh( t' ) - \cosh( r ) } \sqrt{  \sinh( t-t' ) \sinh( t' )  }}
	\end{align*}
	We choose
	\begin{align*}
		\theta_{0} = \min\left( \frac{\pi}{2} , \left( \frac{ \cosh( t' ) - \cosh( r ) }{ \sinh( t' ) } \right)^{\alpha} \right),
	\end{align*}
	with $ \alpha \in ( 0, \frac{1}{2} ) $. The previous integral expression has the following upper bound:
	\begin{align*}
		&
	2 \, \frac{1}{ \sqrt{ \cosh( t - t' ) } \, \sinh( t' )^{ \alpha }} \int_{0}^{t'} \frac{ \sinh( r ) \, dr }{ ( \cosh( t' ) - \cosh( r ) )^{ 1 - \alpha } } 
	\\
		& + \ 
		\frac{\pi^{2}}{\sqrt{2}} \,  \frac{ 1 }{ \sqrt{ \sinh( t - t' ) }  \sinh( t' )^{ \frac{1}{2} - \alpha }  } \int_{0}^{t'}\frac{ \sinh( r ) \, dr }{ ( \cosh( t' ) - \cosh( r ) )^{ \frac{1}{2} + \alpha } }
		\\
		& + \
		\pi \int_{0}^{t'} \frac{ \sinh( r ) \, dr }{\sqrt{ \cosh( t' ) - \cosh( r ) } \sqrt{  \sinh( t-t' ) \sinh( t' )  }}
	\end{align*}
	The upper bound was obtained by replacing $ \theta_{0} $ with 
	\begin{align*}
		&
		\left( \frac{ \cosh( t' ) - \cosh( r ) }{ \sinh( t' ) } \right)^{\alpha}, 
	\end{align*}
	which is possible since the middle term would vanish if the previous ratio is larger than $ \frac{\pi}{2} $. The previous integral expression has the following equivalent expression obtained by integrating:
	\begin{align*}
		&
		2 \, \frac{1}{ \sqrt{ \cosh( t - t' ) } \, \sinh( t' )^{ \alpha }} \, \frac{ ( \cosh( t' ) - 1 )^{\alpha} }{ \alpha } + \frac{\pi^{2}}{\sqrt{2}} \,  \frac{ 1 }{ \sqrt{ \sinh( t - t' ) }  \sinh( t' )^{ \frac{1}{2} - \alpha }  } \frac{ ( \cosh( t' ) - 1 )^{ \frac{1}{2} - \alpha } }{ \frac{1}{2} - \alpha }
		\\
		& + \
		\pi \frac{ \sqrt{ \cosh( t' ) - 1 } }{ \sqrt{ \sinh( t' ) } } \, \frac{1}{\sqrt{ \sinh( t-t' ) }}.
	\end{align*}
	By choosing $ \alpha = \frac{1}{4} $, we have the following upper bound for the previous expression:
	\begin{align*}
		&
		40 \, \frac{1}{\sqrt{ \sinh( t - t' ) }}.	
	\end{align*}
	This is also the desired upper bound for $ F_{t,t',\rho} $.

\end{proof}

We state the second lemma concerning the case of $\rho \in ( t-t', t )$.

\begin{lemma}[]

	\label{Lmm: Case 2}

    Let
	
	\begin{align*}
		&
		F_{t,t',\rho} = \int_{ B( z,t ) \cap B( z', t' ) } \frac{dx}{ \sqrt{ \cosh( t ) - \cosh( d( x,z ) )  } \, \sqrt{ \cosh( t' ) - \cosh( d( x,z' ) ) } },
	\end{align*}
where $ z,z' \in \mathbb{H} $ are any two points separated by distance $ \rho $. Assume that $ t > t' > 0 $. Assume also that $ \rho \in ( t-t', t ) $. Then
	\begin{align*}
		&
		F_{t,t',\rho} \leq \frac{C_{3}}{ \sqrt{ \sinh( t - t' ) } } ,
	\end{align*}
where $C_{3} = 50$.

\end{lemma}

The following figure, Figure \ref{Fig: Case 2}, illustrates the integration domain.

\begin{figure}[H]
    \centering
    \includegraphics[width=230px]{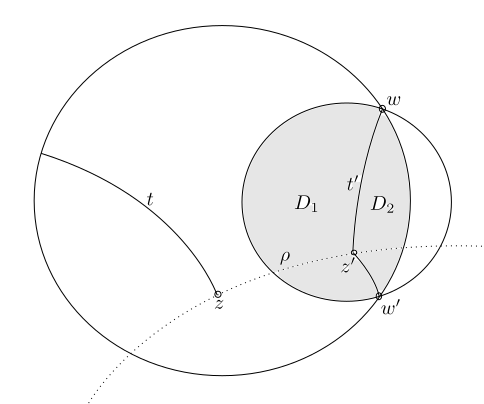}
    \caption{The case of $ \rho \in ( t-t', t ) $. The integration domain $ B( z,t ) \cap B( z',t' ) $ is the intersection of two balls. The center $ z' $ of the smaller ball is contained in the larger ball. The integration domain is split into two parts, $ D_{1} $ and $ D_{2} $, by two geodesics of length $ t' $ going from $ z' $ to the points $ w $ and $ w' $ where the two balls intersect. }
    \label{Fig: Case 2}
\end{figure}

In this case, the integration domain is not a ball, since $ \partial B( z,t ) $ and $\partial B( z', t' )$ intersect; call the intersections $w$ and $ w' $. The segments $[z',w]$ and $[z',w']$ split the integration domain into two disjoint parts, which are treated separately in the proof. For the other side ($ D_{1} $ in Figure \ref{Fig: Case 2}), the idea is to use a similar approach as in the proof of Lemma \ref{Lmm: Case 1}: use polar coordinates around $ z' $ and deal with the small and large angles separately. The remaining side ($D_{2}$ in Figure \ref{Fig: Case 2}) is easier and can be handled using polar coordinates around $ z $.

\begin{proof}[Proof of Lemma \ref{Lmm: Case 2}]

	By our assumption, we know $ \partial B( z,t ) $ and $ \partial B( z',t' ) $ intersect at exactly two points. Call these points $ w $ and $ w' $. The segments $ [z',w] $ and $ [z', w'] $ split the integration domain $ B( z,t ) \cap B( z', t' ) $ into two disjoint parts, $ D_{1} $ and $ D_{2} $.

Hence $ F_{t,t',\rho} $ can be written as
\begin{align}
	&
	\nonumber
	\int_{D_1 } \frac{ dx }{ \sqrt{ \cosh( t ) - \cosh( d( x,z ) ) } \, \sqrt{ \cosh( t' ) - \cosh( d( x,z' ) ) } } 
	\\
	& + \
	\label{Expr: D1 D2 split}
	\int_{D_2 } \frac{ dx }{ \sqrt{ \cosh( t ) - \cosh( d( x,z ) ) } \, \sqrt{ \cosh( t' ) - \cosh( d( x,z' ) ) } }.
\end{align}
We bound the two terms separately.

We begin by bounding the integration over $ D_{1} $. We write $ x $ using polar coordinates around $ z' $, taking $ \theta = 0 $ as the angle in the direction $ [z,z']( z' ) $. Hence $ x = \pi_{1}( \varphi_{r}( z', \theta ) ) $. $ R( r,\theta ) $ is $ d( x,z ) $ in the polar coordinates. We notice that the integration domain of the angles is $ ( \Theta, 2 \pi - \Theta ) $, where $ \Theta $ is the angle between $ [z,z']( z' ) $ and $ [z', w]( z' ) $. The following figure, Figure \ref{Fig: Sub1}, illustrates this.

\begin{figure}[H]
        \centering
        \includegraphics[width=175px]{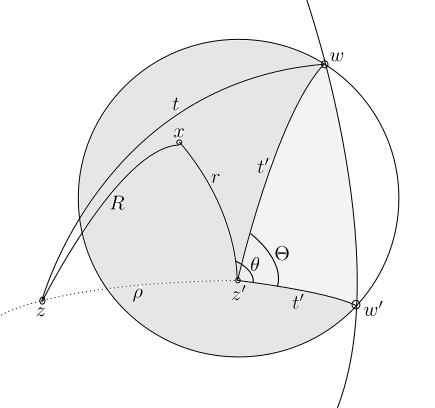}
        \caption{The integration domain $ D_{1} $ consists of the points that are outside the sector formed by $ w,z', $ and $ w' $ in ball $ B( z',t' ) $. Polar coordinates around $ z' $ are used.}
	\label{Fig: Sub1}
\end{figure}

Hence the first term of Expression \eqref{Expr: D1 D2 split} equals
\begin{align}
	&
	\label{Expr: Distances and angles 2}
	2 \, \int_{0}^{t'} \frac{ \sinh( r ) }{ \sqrt{ \cosh( t' ) - \cosh( r ) } } \left( \int_{\Theta}^{\pi} \frac{ d \theta }{ \sqrt{ \cosh( t ) - \cosh( R( r,\theta ) ) } } \right) dr, 
\end{align}
where only half of the angles are considered due to symmetry, giving the coefficient $ 2 $ in front. Using the hyperbolic law of cosines, we record two useful equalities:
\begin{align}
	&
	\label{Expr: Hyperbolic law of cosines t}
	\cosh( t ) = \cosh( \rho ) \cosh( t' ) + \sinh( \rho ) \sinh( t' ) \cos( \Theta ),
	\\
	&
	\label{Expr: Hyperbolic law of cosines R}
	\cosh( R( r, \theta ) ) = \cosh( \rho ) \cosh( r ) + \sinh( \rho ) \sinh( r ) \cos( \theta ).
\end{align}

Focus on the inner integral of Expression \eqref{Expr: Distances and angles 2}. First, we write it using Expressions \eqref{Expr: Hyperbolic law of cosines t} and \eqref{Expr: Hyperbolic law of cosines R}:
\begin{align*}
	&
	\int_{\Theta}^{\pi} \frac{d \theta}{ \sqrt{ \cosh( \rho ) \cosh( t' ) + \sinh( \rho ) \sinh( t' ) \cos( \Theta ) - \cosh( \rho ) \cosh( r ) - \sinh( \rho ) \sinh( r ) \cos( \theta ) } }.
\end{align*}
We then split it into three terms by dividing the integration domain:
\begin{align}
	&
	\nonumber
	\int_{\Theta}^{ \Theta + \theta_{0} } \frac{ d\theta }{ \sqrt{ \cosh( \rho )(  \cosh( t' ) - \cosh( r )  ) + \sinh( \rho )( \sinh( t' ) \cos( \Theta ) - \sinh( r ) \cos( \theta ) ) } }
	\\
	& + \
	\nonumber
	\int_{\Theta + \theta_{0}}^{ \frac{\pi}{2} } \frac{ d\theta }{ \sqrt{ \cosh( \rho )(  \cosh( t' ) - \cosh( r )  ) + \sinh( \rho )( \sinh( t' ) \cos( \Theta ) - \sinh( r ) \cos( \theta ) ) } }
	\\
	& + \
	\label{Expr: Angles split case 2}
	\int_{ \frac{\pi}{2} }^{ \pi } \frac{ d\theta }{ \sqrt{ \cosh( \rho )(  \cosh( t' ) - \cosh( r )  ) + \sinh( \rho )( \sinh( t' ) \cos( \Theta ) - \sinh( r ) \cos( \theta ) ) } }.
\end{align}
We remark that $ \Theta \in ( 0, \frac{\pi}{2} ) $ since $ z' $ is contained in $ B( z,t ) $ as $ \rho \in ( t-t', t ) $. $ \theta_{0} \in ( 0, \frac{\pi}{2} - \Theta ) $ and will be specified later. We note that $ \theta_{0} $ can depend on $ r $. We bound these terms separately.

We start by bounding the first term of Expression \eqref{Expr: Angles split case 2}. It is bounded from above by
\begin{align*}
	&
	\int_{\Theta}^{\Theta + \theta_{0}} \frac{ d\theta }{ \sqrt{ \cosh( \rho )\, ( \cosh( t' ) - \cosh( r ) ) } } = \frac{ \theta_{0} }{ \sqrt{ \cosh( \rho )\, ( \cosh( t' ) - \cosh( r ) ) } },
\end{align*}
since $ \sinh( \rho )( \sinh( t' ) \cos( \Theta ) - \sinh( r ) \cos( \theta ) ) $ is positive.

Next, we bound the second term of Expression \eqref{Expr: Angles split case 2}. It is bounded from above by 
\begin{align*}
	&
	\int_{\Theta + \theta_{0}}^{\frac{\pi}{2}} \frac{d \theta}{ \sqrt{ \sinh( \rho ) } \, \sqrt{ \sinh( t' ) } \sqrt{ \cos( \Theta ) - \cos( \theta ) } },
\end{align*}
since $ \cosh( \rho )\, ( \cosh( t' ) - \cosh( r ) ) $ is positive and $ t' > r $. By examining the behavior of the function, we notice that
\begin{align*}
	&
	\cos( \Theta ) - \cos( \theta ) \geq 1 - \cos( \theta - \Theta ), \qquad \theta \in ( \Theta, \tfrac{\pi}{2} ).
\end{align*}
Using this gives the following upper bound for the previous integral:
\begin{align*}
	&
	\frac{1}{ \sqrt{ \sinh( \rho ) } \, \sqrt{ \sinh( t' ) } } \int_{\Theta + \theta_{0}}^{\frac{\pi}{2}} \frac{d \theta}{ \sqrt{ 1 - \cos( \theta - \Theta ) } } \leq \frac{1}{ \sqrt{ \sinh( \rho ) } \, \sqrt{ \sinh( t' ) } } \frac{ \frac{\pi}{2} }{ \sqrt{ 1 - \cos( \theta_{0} ) } }.
\end{align*}
This is bounded from above by
\begin{align*}
	&
	\frac{ \pi^{2} }{ 2 \sqrt{2}} \frac{1}{\sqrt{ \sinh( \rho ) \, \sinh( t' ) }} \frac{1}{\theta_{0}}
\end{align*}
using
\begin{align*}
	1 - \cos( x ) \geq \frac{2}{\pi^{2}} x^{2}, \qquad x \in ( -\pi, \pi ).
\end{align*}

Finally, we bound the third term of Expression \eqref{Expr: Angles split case 2}. It is bounded from above by
\begin{align*}
	&
	\int_{\frac{\pi}{2}}^{\pi} \frac{d \theta}{ \sqrt{ \sinh( \rho ) \, \sinh( t' ) \,  (-\cos( \theta ))  } } \leq \frac{3}{ \sqrt{ \sinh( \rho ) \, \sinh( t' ) } },	
\end{align*}
since $ \cosh( \rho )( \cosh( t' ) - \cosh( r ) ) $ and $ \sinh( \rho ) \sinh( t' ) \cos( \Theta ) $ are positive. The inequality is obtained by integrating $ \theta $ and then bounding the result from above by $ 3 $. 

The three terms of Expression \eqref{Expr: Angles split case 2} are bounded. Hence we have an upper bound for Expression \eqref{Expr: Angles split case 2}, and thus also for Expression \eqref{Expr: Distances and angles 2}, which is: 
\begin{align*}
	&
	\frac{2}{ \sqrt{ \cosh( \rho ) } } \int_{0}^{t'} \frac{\sinh( r ) \, \theta_{0} \, dr}{  \cosh( t' ) - \cosh( r ) } 
	 +  
	\frac{\pi^{2}}{ \sqrt{2} } \, \frac{1}{\sqrt{ \sinh( \rho ) \, \sinh( t' ) }} \int_{0}^{t'} \frac{ \sinh( r ) \, dr }{ \theta_{0} \, \sqrt{ \cosh( t' ) - \cosh( r ) } }
	\\
	& + \
	\frac{6}{ \sqrt{ \sinh( \rho ) \, \sinh( t' ) } } \int_{0}^{t'} \frac{\sinh( r ) \, dr}{\sqrt{ \cosh( t' ) - \cosh( r) }}.	
\end{align*}
We choose
\begin{align*}
	\theta_{0} = \min \left( \frac{\pi}{2} - \Theta, \left( \frac{ \cosh( t' ) - \cosh( r ) }{ \sinh( t' ) } \right)^{\alpha} \right),
\end{align*}
with $ \alpha \in ( 0, \frac{1}{2} ) $. The previous integral is then bounded from above by
\begin{align*}
	&
	\frac{2}{ \sqrt{ \cosh( \rho ) } \sinh( t' )^{\alpha} } \int_{0}^{t'} \frac{\sinh( r ) \, dr}{ ( \cosh( t' ) - \cosh( r ) )^{ 1 - \alpha } } 
	\\
	& + \  
\frac{\pi^{2}}{ \sqrt{2} } \frac{1}{\sqrt{ \sinh( \rho ) }  \sinh( t' )^{ \frac{1}{2} - \alpha } } \int_{0}^{t'} \frac{ \sinh( r ) \, dr }{ ( \cosh( t' ) - \cosh( r ) )^{ \frac{1}{2} + \alpha } }
	\\
	& + \
	\frac{6}{ \sqrt{ \sinh( \rho ) \, \sinh( t' ) } } \int_{0}^{t'} \frac{\sinh( r ) \, dr}{\sqrt{ \cosh( t' ) - \cosh( r ) }}.
\end{align*}
The upper bound was obtained by replacing $ \theta_{0} $ with
\begin{align*}
	&
	\left(  \frac{\cosh( t' ) - \cosh( r ) }{ \sinh( t' ) } \right)^{\alpha};
\end{align*}
this is possible since the middle term would vanish if
\begin{align*}
	&
	\left(  \frac{\cosh( t' ) - \cosh( r ) }{ \sinh( t' ) } \right)^{\alpha} \geq \frac{\pi}{2} - \Theta.
\end{align*}
Evaluating the integrals gives
\begin{align*}
	&
	\frac{2}{\alpha} \, \frac{ ( \cosh( t' ) - 1 )^{\alpha} }{ \sinh( t' )^{\alpha} } \, \frac{1}{\sqrt{ \cosh( \rho ) }} + \frac{\pi^{2}}{\sqrt{2} ( \frac{1}{2} - \alpha ) } \, \frac{ ( \cosh( t' ) - 1 )^{\frac{1}{2} - \alpha} }{ \sinh( t' )^{ \frac{1}{2}- \alpha } } \, \frac{1}{\sqrt{\sinh(\rho)}} + 6 \, \frac{ \sqrt{ \cosh( t' ) - 1 } }{ \sqrt{ \sinh( t' ) } } \, \frac{1}{\sqrt{ \sinh( \rho ) }}.
\end{align*}
Choosing $ \alpha = \frac{1}{4} $ yields the upper bound:
\begin{align*}
	&
	\frac{42}{\sqrt{ \sinh( \rho ) }}.
\end{align*}
This finishes the bound for the integration over $ D_{1} $ in Expression \eqref{Expr: D1 D2 split}. 

We proceed to bound the integration over $ D_{2} $ in Expression \eqref{Expr: D1 D2 split}. We write $ x $ using polar coordinates around $ z $, taking $ \theta = 0 $ in the direction $ [z,z']( z ) $. Hence $ x = \pi_{1}( \varphi_{R}( z, \theta ) ) $. $ r( R, \theta ) $ is $ d( x,z' ) $ in the polar coordinates. We notice that the integration domain of the angles is $ ( -\Psi, \Psi ) $, where $ \Psi $ is the angle between $ [z,z']( z ) $ and $ [z, w]( z ) $. Furthermore, the integration domain of the distances depends on the angle: $ R \in ( R( \theta ), t ) $, where $ R( \theta ) $ is the minimum distance required to reach the integration domain in the direction $ \theta $. Figure \ref{Fig: Sub2} illustrates this.

\begin{figure}[H]
        \centering
        \includegraphics[width=175px]{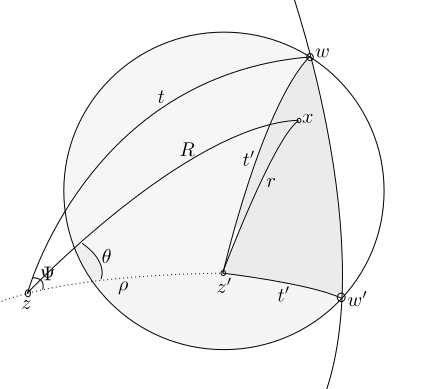}
	\caption{The integration domain $ D_{2} $ consists of points in the area surrounded by the segments $ [z', w] $ and $ [z',w'] $ and the arc between $ w $ and $ w' $ on $ \partial B( z,t ) $. }
    \label{Fig: Sub2}
\end{figure}

Hence the second term of Expression \eqref{Expr: D1 D2 split} equals
\begin{align}
	&
	\label{Expr: Sector}
	2 \int_{0}^{\Psi} \int_{R( \theta )}^{t} \frac{ \sinh( R ) }{ \sqrt{ \cosh( t ) - \cosh( R ) } \, \sqrt{ \cosh( t' ) - \cosh( r( R,\theta ) ) } } dR \, d\theta  ,
\end{align}
where only half of the angles are considered due to symmetry, giving the coefficient $ 2 $ in front. Using the hyperbolic law of cosines, we record two useful equalities:
\begin{align*}
	&
	\cosh( t' ) = \cosh( t ) \cosh( \rho ) - \sinh( t ) \sinh( \rho ) \cos( \Psi ),
	\\
	&
	\cosh( r( R,\theta ) ) = \cosh( R ) \cosh( \rho ) - \sinh( R ) \sinh( \rho ) \cos( \theta ).
\end{align*}
We remark that $ r( R,\theta ) $ grows as $ R $ grows. Hence we have the following inequality using the previous equalities:
\begin{align}
	&
	\label{Expr: t' - r ineq}
	\cosh( t' ) - \cosh( r( R,\theta ) ) \geq \sinh( t ) \, \sinh( \rho ) \, ( \cos( \theta ) - \cos( \Psi ) ).
\end{align}

Focus on the inner integral of Expression \eqref{Expr: Sector}. We acquire an upper bound by using the inequality of Expression \eqref{Expr: t' - r ineq}:
\begin{align*}
	&
	\int_{ R( \theta ) }^{ t } \frac{ \sinh( R ) \, dR }{ \sqrt{ \cosh( t ) - \cosh( R ) } \, \sqrt{ \sinh( \rho ) \, \sinh( t ) \, ( \cos( \theta ) - \cos( \Psi ) ) } } 
	\\
	& \leq \ 
	\frac{  \sqrt{ \cosh( t ) - 1 } }{ \sqrt{ \sinh( t ) } } \, \frac{1}{\sqrt{ \cos( \theta ) - \cos( \Psi ) }} \, \frac{1}{\sqrt{ \sinh( \rho )}}
	\\
	& \leq \
	\frac{1}{\sqrt{ \cos( \theta ) - \cos( \Psi ) }} \, \frac{1}{\sqrt{ \sinh( \rho ) }},
\end{align*}
where the first inequality is obtained by evaluating the integral over the enlarged domain $ (0,t) $.

We now have a bound for the inner integral of Expression \eqref{Expr: Sector}, giving the following upper bound for the whole expression:
\begin{align*}
	&
	\frac{2}{\sqrt{ \sinh( \rho ) }} \int_{0}^{\Psi} \frac{d\theta}{ \sqrt{ \cos( \theta ) - \cos( \Psi ) } } \leq \frac{2}{\sqrt{ \sinh( \rho ) }} \int_{0}^{\Psi} \frac{ d\theta }{ \sqrt{ \Psi - \theta } \, \sqrt{ \sin( \theta ) } } .
\end{align*}
The inequality follows from the mean value theorem. The previous expression is further bounded by
\begin{align*}
	&
	\sqrt{2} \sqrt{\pi} \, \frac{1}{\sqrt{ \sinh( \rho ) }} \int_{0}^{\Psi} \frac{ d\theta }{\sqrt{  \theta} \, \sqrt{ \Psi - \theta } },
\end{align*}
using that
\begin{align*}
	\sin( x ) \geq \frac{2}{\pi} x, \qquad x \in ( - \frac{\pi}{2}, \frac{\pi}{2} ).
\end{align*}
The integral is Euler's beta integral and equals $ \pi $, see \cite[Chp. 5.12.]{Loz03}. Hence the expression equals
\begin{align*}
	&
	\sqrt{2} \pi \sqrt{\pi} \, \frac{1}{\sqrt{ \sinh( \rho ) }}.
\end{align*}

We have now bounded the integration over $ D_{2} $ in Expression \eqref{Expr: D1 D2 split}.

Combining the bounds for $ D_{1} $ and $ D_{2} $, we get the final upper bound for $ F_{t,t',\rho} $:
\begin{align*}
	&
	( 42 + \sqrt{2} \pi \sqrt{\pi} ) \, \frac{1}{\sqrt{ \sinh( \rho ) }} \leq 50 \, \frac{1}{\sqrt{ \sinh( \rho ) }}.
\end{align*}
This concludes the proof of Lemma \ref{Lmm: Case 2}.

\end{proof}

We state the third lemma concerning the case of $ \rho \in ( t, t + t' ) $.

\begin{lemma}[]
    \label{Lmm: Case 3}
    Assume that $ t > \rho $. Then:
    \begin{align}
        \label{Expr: Big rho bound}
	F_{t,t',\rho} \leq C_{4} \, \frac{1}{ \sqrt{ \sinh( \rho ) } },
    \end{align}
    where $C_{4} = 38$.
\end{lemma}

The following figure, Figure \ref{Fig: Case 3}, illustrates the integration domain.

\begin{figure}[H]
    \centering
    \includegraphics[width=300px]{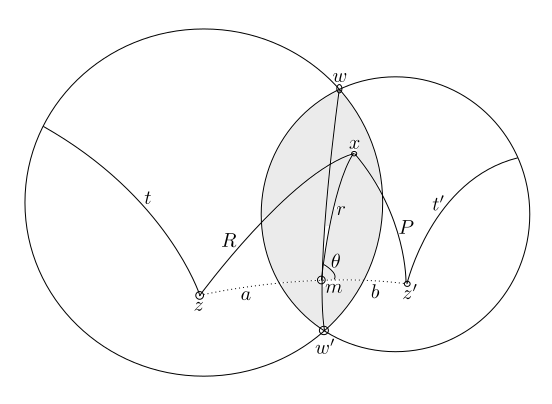}
    \caption{The case of $ \rho \in ( t, t + t' ) $. The integration domain $ B( z,t ) \cap B( z',t' ) $ is the intersection of two balls. The middle point $ z' $ of the smaller ball is not contained in the bigger ball. The integration domain is split into two parts by $ [w,w'] $. Notice that $ a + b = \rho $, the separation of $ z $ and $ z' $.}
    \label{Fig: Case 3}
\end{figure}

The proof is again based on splitting the integration domain into two. This time the splitting is done by $ [w,w'] $, which separates the domain into two sides: one closer to $ z $ and one closer to $ z' $. Here, $ w $ and $ w' $ are points of $ \partial B( z,t ) \cap \partial B( z', t' ) $. The technique used to bound both sides is the same and relies on hyperbolic geometry, in particular the hyperbolic law of cosines.

\begin{proof}[Proof of Lemma \ref{Lmm: Case 3}]

	The integration domain of $ F_{t,t',\rho} $ is the intersection of $B(z,t)$ and $B(z',t')$. Since $t < \rho+t'$, the intersection is not a ball. Let $w$ and $w'$ be the intersections of the boundaries of the balls, and let $m$ be the intersection of $[z,z']$ with $[w,w']$. We parametrize the domain using polar coordinates around $m$, with $\theta=0$ corresponding to the direction $[z,z']( m )$. Hence $ x = \pi_{1}( \phi_{r}( m, \theta ) ) $. To ease notation, we let $ | [z,m] | = a $, $ | [m,z'] | = b $, and $ | [m,w] | = h $. We denote $ d( z,x ) $ and $ d( z',x ) $ in polar coordinates by $ R = R( r,\theta ) $ and $ P = P( r,\theta ) $, respectively. Then $ F_{t,t',\rho} $ can be written as
$$
\int_0^{2\pi} \int_0^{r(\theta)} \frac{\sinh r \, dr \, d\theta}{\sqrt{\cosh(t)-\cosh(R(r,\theta))} \, \sqrt{\cosh(t')-\cosh(P(r,\theta))}},
$$
where $ r( \theta ) $ is the maximal distance from $ m $ in direction $ \theta $ within the integration domain. We split the integration into two parts:
\begin{align}
	&
	\nonumber
	\int_{-\frac{\pi}{2}}^{\frac{\pi}{2}} \int_0^{r(\theta)} \frac{\sinh r \, dr \, d\theta}{\sqrt{\cosh(t)-\cosh(R(r,\theta))} \, \sqrt{\cosh(t')-\cosh(P(r,\theta))}} 
	\\
	& + \
	\label{Expr: Side splitting}
	\int_{\frac{\pi}{2}}^{\frac{3\pi}{2}} \int_0^{r(\theta)} \frac{\sinh r \, dr \, d\theta}{\sqrt{\cosh(t)-\cosh(R(r,\theta))} \, \sqrt{\cosh(t')-\cosh(P(r,\theta))}}.
\end{align}
The first term integrates over the side closer to $ z' $, while the second integrates over the side closer to $ z $; these sides are separated by $ [w,w'] $.

For the first term of Expression \eqref{Expr: Side splitting}, the hyperbolic law of cosines gives:
\begin{align}
	&
	\label{Expr: cosh(t)}
	\cosh( t ) = \cosh( a ) \cosh( r( \theta ) ) + \sinh( a )\sinh( r( \theta ) ) \cos( \theta ),
	\\
	&
	\label{Expr: cosh(R)}
	\cosh( R ) = \cosh( a ) \cosh( r )+ \sinh( a ) \sinh( r ) \cos( \theta ),
	\\
	&
	\label{Expr: cosh(P)}
	\cosh( P ) = \cosh( b ) \cosh( r ) - \sinh( b ) \sinh( r ) \cos( \theta ),
\end{align}
where $ \cos( \theta ) > 0 $. Also, we have
\begin{align*}
	&
	\cosh( t )= \cosh( a ) \cosh( h ),
	\\
	&
	\cosh( t' ) = \cosh( b ) \cosh( h ),
\end{align*}
which implies
\begin{align*}
	&
	\cosh( t' ) = \frac{\cosh( b )}{\cosh( a )} \cosh( t ).
\end{align*}
Inserting Equality \eqref{Expr: cosh(t)} gives
\begin{align}
	&
	\label{Expr: cosh(t')}
	\cosh( t' ) = \cosh( b ) \cosh( r( \theta ) ) + \frac{\cosh( b )}{\cosh( a )} \sinh( a ) \sinh( r( \theta ) ) \cos( \theta ).
\end{align}
Lower bounds for the differences are
\begin{align*}
	&
	\cosh( t ) - \cosh( R ) \geq \cosh( a ) \left( \cosh( r( \theta ) ) - \cosh( r ) \right), 
	\\
	&
	\cosh( t' ) - \cosh( P ) \geq \cos( \theta ) \left( \frac{ \cosh( b ) }{\cosh( a ) } \sinh( a ) \sinh( r( \theta ) ) + \sinh( b )\sinh( r ) \right),
\end{align*}
and their product satisfies
\begin{align*}
	&
	( \cosh( t ) - \cosh( R ) )( \cosh( t' ) - \cosh( P ) ) \geq \cos( \theta ) \, ( \cosh( r( \theta ) ) - \cosh( r ) ) \, \sinh( r ) \, \sinh( \rho ).
\end{align*}

For the second term of Expression \eqref{Expr: Side splitting}, using the same approach (the roles of $ t,R $ and $ t',P $  change), we have the following lower bound:
\begin{align*}
	&
	( -1 ) \, \cos( \theta ) \, ( \cosh( r( \theta ) ) - \cosh( r ) ) \, \sinh( r ) \, \sinh( \rho ).
\end{align*}
Notice that $ -1 $ comes from the fact that $ \cos( x ) $ is negative when $ x \in ( \tfrac{\pi}{2}, \tfrac{3 \pi}{2} ) $. These bounds give the following upper bound for $ F_{t,t',\rho} $:
\begin{align}
	&
	\label{Expr: Tidy form}
	\frac{1}{\sqrt{ \sinh( \rho ) }} \int_{-\pi}^{\pi} \frac{1}{ \sqrt{| \cos( \theta ) |} } \int_{0}^{r( \theta )} \frac{\sinh( r )}{ \sqrt{ \cosh( r( \theta ) ) - \cosh( r ) } \, \sqrt{ \sinh( r ) } } \, dr \, d\theta.
\end{align}

For the inner integral, if \( r(\theta) < 2 \), the mean value theorem gives the following upper bound:
\begin{align*}
	&
	\int_{0}^{r( \theta )} \frac{1}{\sqrt{ r( \theta ) - r }} \, dr
	= 
	2\sqrt{ r( \theta ) } \leq 2 \sqrt{2}.
\end{align*}
If \( r(\theta) > 2 \), we split the inner integral of Expression \eqref{Expr: Tidy form} at \( c \in (0, r(\theta)) \) to write it as:
\begin{align}
	&
	\label{Expr: c split}
	\int_{0}^{c} \frac{\sinh( r )}{ \sqrt{ \cosh( r( \theta ) ) - \cosh( r ) } \, \sqrt{ \sinh( r ) } } \, dr \ + \ \int_{c}^{r( \theta )} \frac{\sinh( r )}{ \sqrt{ \cosh( r( \theta ) ) - \cosh( r ) } \, \sqrt{ \sinh( r ) } } \, dr.
\end{align}
c is specified later.

The first term of the previous expression has the following upper bound using the boundary of the integration domain:
\begin{align*}
\frac{1}{\sqrt{ \cosh( r(\theta) ) - \cosh( c ) }}
\int_{0}^{c} \frac{ e^{\tfrac{r}{2}} }{\sqrt{2}} \, dr
&=
\frac{2}{\sqrt{2}}
\frac{ e^{ \tfrac{c}{2} } - 1 }{\sqrt{ \cosh( r(\theta) ) - \cosh( c ) }}
\leq
\sqrt{2}
\frac{ \sinh\!\left( \tfrac{c}{2} \right) }
{ \sqrt{ \sinh(c)\,( r(\theta) - c ) } },
\end{align*}
where the inequality follows from the mean value theorem together with the bound for \( \sinh \).
Since
\begin{align*}
\frac{\sinh\!\left( \tfrac{x}{2} \right)}{\sqrt{\sinh(x)}} \leq 1,
\end{align*}
we conclude that the first term of Expression \eqref{Expr: c split} is bounded from above by
\begin{align*}
\frac{\sqrt{2}}{\sqrt{ r(\theta) - c }} .
\end{align*}

The second term of Expression \eqref{Expr: c split} is bounded from above by the following expression using the mean value theorem:
\begin{align*}
\int_{c}^{r(\theta)} \frac{1}{\sqrt{r(\theta) - r}} \, dr
= 2 \sqrt{r(\theta) - c}.
\end{align*}

With both terms bounded, Expression \eqref{Expr: c split} has the following upper bound by choosing \( c = r(\theta) - 1 \):
\begin{align*}
\frac{\sqrt{2}}{\sqrt{ r(\theta) - c }} + 2 \sqrt{ r(\theta) - c } = 2 + \sqrt{2}.
\end{align*}
Hence, Expression \eqref{Expr: Tidy form} is bounded by
\begin{align*}
\frac{2 + \sqrt{2}}{\sqrt{ \sinh(\rho) }}
\int_{-\pi}^{\pi} \frac{1}{\sqrt{|\cos(\theta)|}} \, d\theta
\leq
\frac{38}{\sqrt{ \sinh(\rho) }},
\end{align*}
where the last integral is bounded from above by \(11\). This completes the proof of Lemma \ref{Lmm: Case 3}.

\end{proof}


\section{Proof of main theorems}\label{Sec: Main theorems}

In this section, we give proofs of the main theorems of this paper, Theorems \ref{Thm: Large-scale analogue}, \ref{Thm: Probabilistic main theorem}, and \ref{Thm: Quantitative main theorem}. We first prove the quantitative main theorem, Theorem \ref{Thm: Quantitative main theorem}, as it is needed for the proofs of Theorems \ref{Thm: Large-scale analogue} and \ref{Thm: Probabilistic main theorem}. For clarity of the proof, we have presented auxiliary results in separate sections: Sections \ref{Sec: Spectral data}, \ref{Sec: Propagators}, \ref{Sec: Geometric data}, and \ref{Sec: Weight function}.

After proving Theorem \ref{Thm: Quantitative main theorem}, we prove the limit-form main theorem, Theorem \ref{Thm: Large-scale analogue}. This result follows straightforwardly from Theorems \ref{Thm: Quantitative main theorem} and \ref{Thm: Weyl} using Properties \textbf{(BSC)}, \textbf{(EXP)}, and \textbf{(UND)}, where Theorem \ref{Thm: Weyl} is the large-scale Weyl law.

Proving Theorem \ref{Thm: Probabilistic main theorem} relies on Theorems \ref{Thm: Quantitative main theorem} and \ref{Thm: Random surface theory}. The proof is very similar to that of Theorem \ref{Thm: Large-scale analogue}, but as the probabilistic setting does not impose Properties \textbf{(BSC)}, \textbf{(EXP)}, and \textbf{(UND)}, we need Theorem \ref{Thm: Random surface theory} to record some key results in random surface theory to replace these properties. 

We proceed to prove Theorem \ref{Thm: Quantitative main theorem} and restate the theorem for the convenience of the reader.

\begin{theorem*}

Let $ X $ be a compact connected hyperbolic surface with Laplace--Beltrami operator $ -\Delta_{X} $ that has eigenvalues
\begin{align*}
	0 = \lambda_{0} < \lambda_{1} \leq \dots \leq \lambda_{j} \leq \dots \xrightarrow{j \to \infty} \infty
\end{align*}
together with a corresponding orthonormal eigenbasis $ \{ \psi_{j} \} $ over $ L^{2}( X ) $. Define associated values of the eigenvalues as follows: 
\begin{align*}
	\rho_{j} = \sqrt{ \lambda_{j} - \tfrac{1}{4} }.
\end{align*}
Let $I \subset (0,\infty)$ be a compact interval and let $a$ be a mean $0$ function on $ X $ acting as a multiplication operator. Let 
\[
0 < \delta < \tfrac{2}{9} \min(I),
\qquad
T\delta = \tfrac{\pi}{2},
\]
and fix $\tau \in \mathbb{R}$. Then
\begin{align*}
\sum_{j:\rho_j \in I}
\ \sum_{\substack{ k: \rho_{k} \in I, \\ k \neq j, \\ |\rho_j - \rho_k - \tau| < \delta}}
\left| \langle \psi_j, a \, \psi_k \rangle \right|^2
\leq C_{8} \, \max( I )^{4} \, \frac{1}{T \, \beta^{3}} \, \left( \| a \|_{2}^{2} \, + \,  \| a \|_{\infty}^{2} \, | X \setminus X( 4T ) | \frac{e^{6 T}}{ \operatorname{InjRad}_{X} } \right),
\end{align*}
where $ | X \setminus X( 4T ) | $ is the volume of the points of $ X $ with injectivity radius smaller than $ 4T $, and where $\beta = \beta( \lambda_{1} )$, defined by
\begin{align*}
	\beta( x ) = 
	\begin{cases}	
		1 - \sqrt{1 - 4x}, \qquad &\text{ if } x \leq \frac{1}{4},
		\\
		1, \qquad &\text{ if } x > \frac{1}{4}.
	\end{cases}
\end{align*}
The coefficient appearing is $C_{8} = 5120000 \pi^{2}$.

\end{theorem*}

The proof presented here is clear since the technical details are dealt with in different sections; see Sections \ref{Sec: Spectral data}, \ref{Sec: Propagators}, \ref{Sec: Geometric data}, and \ref{Sec: Weight function}. The core idea is to split the problem into a geometric side and a spectral side using a well-chosen propagator. These two sides can then be dealt with separately. 

\begin{proof}[Proof of Theorem \ref{Thm: Quantitative main theorem}]
Define functions
\[
h_t(x) = x^{-1}\sin(tx),
\qquad
h_{t,\tau}(x) = \cos(t\tau)\, h_t(x),
\]
and operators
\[
P_t = h_t\left(\sqrt{-\Delta_X - \tfrac14}\right),
\qquad
P_{t,\tau} = h_{t,\tau}\left(\sqrt{-\Delta_X - \tfrac14}\right).
\]
Then, due to the way the operators are defined, we have
$$P_t \psi_j = h_t(\rho_j)\psi_j, \qquad P_{t,\tau}\psi_j = h_{t,\tau}(\rho_j)\psi_j.$$
Using the above, we can write
\[
|\langle \psi_j, a\psi_k \rangle|
=
\left|
\frac{1}{T}\int_0^T h_{t,\tau}(\rho_j)\, h_t(\rho_k)\, dt
\right|^{-1}
\left|
\left\langle \psi_j,\,
\frac{1}{T}\int_0^T P_{t,\tau}^* a P_t \, dt
\,\psi_k
\right\rangle
\right|.
\]

The left factor containing the spectral data is bounded using Proposition \ref{Prp: Spectral data}, which can be applied since $I$ is compact in $(0,\infty)$, $ T \delta = \tfrac{\pi}{2} $, and $\delta < \tfrac{2}{9} \min(I)$:
\[
\left|
\frac{1}{T}\int_0^T h_{t,\tau}(\rho_j) h_t(\rho_k)\, dt
\right|^{-1} \leq 8 \pi \rho_{j} \rho_{k}  \leq  8 \pi \max( I )^{2}.
\]

Thus,
\begin{align*}
\sum_{j:\rho_j \in I} 
\ \sum_{\substack{ k: \rho_{k} \in I, \\ k \neq j , \\ |\rho_j - \rho_k - \tau|<\delta}}
|\langle\psi_j,a\psi_k\rangle|^2
& \leq
64 \pi^{2} \, \max(I)^4 \, \sum_{j:\rho_j \in I} 
\ \sum_{\substack{ k: \rho_{k} \in I, \\ k \neq j , \\ |\rho_j - \rho_k - \tau|<\delta}}
\left| \left\langle\psi_j, \frac{1}{T} \int_{0}^{T} P_{t,\tau}^{*}a P_{t} \, dt \psi_k \right\rangle \right|^2
\\
& \leq
64 \pi^{2} \, \max( I )^{4} \, \left\|
\frac{1}{T}\int_0^T P_{t,\tau}^* a P_t\, dt
\right\|_{\mathrm{HS}}^2,
\end{align*}
where the Hilbert-Schmidt norm bound is obtained when considering all basis elements. The Hilbert-Schmidt norm contains the geometric data of the problem.

Next, we apply Proposition \ref{Prp: Geometric data} to bound the Hilbert-Schmidt norm:
$$
\left\|
\frac{1}{T}\int_0^T P_{t,\tau}^* a P_t\, dt
\right\|_{\mathrm{HS}}^2
\leq
\frac{C_{7}}{T \, \beta^{3}} \, \left( \| a \|_{2}^{2} \, + \, \| a \|_{\infty}^{2} \, | X \setminus X( 4T ) | \frac{e^{6 T}}{ \operatorname{InjRad}_{X} } \right),
$$
where $ | X \setminus X( 4T ) | $ denotes the volume of the points of $ X $ with injectivity radius smaller than $ 4T $. Also, $ C_{7} = 80000$. Hence we can conclude that
\begin{align*}
	&
	\sum_{j: \rho_j \in I}
\ \sum_{\substack{ k: \rho_{k} \in I, \\ k \neq j, \\ |\rho_j - \rho_k - \tau| < \delta}}
\left| \langle \psi_j, a \psi_k \rangle \right|^2 \leq C_{8} \, \max( I )^{4} \, \frac{1}{T \, \beta^{3}} \, \left( \| a \|_{2}^{2} \, + \,  \| a \|_{\infty}^{2} \, | X \setminus X( 4T ) | \frac{e^{6 T}}{ \operatorname{InjRad}_{X} } \right),
\end{align*}
where $C_{7} \cdot 64 \pi^{2} = 5 120 000 \pi^{2} = C_{8}$.

\end{proof}

We now turn to Theorem \ref{Thm: Large-scale analogue}, the limit-form main theorem. We also restate the theorem for convenience.

\begin{theorem*}	
	
Let $ (X_{n}) $ be a sequence of compact connected hyperbolic surfaces with corresponding Laplace{\allowbreak}--Beltrami operators \( (-\Delta_{X_{n}}) \). The Laplacian $ -\Delta_{X_{n}} $ has eigenvalues
	\begin{align*}
		0 = \lambda_{0}^{( n )} < \lambda_{1}^{( n )} \leq \dots \leq \lambda_{j}^{( n )} \leq \dots \xrightarrow{j \to \infty} \infty
	\end{align*}
together with a corresponding orthonormal eigenbasis $ \{ \psi_{j}^{( n )} \} $ over $ L^{2}( X_n ) $. Define associated values of the eigenvalues as follows:
\begin{align*}
	\rho_{j}^{( n )} = \sqrt{ \lambda_{j}^{( n )} - \tfrac{1}{4} }.
\end{align*}
Assume that the sequence has the following properties:
\begin{itemize}
    \item[\textbf{(BSC)}] \textbf{Benjamini--Schramm convergence:}
    for every \( R > 0 \),
    $$
        \lim_{n \to \infty} 
        \frac{
            \big| \{ x \in X_{n} : \operatorname{InjRad}_{X_{n}}(x) < R \} \big|
        }{
            |X_{n}|
        } = 0,
    $$
    \item[\textbf{(EXP)}] \textbf{Expander property:}
    the sequence admits a uniform lower bound on the spectral gap,
    \item[\textbf{(UND)}] \textbf{Uniform discreteness:}
    the sequence admits a uniform lower bound on the injectivity radius.
\end{itemize}
Let $ I \subset ( 0,\infty ) $ be a compact interval and let $ ( a_{n} ) $ be a sequence of uniformly bounded functions on $ ( X_{n} ) $, each acting as a multiplication operator. Then the following properties hold:
\begin{itemize}
    \item[(i)] 
    $$
        \lim_{n \to \infty} \
        \frac{1}{ \# \{ j  :  \rho_{j}^{( n )} \in I \} }
        \sum_{j : \rho_{j}^{(n)} \in I}
        \left|
            \left\langle \psi_{j}^{(n)}, a_{n} \psi_{j}^{(n)} \right\rangle
            - \frac{1}{| X_{n} |} \int_{X_{n}} a_n \, dx
        \right|^{2} = 0,
    $$
    \item[(ii)] 
    For every \( \epsilon > 0 \) there exists \( \delta(\epsilon) > 0 \) such that
    \[
        \limsup_{n \to \infty} \
        \frac{1}{ \# \{ j  :  \rho_{j}^{( n )} \in I \} }
        \sum_{j : \rho_{j}^{(n)} \in I}
        \sum_{\substack{k: \rho_{k}^{( n )} \in I, \\ k \neq j, \\ | \rho_{j}^{(n)} - \rho_{k}^{(n)} | < \delta( \epsilon )}}
        \left| \left\langle \psi_{j}^{(n)}, a_{n} \psi_{k}^{(n)} \right\rangle \right|^{2}
        < \epsilon,
    \]
    \item[(iii)]
    For every \( \epsilon > 0 \) there exists \( \delta(\epsilon) > 0 \) such that for all \( \tau \in \mathbb{R} \),
    \[
        \limsup_{n \to \infty} \
        \frac{1}{ \# \{  j : \rho_{j}^{( n )} \in I \} }
        \sum_{j: \rho_{j}^{(n)} \in I}
        \sum_{\substack{k : \rho_{k}^{( n )} \in I, \\ k \neq j , \\ | \rho_{j}^{(n)} - \rho_{k}^{(n)} - \tau | < \delta( \epsilon )}}
        \left| \left\langle \psi_{j}^{(n)}, a_{n} \psi_{k}^{(n)} \right\rangle \right|^{2}
        < \epsilon.
    \]
\end{itemize}

\end{theorem*}

Properties \emph{(ii)} and \emph{(iii)} remain to be proven; Property \emph{(i)} was established by Le Masson and Sahlsten (see Theorem 1.1 \cite{LS17}). We remark that the proof presented here for Properties \emph{(ii)} and \emph{(iii)} could be adapted with slight modifications to also include Property \emph{(i)}, thus giving a new proof of the large-scale quantum ergodicity theorem of Le Masson and Sahlsten. 

Passing to the limit and thus acquiring a small upper bound $ \epsilon $ is based on carefully using Properties \textbf{(BSC)}, \textbf{(EXP)}, and \textbf{(UND)} together with the large-scale Weyl law, which asymptotically relates the volume of the surface with the number of eigenvalues in an energy window as we proceed through the sequence of surfaces. Before the proof, we state the large-scale Weyl law (see \cite[Sec.9]{LS17}). 

\begin{theorem}[Large-scale Weyl law]
	\label{Thm: Weyl}
Let $(X_n)$ be a sequence of compact connected hyperbolic surfaces satisfying Property \textbf{(BSC)} stated in Theorem \ref{Thm: Large-scale analogue}. $ X_{n} $ has Laplace--Beltrami operator $ -\Delta_{X_{n}} $ with eigenvalues
\begin{align*}
	0 = \lambda_{0}^{( n )} < \lambda_{1}^{( n )} \leq \dots \leq \lambda_{j}^{( n )} \leq \dots \xrightarrow{j \to \infty} \infty
\end{align*}
together with a corresponding orthonormal eigenbasis $ \{ \psi_{j}^{( n )} \} $ over $ L^{2}( X_{n} ) $.
Define associated values of the eigenvalues as follows:
\begin{align*}
	\rho_{j}^{( n )} = \sqrt{ \lambda_{j}^{( n )} - \tfrac{1}{4} }.
\end{align*}
Let $ I \subset ( 0, \infty ) $ be a compact interval and let
\begin{align*}
	J = \left \{ x^{2} + \tfrac{1}{4} :  x \in I \right \}.
\end{align*}
Then
$$
\frac{ \# \{ j  :  \rho_{j}^{( n )} \in I \} }{|X_n|} = \frac{ \# \{  j  :  \lambda_{j}^{( n )} \in J \} }{ | X_{n} | } \xrightarrow{n \to \infty} C,
$$
where $ C $ is a constant depending on the sequence of surfaces and the interval $ I $.

\end{theorem}

We include both $ \lambda_{j}^{( n )} $ and $ \rho_{j}^{( n )} $ since Le Masson and Sahlsten proved their theorem using the $ \lambda_{j}^{( n )} $ approach. However, the theorem works just as well for $ \rho_{j}^{( n )} $. We especially remark that the large-scale Weyl law follows from the Benjamini--Schramm convergence property, Property \textbf{(BSC)}.

We proceed to prove Theorem \ref{Thm: Large-scale analogue}.

\begin{proof}[Proof of Theorem~\ref{Thm: Large-scale analogue}]
	We prove Properties \emph{(ii)} and \emph{(iii)} simultaneously. Let $\epsilon>0$. We aim to find $ \delta = \delta(\epsilon)>0$ such that, for all $\tau \in \mathbb{R}$,
\begin{align}
\label{Expr: Main theorem}
\limsup_{n\to\infty} \
\frac{1}{N(X_n,I)}
\sum_{j: \rho_j^{(n)} \in I}
\ \sum_{\substack{k:\rho_k^{(n)} \in I, \\ k \neq j, \\ | \rho_j^{(n)} - \rho_k^{(n)} - \tau| < \delta}}
\left|
\langle \psi_j^{(n)}, a_n \psi_k^{(n)} \rangle
\right|^2 
< \epsilon,
\end{align}
where
\begin{align*}
	N( X_{n}, I ) =  \# \{ j  :   \rho_{j}^{( n )} \in I \}.
\end{align*}
Let us define the following function:
\begin{align*}
	\bar{a}_{n} = a_{n} - \frac{1}{| X_{n} |} \int_{X_{n}} a_{n} \, dx.
\end{align*}
$ \bar{a}_{n} $ is a mean-zero normalization of $ a_{n} $. Due to orthogonality of $ \psi_{j}^{( n )} $ and $ \psi_{k}^{( n )} $, we have
\begin{align*}
	\langle \psi_{j}^{( n )} , a_{n} \psi_{k}^{( n )} \rangle = \langle \psi_{j}^{( n )}, \bar{a}_{n} \psi_{k}^{( n )} \rangle.
\end{align*}
Hence we have the following equal expression for the left-hand side of Expression \eqref{Expr: Main theorem}:
\begin{align*}
	\limsup_{n\to\infty} \
\frac{1}{N(X_n,I)}
\sum_{j: \rho_j^{(n)} \in I}
\ \sum_{\substack{k: \rho_k^{(n)} \in I, \\ k \neq j, \\ | \rho_j^{(n)} - \rho_k^{(n)} - \tau| < \delta}}
\left|
\langle \psi_j^{(n)}, \bar{a}_n \psi_k^{(n)} \rangle
\right|^2
\end{align*}

We can apply Theorem \ref{Thm: Quantitative main theorem} to the double sum, given that we assume $ \delta \in ( 0, \tfrac{2}{9} \min( I ) ) $. We get the following upper bound:
\begin{align}
	\label{Expr: Before properties}
	\limsup_{n \to \infty} \ \frac{1}{N( X_{n}, I )} \, C_{8} \, \max( I )^{4} \, \frac{1}{T \, \beta( \lambda_{1}^{(n)} )^{3}} \, \left( \| \bar{a}_{n} \|_{2}^{2} \, + \, \| \bar{a}_{n} \|_{\infty}^{2} \, | X_{n} \setminus X_{n}( 4T ) | \frac{e^{6 T}}{ \operatorname{InjRad}_{X_{n}} } \right),
\end{align}
where $ | X_{n} \setminus X_{n}( 4T ) | $ is the volume of the points of $ X_{n} $ with injectivity radius smaller than $ 4T $, $ T \delta = \tfrac{\pi}{2} $, and
\begin{align*}
	\beta( x ) =
	\begin{cases}	
		1 - \sqrt{1 - 4x}, \qquad & \text{if } x \leq \frac{1}{4},
		\\
		1, \qquad & \text{if } x > \frac{1}{4}.
	\end{cases}
\end{align*}
The constant appearing is $C_{8} = 5 120 000 \pi^{2}$.

Then we apply the properties of the sequence to simplify the previous expression. First we note that 
$$ \| \bar{a}_{n} \|_{2} \leq \sqrt{ | X_{n} | } \, \| \bar{a}_{n} \|_{\infty}. $$ 
Since $ ( a_{n} ) $ is uniformly bounded, we have 
$$ \| \bar{a}_{n} \|_{\infty} \leq 2 \| a_{n} \| \leq 2 \, a_{max}, $$
where $ a_{max} $ is the universal bound for $ ( a_{n} ) $.
Property \textbf{(EXP)} implies a positive universal lower bound for the spectral gaps, which ensures $ \beta( \lambda_{1}^{( n )} ) $ has a universal positive lower bound. Denote this by $ \beta_{min} $. Finally, Property \textbf{(UND)} implies a universal positive lower bound for the injectivity radii. Denote this by $ \operatorname{InjRad}_{min} $. Using these observations gives the following upper bound for Expression \eqref{Expr: Before properties}:
\begin{align*}
	&
	\limsup_{n \to \infty} \ 2 \cdot C_{8} \, \frac{ \max( I )^{4} \, a_{max}^{2} }{N( X_{n}, I ) \, T \, \beta_{min}^{3}  }\left( | X_{n} | + \frac{ | X_{n} \setminus X_{n}( 4T ) | \, e^{6T} }{ \operatorname{InjRad}_{min} } \right)
	\\
	& = \
	\frac{2 C_{8} \, \max( I )^{4} \, a_{max}^{2}  }{ T \, \beta_{min}^{3} }\left( \limsup_{n \to \infty} \left(  \frac{| X_{n} |}{ N( X_{n}, I ) } \right) +  \frac{ e^{6T} }{ \operatorname{InjRad}_{min} } \, \limsup_{n \to \infty}\left( \frac{ | X_{n} \setminus X_{n}( 4T ) | }{ N( X_{n}, I ) } \right) \right),	
\end{align*}
where the equality follows from splitting the limit supremum into two terms.

Consider the second limit supremum inside the parentheses. Due to the large-scale Weyl law, Theorem \ref{Thm: Weyl}, $ N( X_{n}, I ) $ grows asymptotically like the volume $ | X_{n} | $. Hence the ratio inside the limit supremum approaches $ 0 $ since Property \textbf{(BSC)} forces the volume of points with small injectivity radius relative to the total volume to go to zero. Therefore, the previous expression equals
\begin{align*}
	\frac{2 C_{8} \, \max( I )^{4} \, a_{max}^{2}  }{ T \, \beta_{min}^{3} } \, \limsup_{n \to \infty} \left(  \frac{| X_{n} |}{ N( X_{n}, I ) } \right).
\end{align*}
Using the large-scale Weyl law, the ratio approaches some constant $ C_{9} $. Thus we have
\begin{align*}
	\frac{2 C_{8} \, C_{9} \, \max( I )^{4} \, a_{max}^{2}  }{ T \, \beta_{min}^{3} }
	=
	\frac{4 \, C_{8} \, C_{9} \, \delta \, \max( I )^{4} }{ \beta_{min}^{3} \, \pi },
\end{align*}
where the equality uses $ T  \delta = \tfrac{\pi}{2} $. If we choose
$$ \delta =  \min \left( \epsilon \cdot \left( \frac{4 \, C_{8} \, C_{9} \, \max( I )^{4} }{ \beta_{min}^{3} \, \pi } \right)^{-1} , \tfrac{2}{9} \, \min( I )  \right),$$
then the previous expression is bounded by $ \epsilon $. We assume this choice in the beginning of the proof to conclude the proof.

\end{proof}

In the end, we prove the probabilistic theorem, Theorem \ref{Thm: Probabilistic main theorem}. We restate the theorem for convenience.

\begin{theorem*}
	
	Let $ ( X_{g} ) $ be a sequence of $ \mathbb{P}_{g} $-random surfaces, where $ \mathbb{P}_{g} $ is a Weil--Petersson probability distribution on the moduli space $ \mathcal{M}_{g} $, and where $ g $ is the genus. Let $-\Delta_{g}$ be the Laplace--Beltrami operator of $ X_{g} $ with eigenvalues 
	\begin{align*}
		0 = \lambda_{0}^{( g )} < \lambda_{1}^{( g )} \leq \dots \leq \lambda_{j}^{( g )} \leq \dots \xrightarrow{j \to \infty} \infty
	\end{align*}
	together with a corresponding orthonormal eigenbasis $ \{ \psi_{j}^{( g )} \} $ over $ L^{2}( X_{g} ) $. Define associated values of the eigenvalues as follows:
	\begin{align*}
		&
		\rho_{j}^{( g )} = \sqrt{ \lambda_{j}^{( g )} - \tfrac{1}{4} }.
	\end{align*}
	Let $ I \subset ( 0,\infty ) $ be a compact interval, and let $ ( a_{g} ) $ be a sequence of uniformly bounded functions on $ X_{g} $, each acting as a multiplication operator. Then Properties \emph{(i)}--\emph{(iii)} of Theorem \ref{Thm: Large-scale analogue} hold with high probability as the genus $ g $ grows to infinity.
	
\end{theorem*}

Properties \emph{(ii)} and \emph{(iii)} remain to be proven; Property \emph{(i)} was already proven by Le Masson and Sahlsten \cite{LS24}. The proof is very close to that of Theorem \ref{Thm: Large-scale analogue}. Now, we do not assume Properties \textbf{(BSC)}, \textbf{(EXP)}, and \textbf{(UND)}, so these need to be replaced. The following theorem records analogous tools acquired from the theory of Weil--Petersson random surfaces.

\begin{theorem}[]
	\label{Thm: Random surface theory}
	Let $ X $ be a $ \mathbb{P}_{g} $-random surface, where $ \mathbb{P}_{g} $ is the Weil--Petersson probability distribution on the moduli space $ \mathcal{M}_{g} $ with genus $ g $. It has eigenvalues
\begin{align*}
	&
	0 = \lambda_{0} < \lambda_{1}  \leq  \dots \leq \lambda_{j} \leq \dots \xrightarrow{j \to \infty} \infty
\end{align*}
together with associated values
\begin{align*}
	\rho_{j} = \sqrt{ \lambda_{j} - \tfrac{1}{4} }.
\end{align*}
Then $ X $ has the following properties with probability $ 1 - o( 1 )  $ for large genus:
\begin{itemize}
	\item [(a)] $$ \operatorname{InjRad}_{X} \geq g^{ - \frac{1}{24} } \, \log( g )^{ \frac{9}{16} } ,$$
	\item [(b)] $$  \frac{ | \{ x \in X : \operatorname{InjRad}_{X}( x ) < \frac{1}{6} \, \log( g ) \} | }{| X |} = \mathcal{O}( g^{-\frac{1}{4}} ) ,$$
	\item [(c)] $$ \lambda_{1} \geq \frac{1}{4} -  \eta ,$$
\end{itemize}
where $ \eta \in ( 0,\frac{1}{4} ) $. Additionally,
\begin{itemize}
	\item [(d)] $$ \# \{ j : \rho_{j} \in I \} = \# \{ j : \lambda_{j} \in J \} = | X | \cdot \mathcal{O} \left( \max( J ) - \min( J ) + \sqrt{ \frac{ \max( J ) + 1 }{ \log( g ) } } \right) ,$$
\end{itemize}
where $ I \subset ( 0,\infty ) $ is a compact interval and
\begin{align*}
	J = \left \{ x^{2} + \tfrac{1}{4} : x \in I \right \}.
\end{align*}

\end{theorem}
The theorem is taken from two articles: \cite{Mon22, HMT25}. Properties \emph{(a)}, \emph{(b)}, and \emph{(d)}, together with the rates, are taken from \cite{Mon22}, and Property \emph{(c)}, together with the rate, is taken from \cite{HMT25}. When comparing the proof of Theorem \ref{Thm: Probabilistic main theorem} with the proof of Theorem \ref{Thm: Large-scale analogue}, Property \emph{(a)} replaces Property \textbf{(UND)}, Property \emph{(b)} replaces Property \textbf{(BSC)}, and Property \emph{(c)} replaces Property \textbf{(EXP)}. Additionally, Property \emph{(d)} is a suitably modified version of the large-scale Weyl law of Le Masson and Sahlsten, applied in this probabilistic setting.

\begin{proof}[Proof of Theorem \ref{Thm: Probabilistic main theorem}]
	
	We prove Properties \emph{(ii)} and \emph{(iii)} simultaneously. Let $ \epsilon > 0 $. We aim to find $ \delta = \delta( \epsilon ) $ such that for every $ \tau \in \mathbb{R}$, we have
	\begin{align}
		&
		\label{Expr: Probabilistic main theorem}
		\limsup_{g \to \infty} \ \frac{1}{N( X_{g}, I )} \sum_{ j: \rho_{j}^{( g )} \in I  }^{} \sum_{ \substack{ k : \rho_{k}^{( g )} \in I, \\ k \neq j , \\ | \rho_{j}^{( g )} - \rho_{k}^{( g )} - \tau | < \delta }  }^{} \left| \langle \psi_{j}^{( g )}, a_{g} \psi_{k}^{( g )} \rangle \right|^{2} < \epsilon,
	\end{align}
	where 
	\begin{align*}
		N( X_{g}, I ) = \# \{ j : \rho_{j}^{( g )} \in I \}.
	\end{align*}
	Let us define the following function:
	\begin{align*}
		\bar{a}_{g} = a_{g} - \frac{1}{| X_{g} |} \int_{X_{g}} a_{g} \, dx.
	\end{align*}
	$ \bar{a}_{g} $ is a mean-zero normalization of $ a_{g} $. Due to the orthogonality of $ \psi_{j}^{( g )} $ and $ \psi_{k}^{( g )} $, we have
	\begin{align*}
		&
		\langle \psi_{j}^{( g )}, a_{g} \psi_{k}^{( g )} \rangle = \langle \psi_{j}^{( g )}, \bar{a}_{g} \psi_{k}^{( g )} \rangle.
	\end{align*}
	Hence the left-hand side of Expression \eqref{Expr: Probabilistic main theorem} equals
	\begin{align*}
		&
		\limsup_{g \to \infty} \ \frac{1}{N( X_{g}, I )} \sum_{ j: \rho_{j}^{( g )} \in I  }^{} \sum_{ \substack{ k : \rho_{k}^{( g )} \in  I , \\ k \neq j , \\ | \rho_{j}^{( g )} - \rho_{k}^{( g )} - \tau | < \delta }  }^{} \left| \langle \psi_{j}^{( g )}, \bar{a}_{g} \psi_{k}^{( g )} \rangle \right|^{2}.
	\end{align*}
	
	We can apply Theorem \ref{Thm: Quantitative main theorem} to the double sum, given that we assume $ \delta \in ( 0, \tfrac{2}{9} \min( I ) ) $. We get the following upper bound for the previous expression:
	\begin{align*}
		&
		\limsup_{g \to \infty} \ \frac{1}{N( X_{g}, I )} \, C_{8} \, \max( I )^{4} \, \frac{1}{T \, \beta( \lambda_{1}^{( g )} )^{3}} \, \left( \| \bar{a}_{g} \|_{2}^{2} +  \| \bar{a}_{g} \|_{\infty}^{2} \, | X_{g} \setminus X_{g}( 4T ) | \, \frac{e^{6T}}{ \operatorname{InjRad}_{X_{g}} } \right),	
	\end{align*}
	where $ | X_{g} \setminus X_{g}( 4T ) | $ is the volume of the points of $ X_{g} $ with injectivity radius less than $ 4T $, $ T \delta = \tfrac{\pi}{2} $, and where
	\begin{align*}
		\beta( x ) =
		\begin{cases}	
			1- \sqrt{1 - 4x}, \qquad & \text{ if } x \leq \frac{1}{4},
			\\
			1, \qquad & \text{ if } x > \frac{1}{4}.
		\end{cases}
	\end{align*}
	The constant appearing is $C_{8} = 5 120 000 \pi^{2}$.
	
	First, notice that 
	$$ \| \bar{a}_{g} \|_{2} \leq \sqrt{| X_{g} |} \, \| \bar{a}_{g }\|_{\infty} .$$ 
	Since the sequence $ ( a_{g} ) $ is uniformly bounded, we have
	\begin{align*}
		\| \bar{a}_{g}\|_{\infty} \leq 2 \| a_{g} \|_{\infty} \leq 2 a_{max},
	\end{align*}
	where $ a_{max} $ is the universal bound for $ ( a_{g} ) $. Hence the previous limit supremum expression is bounded from above by
	\begin{align}
		&
		\label{Expr: Before probability}
		2 \cdot C_{8} \, \frac{ a_{max}^{2} }{ T } \, \limsup_{g \to \infty}  \frac{1}{ N( X_{g}, I ) \, \beta( \lambda_{1}^{( g )} )^{3} } \left( | X_{g} | +  | X_{g} \setminus X_{g}( 4T ) | \, \frac{ e^{6T} }{ \operatorname{InjRad}_{X_{g}} } \right).
	\end{align}
	By using Theorem \ref{Thm: Random surface theory}, we have 	
	\begin{align*}
		&
		\frac{1}{ N( X_{g}, I ) \, \beta( \lambda_{1}^{( g )} )^{3} } \left( | X_{g} | + | X_{g} \setminus X_{g}( 4T ) | \, \frac{ e^{6T} }{ \operatorname{InjRad}_{X_{g}} } \right)
		\\
		& \leq \
		\frac{| X_{g} |}{ | X_{g} | \, \mathcal{O}\left( \max( J ) - \min( J ) + \sqrt{ \frac{ \max( J  ) + 1 }{ \log( g ) } } \right) \, \left( 1 - \sqrt{ 1 - 4( \tfrac{1}{4} - \eta ) } \right)^{3}} \, \left( 1 + \mathcal{O}( g^{ - \frac{1}{4} } ) \, \frac{e^{6T}}{ g^{ - \frac{1}{24}  } \, \log( g )^{ \frac{9}{16} } } \right),
	\end{align*}
	with probability $ 1 - o( 1 ) $, where $ \eta \in ( 0, \tfrac{1}{4} ) $ and
	\begin{align*}
		J = \left \{ x^{2} + \tfrac{1}{4} : x \in I \right \}.
	\end{align*}
	The previous expression has a tidier upper bound:
	\begin{align*}
		&
		\frac{ 1 }{ \left( 1 - \sqrt{ 4 \eta } \right)^{3} \, \mathcal{O}( \max( J ) - \min( J ) ) } \, \left( 1 +  e^{6T} \, \frac{ \mathcal{O}( g^{ -\frac{1}{4} } ) }{ g^{ - \frac{1}{24} } \, \log( g )^{ \frac{9}{16} } }  \right).
	\end{align*}
	This means that when passing to the high-genus limit, we have
	\begin{align*}
		&
		\frac{ 1 }{ \left( 1 - \sqrt{ 4 \eta } \right)^{3} \, \mathcal{O}( \max( J ) - \min( J ) ) } \, \left( 1 +  e^{6T} \, \frac{ \mathcal{O}( g^{ -\frac{1}{4} } ) }{ g^{ - \frac{1}{24} } \, \log( g )^{ \frac{9}{16} } }  \right) \xrightarrow{g \to \infty} C_{10}
	\end{align*}
	for some constant $ C_{10} $. This means that with high probability, Expression \eqref{Expr: Before probability} has the following upper bound:
	\begin{align*}
		&
		2 \, C_{8} \, C_{ 10 } \, \frac{a_{max}^{2}}{ T } = \frac{4 \, C_{8} \, C_{10} \, a_{max}^{2}}{\pi} \, \delta,
	\end{align*}
	where the equality follows from using $ T  \delta = \tfrac{\pi}{2} $.
	
	If we had originally chosen
	\begin{align*}
		\delta = \min \left( \epsilon \cdot \left(  \frac{4 \, C_{8} \, C_{10} \, a_{max}^{2} }{\pi} \right)^{-1}, \tfrac{2}{9} \min( I ) \right),
	\end{align*}
	we would have that the previous expression is bounded from above by $ \epsilon $. We assume that we made this choice at the beginning of the proof to conclude the argument.

\end{proof}

\appendix

\section{On quantum mixing of physical-space observables on flat 2-tori}

\label{App: Appendix}

Let $ \mathbb{T}_{L}^{2} $ be a flat $ 2 $-torus with side length $ L $. We consider the sequence $ \{ \mathbb{T}_{j}^{2} \}_{j \in \mathbb{N}} $ of growing flat $ 2 $-tori obtained by letting $ L = j $. The torus $ \mathbb{T}_{j}^{2} $ has eigenvalues and corresponding eigenfunctions
\begin{align*}
	\lambda_{( m,n )}^{( j )} = \frac{4 \pi^{2}}{j^{2}} ( m^{2} + n^{2} ), \qquad \psi_{( m,n )}^{( j )}( x,y ) = e^{ \frac{2 \pi i}{j} ( mx + ny  ) },
\end{align*}
where $ ( m,n ) \in \mathbb{Z}^{2} $. Let $ \rho_{( m,n )}^{( j )} = \sqrt{ \lambda^{( j )}_{( m,n )} - \tfrac{1}{4} } $. We consider Properties \emph{(II)} and \emph{(III)} of Theorem \ref{Thm: Large-scale analogue} and show that the corresponding properties fail for the sequence $ \{ \mathbb{T}_{j}^{2} \} $. In fact, it suffices to show that Property \emph{(II)} fails, since this already implies the failure of Property \emph{(III)}.

\begin{proposition}[Large-scale quantum mixing breaks]
	
	\label{Prp: Large-scale quantum mixing breaks}

	Let $ I \subset ( 0,\infty ) $ be a compact interval. There exists $ \epsilon > 0 $ and a sequence of bounded functions $ \{ a_{j}: \mathbb{T}_{j}^{2} \to \mathbb{C} \} $ such that for every $ \delta > 0 $ we have that
	\begin{align*}
		&
		\limsup_{ j \to \infty} \frac{1}{ \# \{ ( m,n ) \in \mathbb{Z}^{2} : \rho_{( m,n )}^{( j )} \in I \} } \sum_{ ( m,n ) \in \mathbb{Z}^{2} : \rho_{( m,n )}^{( j )} \in I }^{} \, \sum_{ \substack{ ( p,q ) \in \mathbb{Z}^{2} : \rho_{( p,q )}^{( j )} \in I, ( p,q ) \neq ( m,n ), \\ | \rho_{( m,n )}^{( j )} - \rho_{( p,q )}^{( j )} | < \delta } }^{} | \langle \psi_{( m,n )}^{( j )}, a \psi_{( p,q )}^{( j )} \rangle |^{2} \geq \epsilon.
	\end{align*}
	
\end{proposition}

\begin{proof}[Proof of Proposition \ref{Prp: Large-scale quantum mixing breaks}]
	
	Let $ \epsilon > 0 $ be specified later, and define $ a_{j}( x,y ) = e^{ \frac{2 \pi i}{j}( x + y ) } $. Let $ \delta > 0 $ be arbitrary; in particular, it is independent of $ \epsilon $ and the choice of $ a_{j} $.

	We estimate the double sum appearing in the statement from below. A convenient lower bound is obtained by restricting to pairs of indices of the form $ (m,n) $ and $ (m-1,n-1) $:
	\begin{align*}
		&
		\sum_{ ( m,n ) \in \mathbb{Z}^{2} : \rho_{( m,n )}^{( j )} \in I }^{} \boldsymbol 1_{ \substack{ \rho_{( m-1,n-1 )}^{( j )} \in I,  | \rho_{( m,n )}^{(j)} - \rho_{( m-1,n-1 )}^{( j )} | < \delta } } \, | \langle \psi_{( m,n )}^{( j )}, a_{j} \psi^{( j )}_{( m-1,n-1 )} \rangle |^{2}
		\\
		& = \
		\sum_{ ( m,n ) \in \mathbb{Z}^{2} : \rho_{( m,n )}^{( j )} \in I }^{} \boldsymbol 1_{ \substack{ \rho_{( m-1,n-1 )}^{( j )} \in I,  | \rho_{( m,n )}^{(j)} - \rho_{( m-1,n-1 )}^{( j )} | < \delta } },
	\end{align*}
	where the equality follows from the identity
	$$ a_{j} \psi_{( m-1, n-1 )}^{( j )} = \psi_{( m,n )}^{( j )}. $$
Consequently, the $\limsup$ in the statement admits the lower bound
	\begin{align*}
		&
		\limsup_{ j \to \infty } \frac{1}{ \# \{ ( m,n ) \in \mathbb{Z}^{2} : \rho_{( m,n )}^{( j )} \in I \} } \sum_{ ( m,n ) \in \mathbb{Z}^{2} : \rho_{( m,n )}^{( j )} \in I }^{} \boldsymbol 1_{ \substack{ \rho_{( m-1,n-1 )}^{( j )} \in I,  | \rho_{( m,n )}^{(j)} - \rho_{( m-1,n-1 )}^{( j )} | < \delta } }
		\\
		& = \
		\limsup_{ j \to \infty } \frac{ \# \{ ( m,n ) \in \mathbb{Z}^{2} : \rho_{( m,n )}^{( j )}, \rho_{( m-1,n-1 )}^{( j )} \in I, | \rho_{( m,n )}^{( j )} - \rho_{( m-1,n-1 )} | < \delta \} }{ \# \{ ( m,n ) \in \mathbb{Z}^{2} : \rho_{( m,n )}^{( j )} \in I \} }.
	\end{align*}

	As $ j \to \infty $, the spacing between $\rho_{( m,n )}^{( j )}$ and $\rho_{( m-1,n-1 )}^{( j )}$ tends to zero uniformly for indices with $\rho_{( m,n )}^{( j )} \in I$. In particular, for sufficiently large $ j $, the condition $ | \rho_{( m,n )}^{( j )} - \rho_{( m-1,n-1 )}^{( j )} | < \delta $ is automatically satisfied whenever both terms lie in $ I $. Moreover, we have the asymptotic relation
	\begin{align*}
		\# \{ ( m,n ) \in \mathbb{Z}^{2} : \rho_{( m,n )}^{( j )}, \rho_{( m-1,n-1 )}^{( j )} \in I \} \sim \# \{ ( m,n ) \in \mathbb{Z}^{2} : \rho_{( m,n )}^{( j )} \in I \}.
	\end{align*}
	It follows that
	\begin{align*}
		\limsup_{ j \to \infty } \frac{ \# \{ ( m,n ) \in \mathbb{Z}^{2} : \rho_{( m,n )}^{( j )}, \rho_{( m-1,n-1 )}^{( j )} \in I, | \rho_{( m,n )}^{( j )} - \rho_{( m-1,n-1 )} | < \delta \} }{ \# \{ ( m,n ) \in \mathbb{Z}^{2} : \rho_{( m,n )}^{( j )} \in I \} } = C
	\end{align*}
	for some constant $ C > 0 $. Choosing $ \epsilon = \tfrac{C}{2} $ yields the desired conclusion.

\end{proof}

In Subsection \ref{Ssec: Background}, we noted that it is an interesting problem whether analogues of Properties \emph{(II)} and \emph{(III)} of Theorem \ref{Thm: Quantum ergodicity and quantum weak mixing} hold for a fixed flat $ 2 $-torus. The example above does not readily extend to that setting: when the torus is fixed and one considers higher and higher eigenvalues, the spectrum does not exhibit the same clustering, and nearby eigenvalues do not accumulate in the same way.

\begin{ack}

The author would like to thank Prof. Farrell Brumley, Prof. Mostafa Sabri, and Prof. Tuomas Sahlsten  for discussions that led to this article. The author also thanks Prof. Tuomas Sahlsten, Prof. Henrik Ueberschär, Dr. F\'{e}lix Lequen, Dr. Søren Mikkelsen, and Dr. Joe Thomas for their guidance while writing this article. Finally, the author wants to thank Prof. Ze\'ev Rudnick and Prof. Mostafa Sabri for their helpful comments on the earlier drafts.

The idea to use the exponential mixing in this project was inspired by the project that led to the article \cite{HLMSU26} by the author with F{\'e}lix Lequen, S{\o}ren Mikkelsen, Tuomas Sahlsten, and Henrik Ueberschär on quantum mixing of eigenfunctions for Schrödinger operators on hyperbolic surfaces. We thank the other authors of that paper.

The research conducted here is supported by the Vilho, Yrjö and Kalle Väisälä Foundation, and the Research Council of Finland, project: Quantum chaos on large and many-body systems, grant numbers 347365, 353738.

\end{ack}

\bibliographystyle{alpha}
\bibliography{NewReferences}

@misc{ABL22,
      title={Eigenfunctions and Random Waves in the {B}enjamini-{S}chramm limit}, 
      author={M. Abert and N. Bergeron and E. {Le Masson}},
      year={2022},
      eprint={1810.05601},
      archivePrefix={arXiv},
      primaryClass={math.SP},
      url={https://arxiv.org/abs/1810.05601}, 
      note={arXiv:1810.05601}
}

@article{Ana17,
  AUTHOR = {Anantharaman, N.},
     TITLE = {Quantum ergodicity on regular graphs},
   JOURNAL = {Comm. Math. Phys.},
  FJOURNAL = {Communications in Mathematical Physics},
    VOLUME = {353},
      YEAR = {2017},
    NUMBER = {2},
     PAGES = {633--690},
      ISSN = {0010-3616,1432-0916},
   MRCLASS = {58J51 (05C50 60G50)},
  MRNUMBER = {3649482},
MRREVIEWER = {Dubi\ Kelmer},
       DOI = {10.1007/s00220-017-2879-9},
       URL = {https://doi.org/10.1007/s00220-017-2879-9},
}

@book{Ana22,
  AUTHOR = {Anantharaman, N.},
     TITLE = {Quantum ergodicity and delocalization of {S}chr\"odinger
              eigenfunctions},
    SERIES = {Zurich Lectures in Advanced Mathematics},
 PUBLISHER = {European Mathematical Society (EMS), Z\"urich},
      YEAR = {[2022] \copyright 2022},
     PAGES = {134},
      ISBN = {978-3-98547-015-0; 978-3-98547-515-5},
   MRCLASS = {58-02 (37D40 58J51 81Q50)},
  MRNUMBER = {4477342},
       DOI = {10.4171/zlam/27},
       URL = {https://doi.org/10.4171/zlam/27},
}

@article{AL15,
  AUTHOR = {Anantharaman, N. and {Le Masson}, E.},
     TITLE = {Quantum ergodicity on large regular graphs},
   JOURNAL = {Duke Math. J.},
  FJOURNAL = {Duke Mathematical Journal},
    VOLUME = {164},
      YEAR = {2015},
    NUMBER = {4},
     PAGES = {723--765},
      ISSN = {0012-7094,1547-7398},
   MRCLASS = {11B75 (05C50 58J51 81Q50)},
  MRNUMBER = {3322309},
MRREVIEWER = {C\'esar\ R.\ de Oliveira},
       DOI = {10.1215/00127094-2881592},
       URL = {https://doi.org/10.1215/00127094-2881592},
}

@misc{AM24,
  title={Spectral gap of random hyperbolic surfaces}, 
      author={N. Anantharaman and L. Monk},
      year={2024},
      eprint={2403.12576},
      archivePrefix={arXiv},
      primaryClass={math.GT},
      url={https://arxiv.org/abs/2403.12576}, 
      note={arXiv:2403.12576}
}

@article{AS17,
  AUTHOR = {Anantharaman, N. and Sabri, M.},
     TITLE = {Quantum ergodicity for the {A}nderson model on regular graphs},
   JOURNAL = {J. Math. Phys.},
  FJOURNAL = {Journal of Mathematical Physics},
    VOLUME = {58},
      YEAR = {2017},
    NUMBER = {9},
     PAGES = {091901, 10},
      ISSN = {0022-2488,1089-7658},
   MRCLASS = {81Q10 (05C63)},
  MRNUMBER = {3707062},
MRREVIEWER = {C\'esar\ R.\ de Oliveira},
       DOI = {10.1063/1.5000962},
       URL = {https://doi.org/10.1063/1.5000962},
}

@article{AS19,
  AUTHOR = {Anantharaman, N. and Sabri, M.},
     TITLE = {Quantum ergodicity on graphs: from spectral to spatial
              delocalization},
   JOURNAL = {Ann. of Math. (2)},
  FJOURNAL = {Annals of Mathematics. Second Series},
    VOLUME = {189},
      YEAR = {2019},
    NUMBER = {3},
     PAGES = {753--835},
      ISSN = {0003-486X,1939-8980},
   MRCLASS = {58J51 (60B20 81Q10)},
  MRNUMBER = {3961083},
MRREVIEWER = {Emmanuel\ Tr\'elat},
       DOI = {10.4007/annals.2019.189.3.3},
       URL = {https://doi.org/10.4007/annals.2019.189.3.3},
}

@article{BS19,
  AUTHOR = {Backhausz, \'A. and Szegedy, B.},
     TITLE = {On the almost eigenvectors of random regular graphs},
   JOURNAL = {Ann. Probab.},
  FJOURNAL = {The Annals of Probability},
    VOLUME = {47},
      YEAR = {2019},
    NUMBER = {3},
     PAGES = {1677--1725},
      ISSN = {0091-1798,2168-894X},
   MRCLASS = {05C80 (05C50 60B20)},
  MRNUMBER = {3945757},
       DOI = {10.1214/18-AOP1294},
       URL = {https://doi.org/10.1214/18-AOP1294},
}

@article{BHKY17,
  AUTHOR = {Bauerschmidt, R. and Huang, J. and Knowles, A.
              and Yau, H.-T.},
     TITLE = {Bulk eigenvalue statistics for random regular graphs},
   JOURNAL = {Ann. Probab.},
  FJOURNAL = {The Annals of Probability},
    VOLUME = {45},
      YEAR = {2017},
    NUMBER = {6A},
     PAGES = {3626--3663},
      ISSN = {0091-1798,2168-894X},
   MRCLASS = {05C80 (05C50 15B52 60B20)},
  MRNUMBER = {3729611},
MRREVIEWER = {Steven\ Joel\ Miller},
       DOI = {10.1214/16-AOP1145},
       URL = {https://doi.org/10.1214/16-AOP1145},
}

@article{BHY19,
  AUTHOR = {Bauerschmidt, R. and Huang, J. and Yau, H.-T.},
     TITLE = {Local {K}esten-{M}c{K}ay law for random regular graphs},
   JOURNAL = {Comm. Math. Phys.},
  FJOURNAL = {Communications in Mathematical Physics},
    VOLUME = {369},
      YEAR = {2019},
    NUMBER = {2},
     PAGES = {523--636},
      ISSN = {0010-3616,1432-0916},
   MRCLASS = {60C05 (05C80 60B20)},
  MRNUMBER = {3962004},
       DOI = {10.1007/s00220-019-03345-3},
       URL = {https://doi.org/10.1007/s00220-019-03345-3},
}

@article{BKY17,
  AUTHOR = {Bauerschmidt, R. and Knowles, A. and Yau, H.-T.},
     TITLE = {Local semicircle law for random regular graphs},
   JOURNAL = {Comm. Pure Appl. Math.},
  FJOURNAL = {Communications on Pure and Applied Mathematics},
    VOLUME = {70},
      YEAR = {2017},
    NUMBER = {10},
     PAGES = {1898--1960},
      ISSN = {0010-3640,1097-0312},
   MRCLASS = {05C80 (05C50 60B20)},
  MRNUMBER = {3688032},
MRREVIEWER = {Lyuben\ R.\ Mutafchiev},
       DOI = {10.1002/cpa.21709},
       URL = {https://doi.org/10.1002/cpa.21709},
}

@article{Ber77,
  AUTHOR = {Berry, M. V.},
     TITLE = {Regular and irregular semiclassical wavefunctions},
   JOURNAL = {J. Phys. A},
  FJOURNAL = {Journal of Physics. A. Mathematical and General},
    VOLUME = {10},
      YEAR = {1977},
    NUMBER = {12},
     PAGES = {2083--2091},
      ISSN = {0305-4470,1751-8121},
   MRCLASS = {81.58 (28A65 58F15)},
  MRNUMBER = {489542},
       URL = {http://stacks.iop.org/0305-4470/10/2083},
}

@article{BLL16,
  AUTHOR = {Brooks, S. and {Le Masson}, E. and Lindenstrauss, E.},
     TITLE = {Quantum ergodicity and averaging operators on the sphere},
   JOURNAL = {Int. Math. Res. Not. IMRN},
  FJOURNAL = {International Mathematics Research Notices. IMRN},
      YEAR = {2016},
    NUMBER = {19},
     PAGES = {6034--6064},
      ISSN = {1073-7928,1687-0247},
   MRCLASS = {35P20 (28D05 37A45 58J51)},
  MRNUMBER = {3567266},
MRREVIEWER = {Anton\ Deitmar},
       DOI = {10.1093/imrn/rnv337},
       URL = {https://doi.org/10.1093/imrn/rnv337},
}

@article{BL13,
  AUTHOR = {Brooks, S. and Lindenstrauss, E.},
     TITLE = {Non-localization of eigenfunctions on large regular graphs},
   JOURNAL = {Israel J. Math.},
  FJOURNAL = {Israel Journal of Mathematics},
    VOLUME = {193},
      YEAR = {2013},
    NUMBER = {1},
     PAGES = {1--14},
      ISSN = {0021-2172,1565-8511},
   MRCLASS = {11B75 (05C50)},
  MRNUMBER = {3038543},
MRREVIEWER = {Song\ Guo},
       DOI = {10.1007/s11856-012-0096-y},
       URL = {https://doi.org/10.1007/s11856-012-0096-y},
}

@misc{BLS26,
      title={Quantum mixing on large {S}chreier graphs}, 
      author={C. Bordenave and C. Letrouit and M. Sabri},
      year={2026},
      eprint={2601.14182},
      archivePrefix={arXiv},
      primaryClass={math.SP},
      url={https://arxiv.org/abs/2601.14182}, 
      note={arXiv:2601.14182},
}

@article{BY17,
  AUTHOR = {Bourgade, P. and Yau, H.-T.},
     TITLE = {The eigenvector moment flow and local quantum unique
              ergodicity},
   JOURNAL = {Comm. Math. Phys.},
  FJOURNAL = {Communications in Mathematical Physics},
    VOLUME = {350},
      YEAR = {2017},
    NUMBER = {1},
     PAGES = {231--278},
      ISSN = {0010-3616,1432-0916},
   MRCLASS = {58J51 (60B20 60H30)},
  MRNUMBER = {3606475},
MRREVIEWER = {Jiang\ Hu},
       DOI = {10.1007/s00220-016-2627-6},
       URL = {https://doi.org/10.1007/s00220-016-2627-6},
}

@article{CES21,
  AUTHOR = {Cipolloni, G. and Erd\H{o}s, L. and Schr\"oder,
              D.},
     TITLE = {Eigenstate thermalization hypothesis for {W}igner matrices},
   JOURNAL = {Comm. Math. Phys.},
  FJOURNAL = {Communications in Mathematical Physics},
    VOLUME = {388},
      YEAR = {2021},
    NUMBER = {2},
     PAGES = {1005--1048},
      ISSN = {0010-3616,1432-0916},
   MRCLASS = {60B20 (15B52 58J51 60F05 81Q50)},
  MRNUMBER = {4334253},
MRREVIEWER = {Pragya\ Shukla},
       DOI = {10.1007/s00220-021-04239-z},
       URL = {https://doi.org/10.1007/s00220-021-04239-z},
}

@article{CdV85,
  AUTHOR = {Colin de Verdi\`ere, Y.},
     TITLE = {Ergodicit\'e{} et fonctions propres du laplacien},
   JOURNAL = {Comm. Math. Phys.},
  FJOURNAL = {Communications in Mathematical Physics},
    VOLUME = {102},
      YEAR = {1985},
    NUMBER = {3},
     PAGES = {497--502},
      ISSN = {0010-3616,1432-0916},
   MRCLASS = {58G25 (35P20)},
  MRNUMBER = {818831},
MRREVIEWER = {J\"urgen\ Eichhorn},
       URL = {http://projecteuclid.org/euclid.cmp/1104114465},
}

@article{Deu18,
  AUTHOR = {Deutsch, J. M.},
     TITLE = {Eigenstate thermalization hypothesis},
   JOURNAL = {Rep. Progr. Phys.},
  FJOURNAL = {Reports on Progress in Physics},
    VOLUME = {81},
      YEAR = {2018},
    NUMBER = {8},
     PAGES = {082001, 16},
      ISSN = {0034-4885,1361-6633},
   MRCLASS = {82B10 (60B20 94A17)},
  MRNUMBER = {3829412},
       DOI = {10.1088/1361-6633/aac9f1},
       URL = {https://doi.org/10.1088/1361-6633/aac9f1},
}

@article{DP12,
  AUTHOR = {Dumitriu, I. and Pal, S.},
     TITLE = {Sparse regular random graphs: spectral density and
              eigenvectors},
   JOURNAL = {Ann. Probab.},
  FJOURNAL = {The Annals of Probability},
    VOLUME = {40},
      YEAR = {2012},
    NUMBER = {5},
     PAGES = {2197--2235},
      ISSN = {0091-1798,2168-894X},
   MRCLASS = {60B20 (05C50 05C80 60C05)},
  MRNUMBER = {3025715},
MRREVIEWER = {Steven\ Joel\ Miller},
       DOI = {10.1214/11-AOP673},
       URL = {https://doi.org/10.1214/11-AOP673},
}

@article{EKYY13,
  AUTHOR = {Erd\H{o}s, L. and Knowles, A. and Yau, H.-T.
              and Yin, J.},
     TITLE = {Spectral statistics of {E}rd{\H{o}}s-{R}\'enyi graphs {I}:
              {L}ocal semicircle law},
   JOURNAL = {Ann. Probab.},
  FJOURNAL = {The Annals of Probability},
    VOLUME = {41},
      YEAR = {2013},
    NUMBER = {3B},
     PAGES = {2279--2375},
      ISSN = {0091-1798,2168-894X},
   MRCLASS = {60B20 (05C80)},
  MRNUMBER = {3098073},
MRREVIEWER = {Longmin\ Wang},
       DOI = {10.1214/11-AOP734},
       URL = {https://doi.org/10.1214/11-AOP734},
}

@article{ESY09a,
  AUTHOR = {Erd\H{o}s, L. and Schlein, B. and Yau,
              H.-T.},
     TITLE = {Local semicircle law and complete delocalization for {W}igner
              random matrices},
   JOURNAL = {Comm. Math. Phys.},
  FJOURNAL = {Communications in Mathematical Physics},
    VOLUME = {287},
      YEAR = {2009},
    NUMBER = {2},
     PAGES = {641--655},
      ISSN = {0010-3616,1432-0916},
   MRCLASS = {60B20 (60E05 82B44)},
  MRNUMBER = {2481753},
MRREVIEWER = {Olivier\ Marchal},
       DOI = {10.1007/s00220-008-0636-9},
       URL = {https://doi.org/10.1007/s00220-008-0636-9},
}

@article{ESY09b,
  AUTHOR = {Erd\H{o}s, L. and Schlein, B. and Yau,
              H.-T.},
     TITLE = {Semicircle law on short scales and delocalization of
              eigenvectors for {W}igner random matrices},
   JOURNAL = {Ann. Probab.},
  FJOURNAL = {The Annals of Probability},
    VOLUME = {37},
      YEAR = {2009},
    NUMBER = {3},
     PAGES = {815--852},
      ISSN = {0091-1798,2168-894X},
   MRCLASS = {15B52 (15A18 15A42 47B80 82B44)},
  MRNUMBER = {2537522},
MRREVIEWER = {Razvan\ Teodorescu},
       DOI = {10.1214/08-AOP421},
       URL = {https://doi.org/10.1214/08-AOP421},
}

@article{Gei15,
  AUTHOR = {Geisinger, L.},
     TITLE = {Convergence of the density of states and delocalization of
              eigenvectors on random regular graphs},
   JOURNAL = {J. Spectr. Theory},
  FJOURNAL = {Journal of Spectral Theory},
    VOLUME = {5},
      YEAR = {2015},
    NUMBER = {4},
     PAGES = {783--827},
      ISSN = {1664-039X,1664-0403},
   MRCLASS = {60B20 (05C50 05C80 35A08 35P20 35R60)},
  MRNUMBER = {3433288},
       DOI = {10.4171/JST/114},
       URL = {https://doi.org/10.4171/JST/114},
}

@article{GL93,
  AUTHOR = {G\'erard, P. and Leichtnam, \'E.},
     TITLE = {Ergodic properties of eigenfunctions for the {D}irichlet
              problem},
   JOURNAL = {Duke Math. J.},
  FJOURNAL = {Duke Mathematical Journal},
    VOLUME = {71},
      YEAR = {1993},
    NUMBER = {2},
     PAGES = {559--607},
      ISSN = {0012-7094,1547-7398},
   MRCLASS = {35P20 (35R60 35S30 58G25)},
  MRNUMBER = {1233448},
MRREVIEWER = {V.\ S.\ Rabinovich},
       DOI = {10.1215/S0012-7094-93-07122-0},
       URL = {https://doi.org/10.1215/S0012-7094-93-07122-0},
}

@book{HS12,
  AUTHOR = {Halmos, P. R. and Sunder, V. S.},
     TITLE = {Bounded integral operators on {$L\sp{2}$}\ spaces},
     SERIES = {Ergebnisse der Mathematik und ihrer Grenzgebiete [Results in Mathematics and Related Areas]},
    VOLUME = {96},
 PUBLISHER = {Springer-Verlag, Berlin-New York},
      YEAR = {1978},
     PAGES = {xv+132},
      ISBN = {3-540-08894-6},
   MRCLASS = {47B38 (47G05)},
  MRNUMBER = {517709},
MRREVIEWER = {W.\ A. J. Luxemburg},
}

@misc{HMT25,
  title={Spectral gap with polynomial rate for {W}eil-{P}etersson random surfaces}, 
      author={W. Hide and D. Macera and J. Thomas},
      year={2025},
      eprint={2508.14874},
      archivePrefix={arXiv},
      primaryClass={math.SP},
      url={https://arxiv.org/abs/2508.14874}, 
      note={ArXiv:2508.14874}
}

@article{HMN25,
  AUTHOR = {Hide, W. and Moy, J. and Naud, F.},
     TITLE = {On the spectral gap of negatively curved surface covers},
   JOURNAL = {Int. Math. Res. Not. IMRN},
  FJOURNAL = {International Mathematics Research Notices. IMRN},
      YEAR = {2025},
    NUMBER = {24},
     PAGES = {Paper No. rnaf357, 24},
      ISSN = {1073-7928,1687-0247},
   MRCLASS = {58J50 (53C21 58J65)},
  MRNUMBER = {5000801},
       DOI = {10.1093/imrn/rnaf357},
       URL = {https://doi.org/10.1093/imrn/rnaf357},
}

@article{HM23,
  AUTHOR = {Hide, W. and Magee, M.},
     TITLE = {Near optimal spectral gaps for hyperbolic surfaces},
   JOURNAL = {Ann. of Math. (2)},
  FJOURNAL = {Annals of Mathematics. Second Series},
    VOLUME = {198},
      YEAR = {2023},
    NUMBER = {2},
     PAGES = {791--824},
      ISSN = {0003-486X,1939-8980},
   MRCLASS = {58J50 (05C50 05C80)},
  MRNUMBER = {4635304},
MRREVIEWER = {J\'ozef\ Dodziuk},
       DOI = {10.4007/annals.2023.198.2.6},
       URL = {https://doi.org/10.4007/annals.2023.198.2.6},
}

@misc{HLMSU26,
      title={Quantum Mixing for {S}chr\"odinger eigenfunctions in {B}enjamini-{S}chramm limit}, 
      author={K. Hippi and F. Lequen and S. Mikkelsen and T. Sahlsten and H. Ueberschär},
      year={2026},
      eprint={2604.21582},
      archivePrefix={arXiv},
      primaryClass={math.SP},
      url={https://arxiv.org/abs/2604.21582}, 
      note={arXiv:2604.21582}
}

@book{Iwa02,
  AUTHOR = {Iwaniec, H.},
     TITLE = {Spectral methods of automorphic forms},
    SERIES = {Graduate Studies in Mathematics},
    VOLUME = {53},
   EDITION = {Second},
 PUBLISHER = {American Mathematical Society, Providence, RI; Revista
              Matem\'atica Iberoamericana, Madrid},
      YEAR = {2002},
     PAGES = {xii+220},
      ISBN = {0-8218-3160-7},
   MRCLASS = {11F72 (11F12 11F37)},
  MRNUMBER = {1942691},
       DOI = {10.1090/gsm/053},
       URL = {https://doi.org/10.1090/gsm/053},
}

@book{Kat92,
   AUTHOR = {Katok, S.},
     TITLE = {Fuchsian groups},
    SERIES = {Chicago Lectures in Mathematics},
 PUBLISHER = {University of Chicago Press, Chicago, IL},
      YEAR = {1992},
     PAGES = {x+175},
      ISBN = {0-226-42582-7; 0-226-42583-5},
   MRCLASS = {20H10 (30F35)},
  MRNUMBER = {1177168},
MRREVIEWER = {I.\ Kra},
}

@article{KRS03,
  AUTHOR = {Kim, H. H. and Ramakrishnan, D. and Sarnak, P.},
     TITLE = {Functoriality for the exterior square of {${\rm GL}_4$} and
              the symmetric fourth of {${\rm GL}_2$}},
      NOTE = {With appendix 1 by Dinakar Ramakrishnan and appendix 2 by Kim
              and Peter Sarnak},
   JOURNAL = {J. Amer. Math. Soc.},
  FJOURNAL = {Journal of the American Mathematical Society},
    VOLUME = {16},
      YEAR = {2003},
    NUMBER = {1},
     PAGES = {139--183},
      ISSN = {0894-0347,1088-6834},
   MRCLASS = {11F70 (11R39 22E46)},
  MRNUMBER = {1937203},
MRREVIEWER = {Mahdi\ Asgari},
       DOI = {10.1090/S0894-0347-02-00410-1},
       URL = {https://doi.org/10.1090/S0894-0347-02-00410-1},
}

@article{KSV07,
  AUTHOR = {Katz, M. G. and Schaps, M. and Vishne, U.},
     TITLE = {Logarithmic growth of systole of arithmetic {R}iemann surfaces
              along congruence subgroups},
   JOURNAL = {J. Differential Geom.},
  FJOURNAL = {Journal of Differential Geometry},
    VOLUME = {76},
      YEAR = {2007},
    NUMBER = {3},
     PAGES = {399--422},
      ISSN = {0022-040X,1945-743X},
   MRCLASS = {53C23 (11R52 57M50)},
  MRNUMBER = {2331526},
MRREVIEWER = {St\'ephane\ Sabourau},
       URL = {http://projecteuclid.org/euclid.jdg/1180135693},
}

@article{LP82,
   AUTHOR = {Lax, P. D. and Phillips, R. S.},
     TITLE = {The asymptotic distribution of lattice points in {E}uclidean
              and non-{E}uclidean spaces},
   JOURNAL = {J. Functional Analysis},
  FJOURNAL = {Journal of Functional Analysis},
    VOLUME = {46},
      YEAR = {1982},
    NUMBER = {3},
     PAGES = {280--350},
      ISSN = {0022-1236},
   MRCLASS = {10J25 (10E05 58G16)},
  MRNUMBER = {661875},
MRREVIEWER = {P.\ G\"unther},
       DOI = {10.1016/0022-1236(82)90050-7},
       URL = {https://doi.org/10.1016/0022-1236(82)90050-7},
}

@book{LS93,
  AUTHOR = {Lazutkin, V. F. and {\v{S}}hnirel'man, A. I.},
     TITLE = {K{AM} theory and semiclassical approximations to
              eigenfunctions},
    SERIES = {Ergebnisse der Mathematik und ihrer Grenzgebiete (3) [Results
              in Mathematics and Related Areas (3)]},
    VOLUME = {24},
      NOTE = {With an addendum by A. I. {\v{S}}hnirel'man},
 PUBLISHER = {Springer-Verlag, Berlin},
      YEAR = {1993},
     PAGES = {x+387},
      ISBN = {3-540-53389-3},
   MRCLASS = {58Fxx (35J05 35P05 58F19 58F27 70H05 81Q20)},
  MRNUMBER = {1239173},
MRREVIEWER = {Helmut\ R\"ussmann},
       DOI = {10.1007/978-3-642-76247-5},
       URL = {https://doi.org/10.1007/978-3-642-76247-5},
}

@article{Lin06,
  AUTHOR = {Lindenstrauss, E.},
     TITLE = {Invariant measures and arithmetic quantum unique ergodicity},
   JOURNAL = {Ann. of Math. (2)},
  FJOURNAL = {Annals of Mathematics. Second Series},
    VOLUME = {163},
      YEAR = {2006},
    NUMBER = {1},
     PAGES = {165--219},
      ISSN = {0003-486X,1939-8980},
   MRCLASS = {11F72 (37A45 37D40)},
  MRNUMBER = {2195133},
MRREVIEWER = {Ze\'ev\ Rudnick},
       DOI = {10.4007/annals.2006.163.165},
       URL = {https://doi.org/10.4007/annals.2006.163.165},
}

@article{LS17,
  AUTHOR = {{Le Masson}, E. and Sahlsten, T.},
     TITLE = {Quantum ergodicity and {B}enjamini-{S}chramm convergence of
              hyperbolic surfaces},
   JOURNAL = {Duke Math. J.},
  FJOURNAL = {Duke Mathematical Journal},
    VOLUME = {166},
      YEAR = {2017},
    NUMBER = {18},
     PAGES = {3425--3460},
      ISSN = {0012-7094,1547-7398},
   MRCLASS = {81Q50 (11F72 37D40)},
  MRNUMBER = {3732880},
MRREVIEWER = {Dubi\ Kelmer},
       DOI = {10.1215/00127094-2017-0027},
       URL = {https://doi.org/10.1215/00127094-2017-0027},
}

@article{LS24,
  AUTHOR = {{Le Masson}, E. and Sahlsten, T.},
     TITLE = {Quantum ergodicity for {E}isenstein series on hyperbolic
              surfaces of large genus},
   JOURNAL = {Math. Ann.},
  FJOURNAL = {Mathematische Annalen},
    VOLUME = {389},
      YEAR = {2024},
    NUMBER = {1},
     PAGES = {845--898},
      ISSN = {0025-5831,1432-1807},
   MRCLASS = {58J51 (11F72 37D40 81Q35)},
  MRNUMBER = {4735964},
MRREVIEWER = {Mostafa\ Sabri},
       DOI = {10.1007/s00208-023-02671-1},
       URL = {https://doi.org/10.1007/s00208-023-02671-1},
}

@article{LW24,
  AUTHOR = {Lipnowski, M. and Wright, A.},
     TITLE = {Towards optimal spectral gaps in large genus},
   JOURNAL = {Ann. Probab.},
  FJOURNAL = {The Annals of Probability},
    VOLUME = {52},
      YEAR = {2024},
    NUMBER = {2},
     PAGES = {545--575},
      ISSN = {0091-1798,2168-894X},
   MRCLASS = {58J50 (11F72 32G15 37D40)},
  MRNUMBER = {4718401},
MRREVIEWER = {Joe\ Thomas},
       DOI = {10.1214/23-aop1657},
       URL = {https://doi.org/10.1214/23-aop1657},
}

@misc{Mag24,
   title={The limit points of the bass notes of arithmetic hyperbolic surfaces}, 
      author={M. Magee},
      year={2024},
      eprint={2403.00928},
      archivePrefix={arXiv},
      primaryClass={math.NT},
      url={https://arxiv.org/abs/2403.00928},
      note={arXiv:2403.00928} 
}

@article{MN20,
   AUTHOR = {Magee, M. and Naud, F.},
     TITLE = {Explicit spectral gaps for random covers of {R}iemann
              surfaces},
   JOURNAL = {Publ. Math. Inst. Hautes \'Etudes Sci.},
  FJOURNAL = {Publications Math\'ematiques. Institut de Hautes \'Etudes
              Scientifiques},
    VOLUME = {132},
      YEAR = {2020},
     PAGES = {137--179},
      ISSN = {0073-8301,1618-1913},
   MRCLASS = {58J50 (05C50 11M36)},
  MRNUMBER = {4179833},
MRREVIEWER = {Anton\ Deitmar},
       DOI = {10.1007/s10240-020-00118-w},
       URL = {https://doi.org/10.1007/s10240-020-00118-w},
}

@article{MNP22,
  AUTHOR = {Magee, M. and Naud, F. and Puder, D.},
     TITLE = {A random cover of a compact hyperbolic surface has relative
              spectral gap {$\frac{3}{16}-\varepsilon$}},
   JOURNAL = {Geom. Funct. Anal.},
  FJOURNAL = {Geometric and Functional Analysis},
    VOLUME = {32},
      YEAR = {2022},
    NUMBER = {3},
     PAGES = {595--661},
      ISSN = {1016-443X,1420-8970},
   MRCLASS = {58J50 (05C50 32G15)},
  MRNUMBER = {4431124},
MRREVIEWER = {Sugata\ Mondal},
       DOI = {10.1007/s00039-022-00602-x},
       URL = {https://doi.org/10.1007/s00039-022-00602-x},
}

@incollection{Mar25,
  AUTHOR = {Marklof, J.},
     TITLE = {Arithmetic quantum chaos},
 BOOKTITLE = {Encyclopedia of mathematical physics. {V}ol. 2. {Q}uantum
              mechanics, quantization and quantum computation},
     PAGES = {465--474},
      NOTE = {Reprinted from [2238867]},
 PUBLISHER = {Academic Press, Amsterdam},
      YEAR = {[2025] \copyright 2025},
      ISBN = {978-0-443-29953-7; 978-0-323-95703-8},
   MRCLASS = {},
  MRNUMBER = {4975307},
}

@incollection{Mar12,
  AUTHOR = {Marklof, J.},
     TITLE = {Selberg's trace formula: an introduction},
 BOOKTITLE = {Hyperbolic geometry and applications in quantum chaos and
              cosmology},
    SERIES = {London Math. Soc. Lecture Note Ser.},
    VOLUME = {397},
     PAGES = {83--119},
 PUBLISHER = {Cambridge Univ. Press, Cambridge},
      YEAR = {2012},
      ISBN = {978-1-107-61049-1},
   MRCLASS = {11F72 (11M36)},
  MRNUMBER = {2885182},
MRREVIEWER = {Xian-Jin\ Li},
}

@article{MR12,
  AUTHOR = {Marklof, J. and Rudnick, Z.},
     TITLE = {Almost all eigenfunctions of a rational polygon are uniformly
              distributed},
   JOURNAL = {J. Spectr. Theory},
  FJOURNAL = {Journal of Spectral Theory},
    VOLUME = {2},
      YEAR = {2012},
    NUMBER = {1},
     PAGES = {107--113},
      ISSN = {1664-039X,1664-0403},
   MRCLASS = {58J51 (35J05 35P10 35P20)},
  MRNUMBER = {2879311},
MRREVIEWER = {Nicolas\ Bedaride},
       DOI = {10.4171/JST/23},
       URL = {https://doi.org/10.4171/JST/23},
}

@article{Mat13,
  AUTHOR = {Matheus, C.},
     TITLE = {Some quantitative versions of {R}atner's mixing estimates},
   JOURNAL = {Bull. Braz. Math. Soc. (N.S.)},
  FJOURNAL = {Bulletin of the Brazilian Mathematical Society. New Series.
              Boletim da Sociedade Brasileira de Matem\'atica},
    VOLUME = {44},
      YEAR = {2013},
    NUMBER = {3},
     PAGES = {469--488},
      ISSN = {1678-7544,1678-7714},
   MRCLASS = {37D40 (37A25)},
  MRNUMBER = {3124746},
       DOI = {10.1007/s00574-013-0022-x},
       URL = {https://doi.org/10.1007/s00574-013-0022-x},
}

@article{McK22,
  AUTHOR = {McKenzie, T.},
     TITLE = {The necessity of conditions for graph quantum ergodicity and
              {C}artesian products with an infinite graph},
   JOURNAL = {C. R. Math. Acad. Sci. Paris},
  FJOURNAL = {Comptes Rendus Math\'ematique. Acad\'emie des Sciences. Paris},
    VOLUME = {360},
      YEAR = {2022},
     PAGES = {399--408},
      ISSN = {1631-073X,1778-3569},
   MRCLASS = {58J51 (05C50 11B75 39A12 81Q35 81Q50)},
  MRNUMBER = {4415731},
MRREVIEWER = {Mostafa\ Sabri},
       DOI = {10.5802/crmath.316},
       URL = {https://doi.org/10.5802/crmath.316},
}

@article{MS23,
  AUTHOR = {McKenzie, T. and Sabri, M.},
     TITLE = {Quantum ergodicity for periodic graphs},
   JOURNAL = {Comm. Math. Phys.},
  FJOURNAL = {Communications in Mathematical Physics},
    VOLUME = {403},
      YEAR = {2023},
    NUMBER = {3},
     PAGES = {1477--1509},
      ISSN = {0010-3616,1432-0916},
   MRCLASS = {58J51 (05C50 39A70 81Q35)},
  MRNUMBER = {4652907},
MRREVIEWER = {C\'esar\ R.\ de Oliveira},
       DOI = {10.1007/s00220-023-04826-2},
       URL = {https://doi.org/10.1007/s00220-023-04826-2},
}

@article{Mir13,
  AUTHOR = {Mirzakhani, M.},
     TITLE = {Growth of {W}eil-{P}etersson volumes and random hyperbolic
              surfaces of large genus},
   JOURNAL = {J. Differential Geom.},
  FJOURNAL = {Journal of Differential Geometry},
    VOLUME = {94},
      YEAR = {2013},
    NUMBER = {2},
     PAGES = {267--300},
      ISSN = {0022-040X,1945-743X},
   MRCLASS = {32G15 (53C22 60B05)},
  MRNUMBER = {3080483},
MRREVIEWER = {Zongliang\ Sun},
       URL = {http://projecteuclid.org/euclid.jdg/1367438650},
}

@phdthesis{Mon21,
  author       = {Monk, L.},
  title        = {Geometry and Spectrum of Typical Hyperbolic Surfaces},
  school       = {Universit{\'e} de Strasbourg},
  year         = {2021},
  address      = {Strasbourg, France},
  note         = {Available at \url{https://lauramonk.github.io/thesis.pdf}}
}

@article{Mon22,
  AUTHOR = {Monk, L.},
     TITLE = {Benjamini-{S}chramm convergence and spectra of random
              hyperbolic surfaces of high genus},
   JOURNAL = {Anal. PDE},
  FJOURNAL = {Analysis \& PDE},
    VOLUME = {15},
      YEAR = {2022},
    NUMBER = {3},
     PAGES = {727--752},
      ISSN = {2157-5045,1948-206X},
   MRCLASS = {58J50 (32G15)},
  MRNUMBER = {4442839},
MRREVIEWER = {Yuhao\ Xue},
       DOI = {10.2140/apde.2022.15.727},
       URL = {https://doi.org/10.2140/apde.2022.15.727},
}

@misc{Loz03,
         key = "{\relax DLMF}",
       title = "{\it NIST Digital Library of Mathematical Functions}",
howpublished = "\url{https://dlmf.nist.gov/}, Release 1.2.6 of 2026-03-15",
         url = "https://dlmf.nist.gov/",
        note = "F. W. J. Olver, A. B. {Olde Daalhuis}, D. W. Lozier, B. I. Schneider,
                R. F. Boisvert, C. W. Clark, B. R. Miller, B. V. Saunders,
                H. S. Cohl, and M. A. McClain, eds."}

@article{Rat87,
  AUTHOR = {Ratner, M.},
     TITLE = {The rate of mixing for geodesic and horocycle flows},
   JOURNAL = {Ergodic Theory Dynam. Systems},
  FJOURNAL = {Ergodic Theory and Dynamical Systems},
    VOLUME = {7},
      YEAR = {1987},
    NUMBER = {2},
     PAGES = {267--288},
      ISSN = {0143-3857,1469-4417},
   MRCLASS = {58F17 (22E40)},
  MRNUMBER = {896798},
MRREVIEWER = {M.\ Rees},
       DOI = {10.1017/S0143385700004004},
       URL = {https://doi.org/10.1017/S0143385700004004},
}

@article{RS94,
  AUTHOR = {Rudnick, Z. and Sarnak, P.},
     TITLE = {The behaviour of eigenstates of arithmetic hyperbolic
              manifolds},
   JOURNAL = {Comm. Math. Phys.},
  FJOURNAL = {Communications in Mathematical Physics},
    VOLUME = {161},
      YEAR = {1994},
    NUMBER = {1},
     PAGES = {195--213},
      ISSN = {0010-3616,1432-0916},
   MRCLASS = {11F72 (11F32 11F37 81Q50)},
  MRNUMBER = {1266075},
MRREVIEWER = {Jens\ Bolte},
       URL = {http://projecteuclid.org/euclid.cmp/1104269797},
}

@article{Sar11,
  AUTHOR = {Sarnak, P.},
     TITLE = {Recent progress on the quantum unique ergodicity conjecture},
   JOURNAL = {Bull. Amer. Math. Soc. (N.S.)},
  FJOURNAL = {American Mathematical Society. Bulletin. New Series},
    VOLUME = {48},
      YEAR = {2011},
    NUMBER = {2},
     PAGES = {211--228},
      ISSN = {0273-0979,1088-9485},
   MRCLASS = {58J51 (11F41 11F67 11F72 37A45)},
  MRNUMBER = {2774090},
MRREVIEWER = {Dieter\ H.\ Mayer},
       DOI = {10.1090/S0273-0979-2011-01323-4},
       URL = {https://doi.org/10.1090/S0273-0979-2011-01323-4},
}

@incollection{Sel65,
  AUTHOR = {Selberg, A.},
     TITLE = {On the estimation of {F}ourier coefficients of modular forms},
 BOOKTITLE = {Proc. {S}ympos. {P}ure {M}ath., {V}ol. {VIII}},
     PAGES = {1--15},
 PUBLISHER = {Amer. Math. Soc., Providence, RI},
      YEAR = {1965},
   MRCLASS = {10.20},
  MRNUMBER = {182610},
MRREVIEWER = {J.\ R.\ Smart},
}

@article{Shn74,
  AUTHOR = {\v{S}nirel'man, A. I.},
     TITLE = {Ergodic properties of eigenfunctions},
   JOURNAL = {Uspehi Mat. Nauk},
  FJOURNAL = {Akademija Nauk SSSR i Moskovskoe Matemati\v ceskoe Ob\v s\v
              cestvo. Uspehi Matemati\v ceskih Nauk},
    VOLUME = {29},
      YEAR = {1974},
    NUMBER = {6(180)},
     PAGES = {181--182},
      ISSN = {0042-1316},
   MRCLASS = {58G99 (35P20)},
  MRNUMBER = {402834},
MRREVIEWER = {D.\ Newton},
}

@incollection{Sun97,
  AUTHOR = {Sunada, T.},
     TITLE = {Quantum ergodicity},
 BOOKTITLE = {Progress in inverse spectral geometry},
    SERIES = {Trends Math.},
     PAGES = {175--196},
 PUBLISHER = {Birkh\"auser, Basel},
      YEAR = {1997},
      ISBN = {3-7643-5755-X},
   MRCLASS = {81Q50 (37N20 47N50 58J40 82B10)},
  MRNUMBER = {1731156},
MRREVIEWER = {Lech\ Jak\'obczyk},
}

@article{TVW13,
  AUTHOR = {Tran, L. V. and Vu, V. H. and Wang, K.},
     TITLE = {Sparse random graphs: eigenvalues and eigenvectors},
   JOURNAL = {Random Structures Algorithms},
  FJOURNAL = {Random Structures \& Algorithms},
    VOLUME = {42},
      YEAR = {2013},
    NUMBER = {1},
     PAGES = {110--134},
      ISSN = {1042-9832,1098-2418},
   MRCLASS = {05C80 (05C50 60B20)},
  MRNUMBER = {2999215},
MRREVIEWER = {Jihyeok\ Choi},
       DOI = {10.1002/rsa.20406},
       URL = {https://doi.org/10.1002/rsa.20406},
}

@phdthesis{Wat02,
  AUTHOR = {Watson, T. C.},
     TITLE = {Rankin triple products and quantum chaos},
    SCHOOL = {Princeton University},
      YEAR = {2002},
     PAGES = {81},
      ISBN = {978-0493-52934-9},
   MRCLASS = {99-05},
  MRNUMBER = {2703041},
       URL = {http://gateway.proquest.com/openurl?url_ver=Z39.88-2004&rft_val_fmt=info:ofi/fmt:kev:mtx:dissertation&res_dat=xri:pqdiss&rft_dat=xri:pqdiss:3039710},
      NOTE = {Available at \url{http://gateway.proquest.com/openurl?url_ver=Z39.88-2004&rft_val_fmt=info:ofi/fmt:kev:mtx:dissertation&res_dat=xri:pqdiss&rft_dat=xri:pqdiss:3039710}}
}

@article{Wri20,
  AUTHOR = {Wright, A.},
     TITLE = {A tour through {M}irzakhani's work on moduli spaces of
              {R}iemann surfaces},
   JOURNAL = {Bull. Amer. Math. Soc. (N.S.)},
  FJOURNAL = {American Mathematical Society. Bulletin. New Series},
    VOLUME = {57},
      YEAR = {2020},
    NUMBER = {3},
     PAGES = {359--408},
      ISSN = {0273-0979,1088-9485},
   MRCLASS = {32G15},
  MRNUMBER = {4108090},
       DOI = {10.1090/bull/1687},
       URL = {https://doi.org/10.1090/bull/1687},
}

@article{WX25,
  title={Prime geodesic theorem and closed geodesics for large genus},
  author={Wu, Y. and Xue, Y.},
  journal={Journal of the European Mathematical Society},
  year={2025}
}

@article{Zel87,
  AUTHOR = {Zelditch, S.},
     TITLE = {Uniform distribution of eigenfunctions on compact hyperbolic
              surfaces},
   JOURNAL = {Duke Math. J.},
  FJOURNAL = {Duke Mathematical Journal},
    VOLUME = {55},
      YEAR = {1987},
    NUMBER = {4},
     PAGES = {919--941},
      ISSN = {0012-7094,1547-7398},
   MRCLASS = {58G25 (58G30)},
  MRNUMBER = {916129},
MRREVIEWER = {Alejandro\ Uribe},
       DOI = {10.1215/S0012-7094-87-05546-3},
       URL = {https://doi.org/10.1215/S0012-7094-87-05546-3},
}

@article{Zel90,
  AUTHOR = {Zelditch, S.},
     TITLE = {Quantum transition amplitudes for ergodic and for completely
              integrable systems},
   JOURNAL = {J. Funct. Anal.},
  FJOURNAL = {Journal of Functional Analysis},
    VOLUME = {94},
      YEAR = {1990},
    NUMBER = {2},
     PAGES = {415--436},
      ISSN = {0022-1236,1096-0783},
   MRCLASS = {58F19 (58G25 81Q99)},
  MRNUMBER = {1081652},
MRREVIEWER = {David\ Gurarie},
       DOI = {10.1016/0022-1236(90)90021-C},
       URL = {https://doi.org/10.1016/0022-1236(90)90021-C},
}

@article{Zel96a,
   AUTHOR = {Zelditch, S.},
     TITLE = {Quantum mixing},
   JOURNAL = {J. Funct. Anal.},
  FJOURNAL = {Journal of Functional Analysis},
    VOLUME = {140},
      YEAR = {1996},
    NUMBER = {1},
     PAGES = {68--86},
      ISSN = {0022-1236,1096-0783},
   MRCLASS = {58G25 (58G18 81Q20 82B10)},
  MRNUMBER = {1404574},
       DOI = {10.1006/jfan.1996.0098},
       URL = {https://doi.org/10.1006/jfan.1996.0098},
}

@article{Zel96b,
   AUTHOR = {Zelditch, S.},
     TITLE = {Quantum ergodicity of {$C^*$} dynamical systems},
   JOURNAL = {Comm. Math. Phys.},
  FJOURNAL = {Communications in Mathematical Physics},
    VOLUME = {177},
      YEAR = {1996},
    NUMBER = {2},
     PAGES = {507--528},
      ISSN = {0010-3616,1432-0916},
   MRCLASS = {46L55 (46L50 82C10)},
  MRNUMBER = {1384146},
       URL = {http://projecteuclid.org/euclid.cmp/1104286339},
}

@misc{Zel05,
  title={Quantum Ergodicity and Mixing}, 
      author={S. Zelditch},
      year={2005},
      eprint={math-ph/0503026},
      archivePrefix={arXiv},
      primaryClass={math-ph},
      url={https://arxiv.org/abs/math-ph/0503026}, 
      note={arXiv:math-ph/0503026}
}

@article{Zel19,
  AUTHOR = {Zelditch, S.},
     TITLE = {Mathematics of quantum chaos in 2019},
   JOURNAL = {Notices Amer. Math. Soc.},
  FJOURNAL = {Notices of the American Mathematical Society},
    VOLUME = {66},
      YEAR = {2019},
    NUMBER = {9},
     PAGES = {1412--1422},
      ISSN = {0002-9920,1088-9477},
   MRCLASS = {81Q50},
  MRNUMBER = {3967933},
MRREVIEWER = {C\'esar\ R.\ de Oliveira},
}

@article{ZZ96,
  AUTHOR = {Zelditch, S. and Zworski, M.},
     TITLE = {Ergodicity of eigenfunctions for ergodic billiards},
   JOURNAL = {Comm. Math. Phys.},
  FJOURNAL = {Communications in Mathematical Physics},
    VOLUME = {175},
      YEAR = {1996},
    NUMBER = {3},
     PAGES = {673--682},
      ISSN = {0010-3616,1432-0916},
   MRCLASS = {58G25 (58F11 58G15)},
  MRNUMBER = {1372814},
MRREVIEWER = {Edoh\ Amiran},
       URL = {http://projecteuclid.org/euclid.cmp/1104276097},
}

\end{document}